\RequirePackage{fix-cm}
\documentclass[smallextended]{svjour3}       
\smartqed  
\usepackage{graphicx}
\usepackage{mathptmx}      
\usepackage{amssymb}
\usepackage{colortbl}
\usepackage{float}
\usepackage{caption}
\usepackage{latexsym,amsfonts,amssymb}
\usepackage{graphicx,graphics}
\usepackage[algosection]{algorithm2e}
\usepackage{amsmath}
\usepackage{epsfig}
\usepackage{subfigure}
\usepackage[misc]{ifsym}
\usepackage{multirow}
\usepackage{epstopdf}
\newtheorem{thm}{Theorem}

\usepackage[misc]{ifsym}
\usepackage[title]{appendix}
\numberwithin{equation}{section}

\begin{document}

\title{Improved uniform error bounds of exponential wave integrator method for long-time dynamics of the space fractional Klein-Gordon equation with weak nonlinearity}


\author{Junqing Jia$^{1}$   \and Xiaoyun Jiang$^{1}$
}


\institute{ \hspace*{2em}Junqing Jia \at
              \hspace*{2em}\email{18353112314@163.com}             \\
           \and
           \Letter\hspace*{1em}Xiaoyun Jiang \at
              \hspace*{2em}\email{wqjxyf@sdu.edu.cn}\\
              \and
           $^{1}$\hspace*{1.5em}School of Mathematics, Shandong University, Jinan 250100, PR China \\
}

\date{Received: date / Accepted: date}

\maketitle

\begin{abstract}
An improved uniform error bound at $O\left(h^m+\varepsilon^2 \tau^2\right)$ is established in $H^{\alpha/2}$-norm  for the long-time dynamics of the nonlinear space fractional Klein-Gordon equation (NSFKGE). A second-order exponential wave integrator (EWI) method is used to semi-discretize NSFKGE in time and the Fourier spectral method in space is applied to derive the full-discretization scheme. \textbf{Regularity compensation oscillation} (RCO) technique is employed to prove the improved uniform error bounds at $O\left(\varepsilon^2 \tau^2\right)$ in temporal semi-discretization and $O\left(h^m+\varepsilon^2 \tau^2\right)$ in full-discretization up to the long- time $T_{\varepsilon}=T / \varepsilon^2$ ($T>0$ fixed), respectively. Complex NSFKGE and oscillatory complex NSFKGE with nonlinear terms of general power exponents are also discussed. Finally, the correctness of the theoretical analysis and the effectiveness of the method are verified by numerical examples.

\end{abstract}

\section{Introduction}
\label{intro}
We consider the following dimensionless nonlinear space fractional Klein-Gordon equation (NSFKGE),
\begin{equation}\label{eq1.1}
\left\{
\begin{aligned}&\partial_{t t} \psi(\textbf{x}, t)+(-\Delta)^{\frac{\alpha}{2}} \psi(\textbf{x}, t)+\beta \psi(\textbf{x}, t)+\varepsilon^{2} \psi^{3}(\textbf{x}, t)=0,  \quad\textbf{x} \in \Omega, \quad t>0,  \\
 &\psi(\textbf{x}, 0)=\psi_{0}(\textbf{x}), \quad \partial_{t} \psi(\textbf{x}, 0)=\psi_{1}(\textbf{x}),  \quad\textbf{x} \in \Omega,\\
\end{aligned}
\right.
\end{equation}
with periodic boundary equations. $\psi:=\psi(\textbf{x}, t)$ is a real-valued function, $\textbf{x}$ is the spatial coordinate, $t$ is time, $\beta>0$. The nonlinearity strength is described by the dimensionless parameter $\varepsilon\in(0,1]$. $\Omega=\Pi_{i=1}^{d}(a_{i},b_{i})\subset \mathbb{R}^{d}(d=1,2,3)$ is a a compact domain imposed periodic boundary condition. $(-\Delta)^{\frac{\alpha}{2}}(1<\alpha\leq2)$ is the space fractional Laplace operator, which is defined by \cite{Ainsworth2017,Huang2014},
\begin{equation*}\label{eq1.2}
\begin{aligned}
 -(-\Delta)^{\frac{\alpha}{2}}\psi(\textbf{x}, t)=-\mathcal{F}^{-1}(|\textbf{$\omega$}|^{\alpha}\widehat{\psi}(\textbf{$\omega$},t)), \quad \widehat{\psi}(\textbf{$\omega$},t)=\mathcal{F}\psi(\textbf{$\omega$}, t),
\end{aligned}
\end{equation*}
where $\textbf{$\omega$}=\left(\omega_{1}, \omega_{2}, \ldots, \omega_{d}\right)$, $\mathcal{F}(\psi)$ denotes the Fourier transform of $\psi, \mathcal{F}^{-1}$ stands for its inverse transform. $\psi_{0}(\textbf{x})$ and $\psi_{1}(\textbf{x})$ are two known real-valued initial functions with no dependence on $\varepsilon$.
When $0<\varepsilon\ll1$, the Eq. (\ref{eq1.1}) with $O(1)$ initial data and $O(\varepsilon^2)$ nonlinearity can be reconverted into the following equivalent NSFKGE with $O(\varepsilon)$ initial data and $O(1)$ nonlinearity by introducing $z(\textbf{x}, t)=\varepsilon \psi(\textbf{x}, t)$,
\begin{equation}\label{eq1.4}
\left\{
\begin{aligned}&\partial_{t t} z(\textbf{x}, t)+(-\Delta)^{\frac{\alpha}{2}} z(\textbf{x}, t)+\beta z(\textbf{x}, t)+z^{3}(\textbf{x}, t)=0, \quad \textbf{x} \in \Omega, \quad t>0, \\
 &z(\textbf{x}, 0)=\varepsilon \psi_{0}(\textbf{x}), \quad \partial_{t} z(\textbf{x}, 0)=\varepsilon \psi_{1}(\textbf{x}), \quad\textbf{x} \in \Omega.
\end{aligned}
\right.
\end{equation}

In particular, when $\alpha=2, \beta=1$, Eq. \eqref{eq1.1} simplifies to the classical nonlinear Klein-Gordon equation, which was first proposed in 1927 by physicists Oskar Klein and Walter Gordon \cite{Kumar2017}. The equation describing the motion of spinless particles, widely used in particle physics, quantum electrodynamics, quantum mechanics and special relativity, etc., see \cite{Bao2014,Duncan1997} and references therein.

For the classical nonlinear Klein-Gordon equation (NKGE), it has been extensively studied numerically \cite{Bao2019,Cao1993,Deeba1996,Dehghan2009,Lindblad2020,Strauss1978}. Efficient and robust exponential-type integrators for NKGEs are introduced in \cite{Bau2018}. Cai and Zhou \cite{Cai2022} developed a class of uniformly accurate nested Picard iterative integrator (NPI) Fourier pseudospectral methods for the NKGE in the nonrelativistic regime. Calvo and Schratz \cite{Calvo2022} proposed  a  novel  class  of  uniformly  accurate  integrators  for  the  NKGE  which  capture  classical as  well  as  highly  oscillatory  nonrelativistic  regimes. Recently, widespread concern has been raised by the long-time dynamics of the NKGE with weak nonlinearity. Long-time error bounds for the numerical methods have been strictly established \cite{Bao2022,Feng2021a}. The improved uniform error bounds for the NKGE have been established with the aid of the \textbf{regularity compensation oscillation} (RCO)  technique \cite{Bao2021b}. 

Fractional operator can successfully describe various processes and materials with memory, nonlocality and genetic characteristics. It has become a powerful tool for mathematical modeling of complex material science and physical processes. Fractional partial differential equations (FPDEs) can be used to describe anomalous diffusion phenomena, long-range interaction and memory effect, see \cite{Benson2000,Jia2021,Magin2006,Metzler2000,Podlubny1999,Sun2018} and references therein. And many effective numerical methods have been developed for solving the FPDEs. For the 2D space fractional nonlinear reaction-diffusion equation, Zeng et al. \cite{Zeng2014} developed a new ADI Galerkin-Legendre spectral method. A semi-implicit time-stepping, second-order stabilized Fourier spectral approach was developed by Zhang et al \cite{Zhang2019}. For the highly oscillating space fractional nonlinear Schr\"{o}dinger equation, Feng investigated an improved error estimate \cite{Feng2021c}. For more numerical methods of solving the FPDEs, we can refer to the references in the above articles.

In \cite{Hendy2022} and references therein presented the physical motivations for considering the NSFKGE. To our knowledge, there are very few works in the literature on the improved uniform error bounds for long-time dynamics of the NSFKGE. In this paper, we established the improved uniform error bounds for long-time dynamics of the NSFKGE with weak nonlinearity by using the RCO tecnique. The main innovation of our paper is the RCO technique is first applied to prove the long-time dynamics of the NSFKGE, which improves the error convergence order in time direction to $O(\varepsilon^2\tau^2)$. In the spatial direction, the Fourier pseudospectral method is used to make the spatial accuracy reach the spectral accuracy. The one-dimensional (1D) and two-dimensional (2D) numerical experiments indicate that our numerical method is effective.

The rest of the paper is organized as follows. In Section 2, we give the fully discrete numerical scheme to the NSFKGE and rigorous convergence results are shown in Section 3. Complex NSFKGE and oscillatory complex NSFKGE with nonlinear terms of general power exponents are discussed in Section 4. Numerical experiments are provided in Section 5 to confirm our error estimates and demonstrate the potency of the numerical method. Finally, conclusions are given in Section 6. In this paper, the notation $A \lesssim B$ is used to denote $|A| \leq C B$, where $C>0$ is a general constant independent of the space step $h$ and time step $\tau$ as well as $\varepsilon$.

\section{Numerical method}
\label{sec:2}
For convenience, we only consider the 1D case, i.e., $\Omega=(a,b)$. Then, Eq. (\ref{eq1.1}) becomes
\begin{equation}\label{eq2.1}
\left\{
\begin{aligned}&\partial_{t t} \psi(x, t)+(-\Delta)^{\frac{\alpha}{2}} \psi(x, t)+\beta \psi(x, t)+\varepsilon^{2} \psi^{3}(x, t)=0, \quad  x \in \Omega, \quad t>0, \\
&\psi(a, t)=\psi(b, t), \quad \partial_{x}\psi(a,t)=\partial_{x}\psi(b,t),  \quad t\geq 0, \\
&\psi(x, 0)=\psi_{0}(x), \quad \partial_{t} \psi(x, 0)=\psi_{1}(x), \quad x \in \overline{\Omega}.
\end{aligned}
\right.
\end{equation}

\subsection{Preliminary}
\label{sec:2.1}
Let $\tau>0$ be the time step and $h=(b-a)/N$ be the space step, where $N$ is an even positive integer. The grid points are denoted as
\begin{center}
$t_{n}:=n\tau, \quad n=0,1,2,\cdots; \quad x_{j}:=a+jh, \quad j=0,1,2,\cdots,N$.
\end{center}
$X_{N}:=\{\psi=(\psi_{0},\psi_{1},\ldots, \psi_{N})^{T}\in\mathbb{C}^{N+1}|\psi_{0}=\psi_{N}\}, \quad C_{per}(\Omega)=\{\psi\in C(\overline{\Omega})|\psi(a)=\psi(b)\}$, and
\begin{equation*}
\begin{aligned}
 &Y_{N}:=\textrm{span}\left\{\textrm{e}^{i\mu_{l}(x-a)}, x\in \overline{\Omega}, \mu_{l}=\frac{2\pi l}{b-a}, l \in {T}_N\right\}, \\
 &{T}_N=\left\{l \mid l=-\frac{N}{2},-\frac{N}{2}+1, \ldots, \frac{N}{2}-1\right\}.
\end{aligned}
\end{equation*}
$P_{N}: L^{2}(\Omega)\rightarrow Y_{N}$ is the standard $L^{2}$-projection operator, $I_{N}:C_{per}(\Omega)\rightarrow Y_{N}$ or $I_{N}:X_{N}\rightarrow Y_{N}$ is the trigonometric interpolation operator, i.e.,
\begin{equation}\label{eq2.1.1}
P_{N}\psi(x)=\sum_{l \in  {T}_N}\widehat{\psi}_{l}\textrm{e}^{i\mu_{l}(x-a)}, \quad I_{N}\psi(x)=\sum_{l \in {T}_N}\tilde{\psi}_{l}\textrm{e}^{i\mu_{l}(x-a)}, \quad x\in \overline{\Omega},
\end{equation}
where
\begin{equation}\label{eq2.1.2}
\widehat{\psi}_{l}=\frac{1}{b-a}\int_{a}^{b}\psi(x)\textrm{e}^{-i\mu_{l}(x-a)}\textrm{d}x,\quad \tilde{\psi}_{l}=\frac{1}{N}\sum_{j=0}^{N-1}\psi_{j}\textrm{e}^{-i\mu_{l}(x_{j}-a)},
\end{equation}
with $\psi_{j}$ defined as $\psi(x_j)$. $P_{N}$ and $I_{N}$ are identity operators on $X_{N}$.

For $m\geq 0$,  $H^{m}(\Omega)$ is the space of functions $\psi(x)\in L^{2}(\Omega)$, the $H^{m}-$norm $\|\cdot\|_{m}$ is defined by
\begin{equation}\label{eq2.1.3}
\|\psi\|_{m}^{2}=\sum_{l\in\mathbb{Z}}(1+\mu_{l}^{2})^{m}|\widehat{\psi_{l}}|^{2}, \quad \textrm{for}\quad \psi(x)=\sum_{l\in\mathbb{Z}}\widehat{\psi}_{l}\textrm{e}^{i\mu_{l}(x-a)},
\end{equation}
where $\widehat{\psi}_{l}(l\in \mathbb{Z})$ are the Fourier coefficients of the function $\psi(x)$. Note that $H^{0}(\Omega)=L^{2}(\Omega)$, we use $\|\cdot\|$ to denote the $L^{2}-$norm $\|\cdot\|_{0}$.

Define the operator $\langle\nabla\rangle_{\alpha}=\sqrt{\beta+(-\Delta)^{\frac{\alpha}{2}}}$ ,
\begin{equation*}
\langle\nabla\rangle_{\alpha}\psi(x)=\sum_{l\in\mathbb{Z}}\sqrt{\beta+|\mu_{l}|^{\alpha}}\widehat{\psi}_{l}\textrm{e}^{i\mu_{l}(x-a)},
\end{equation*}
and the inverse operator $\langle\nabla\rangle_{\alpha}^{-1}$ as
\begin{equation*}
\langle\nabla\rangle_{\alpha}^{-1}\psi(x)=\sum_{l\in\mathbb{Z}}\frac{\widehat{\psi}_{l}}{\sqrt{\beta+|\mu_{l}|^{\alpha}}}\textrm{e}^{i\mu_{k}(x-a)},
\end{equation*}
which leads to
\begin{equation*}
\|\langle\nabla\rangle_{\alpha}^{-1}\psi\|_{s}\lesssim \|\psi\|_{s-\alpha/2}\lesssim  \|\psi\|_{s},\quad  \|\langle\nabla\rangle_{\alpha}\psi\|_{s}\lesssim \|\psi\|_{s+\alpha/2}\lesssim  \|\psi\|_{s+1}.
\end{equation*}

Introducing $\eta(x, t)=\partial_{t} \psi(x, t)$ and
$\varphi(x,t)=\psi(x,t)-i\langle\nabla\rangle_{\alpha}^{-1}\eta(x,t)$, then Eq. (\ref{eq2.1}) is equivalent to the the following relativistic nonlinear space fractional Schr\"{o}dinger equation(NLSFSE),
\begin{equation}\label{eq2.1.4}
\begin{cases}i\partial_{t}\varphi(x,t)+\langle\nabla\rangle_{\alpha}\varphi(x,t)+\frac{\varepsilon^{2}}{8}\langle\nabla\rangle_{\alpha}^{-1} (\varphi+\overline{\varphi})^{3}(x, t)=0, & x \in \Omega, \quad t>0,  \\
\varphi(a,t)=\varphi(b,t), & t\geq 0, \\
 \varphi(x, 0)=\varphi_{0}(x):=\psi_{0}(x)-i\langle\nabla\rangle_{\alpha}^{-1}\psi_{1}(x),  & x \in \overline{\Omega}.\end{cases}
\end{equation}
where $\overline{\varphi}$ is the complex conjugate of $\varphi$. The solution of Eq. (\ref{eq2.1}) is
\begin{equation}\label{eq2.1.5}
\psi(x,t)=\frac{1}{2}(\varphi(x,t)+\overline{\varphi}(x,t)), \quad \eta(x,t)=\frac{i}{2}\langle\nabla\rangle_{\alpha}(\varphi(x,t)-\overline{\varphi}(x,t)).
\end{equation}

\subsection{The time discrete scheme}
\label{sec:2.2}
The exact solution to (\ref{eq2.1.4}) can be obtained by using the variation-of-constants formula,
\begin{equation}\label{eq2.2.1}
\varphi\left(t_n+\tau\right)=\textrm{e}^{i \tau\langle\nabla\rangle_{\alpha}} \varphi\left(t_n\right)+\varepsilon^2 \int_0^\tau \textrm{e}^{i(\tau-s)\langle\nabla\rangle_{\alpha}}G\left(\varphi\left(t_n+s\right)\right) \textrm{d} s,
\end{equation}
where the nonlinear operator $G$ has the following expression
\begin{equation}\label{eq2.2.2}
G(\varphi)=i\langle\nabla\rangle_{\alpha}^{-1} g(\varphi), \quad g(\varphi)=\frac{1}{8}(\varphi+\bar{\varphi})^3,
\end{equation}
treat the integral term by the second order Deuflhard scheme (the trapezoidal quadrature)  \cite{Dong2014},
\begin{equation}\label{eq2.2.3}
\varphi\left(t_n+\tau\right)\approx \textrm{e}^{i \tau\langle\nabla\rangle_{\alpha}} \varphi\left(t_n\right)+\varepsilon^2\cdot \frac{\tau}{2}\left[G\left(\varphi\left(t_n+\tau\right)\right)+ \textrm{e}^{i\tau\langle\nabla\rangle_{\alpha}}G\left(\varphi\left(t_n\right)\right)\right].
\end{equation}
Then the time discrete scheme for Eq. (\ref{eq2.1.4}) is
\begin{equation}\label{eq2.2.4}
\varphi^{[n+1]}= \textrm{e}^{i \tau\langle\nabla\rangle_{\alpha}} \varphi^{[n]}+\varepsilon^2\cdot \frac{\tau}{2}\left[G\left(\varphi^{[n+1]}\right)+ \textrm{e}^{i\tau\langle\nabla\rangle_{\alpha}}G\left(\varphi^{[n]}\right)\right].
\end{equation}
with $\varphi^{[0]}=\varphi_0=\psi_0-i\langle\nabla\rangle_{\alpha}^{-1} \psi_1$, where $\varphi^{[n]}:=\varphi^{[n]}(x)\approx\varphi\left(x, t_n\right)$. Hence, the time discrete scheme for Eq. (\ref{eq2.1}) is
\begin{equation}\label{eq2.2.5}
\psi^{[n+1]}=\frac{1}{2}\left(\varphi^{[n+1]}+\overline{\varphi^{[n+1]}}\right), \quad \eta^{[n+1]}=\frac{i}{2}\langle\nabla\rangle_{\alpha}\left(\varphi^{[n+1]}-\overline{\varphi^{[n+1]}}\right), \quad n=0,1, \ldots,
\end{equation}
with $\psi^{[0]}=\psi_{0}(x)$ and $\eta^{[0]}=\psi_{1}(x)$.

\subsection{The fully discrete scheme}
\label{sec:2.3}
Denote $\varphi_{j}^{n}\approx\varphi(x_{j}, t_{n})$, $n=0,1, \ldots, j=0,1,\cdots, N$. Then,
the fully discrete scheme for (\ref{eq2.1.4}) is
\begin{equation}\label{eq2.3.1}
\varphi_{j}^{n+1}=\sum_{l \in  {T}_N}\textrm{e}^{i \tau\delta_{l}}\widetilde{\varphi}_{l}^{n}\textrm{e}^{i\mu_{l}(x_{j}-a)}+\frac{\varepsilon^2\tau}{2} \left[G_{j}^{n+1}+ \sum_{l \in  {T}_N}\textrm{e}^{i \tau\delta_{l}}\widetilde{G}_{l}^{n} \textrm{e}^{i\mu_{l}(x_{j}-a)} \right].
\end{equation}
with $\varphi_{j}^{0}=\psi_0(x_{j})-i\sum_{l \in {T}_N}\frac{\widetilde{(\psi_{1})}_{l}}{\delta_{l}}\textrm{e}^{i\mu_{l}(x_{j}-a)}$, where
\begin{equation*}
\delta_{l}=\sqrt{\beta+|\mu_{l}|^{\alpha}},\quad G_{j}^{n+1}=i\sum_{l \in {T}_N}\frac{1}{\delta_{l}}\widetilde{(g(\psi^{n+1}))}_{l}\textrm{e}^{i\mu_{l}(x_{j}-a)}=\sum_{l \in {T}_N}\widetilde{G}_{l}^{n+1}\textrm{e}^{i\mu_{l}(x_{j}-a)}.
\end{equation*}
Thus, the fully discrete scheme for (\ref{eq2.1}) is
\begin{equation}\label{eq2.3.2}
\psi_{j}^{n+1}=\frac{1}{2}\left(\varphi_{j}^{n+1}+\overline{\varphi_{j}^{n+1}}\right), \quad \eta_{j}^{n+1}=\frac{i}{2}\sum_{l \in  {T}_N}\delta_{l}\left[ \widetilde{\varphi}_{l}^{n+1}- \widetilde{\overline{\varphi}}_{l}^{n+1} \right]\textrm{e}^{i\mu_{l}(x_{j}-a)},
\end{equation}
for $n=0,1, \ldots$, with $\psi_{j}^{0}=\psi_{0}(x_{j})$ and $\eta_{j}^{0}=\psi_{1}(x_{j})$.

\section{Convergence analysis of the numerical method}
\label{sec:3}

First, we suppose the exact solution $\psi$ satisfy the following assumptions up to the time at $T_{\varepsilon}=T/{\varepsilon^2}(T>0)$,

(A)\quad $\|\psi\|_{C^{2}([0,T_{\varepsilon}]; H^{m+\alpha/2})}\lesssim 1, \quad \|\partial_{t}\psi\|_{C^{2}([0,T_{\varepsilon}]; H^{m})}\lesssim 1, \quad m>1.$

Next, we show the key error estimates listed below.

\begin{thm}\label{thm3.1}
Under the assumption (A), $0<\tau_{0}< 1$ is small enough and independent of $\varepsilon$ such that, when $0<\tau <\min\left\{1, \gamma \frac{\pi(b-a)^{\alpha/2} \tau_{0}^{\alpha/2}}{2 \sqrt{\beta\tau_{0}^{\alpha}(b-a)^{\alpha}+2^{\alpha} \pi^{\alpha}\left(1+\tau_{0}\right)^{\alpha}}}\right\}$, where $\gamma\in(0,1)$ is a fixed constant, the following improved uniform error bound can be obtained,
\begin{equation}\label{eq2.15}
\|\psi(\cdot, t_{n})-\psi^{[n]}\|_{\alpha/2}+\|\partial_{t}\psi(\cdot, t_{n})-\eta^{[n]}\|\lesssim \varepsilon^2\tau^2+\tau_{0}^{m+\alpha/2}, \quad 0\leq n\leq T_{\varepsilon}/{\tau}.
\end{equation}

In particular, if the exact solution is sufficiently smooth, the last term $\tau_{0}^{m+\alpha/2}$ can be ignored for small enough $\tau_{0}$, then the improved uniform error bound for sufficiently small $\tau$ is
\begin{equation}\label{eq2.16}
\|\psi(\cdot, t_{n})-\psi^{[n]}\|_{\alpha/2}+\|\partial_{t}\psi(\cdot, t_{n})-\eta^{[n]}\|\lesssim \varepsilon^2\tau^2, \quad 0\leq n\leq T_{\varepsilon}/{\tau}.
\end{equation}
\end{thm}

\begin{thm}\label{thm3.2}
Under the assumption (A), there exist $h_{0}>0$ and $0<\tau_{0}<1$ are small enough and independent of $\varepsilon$ such that, when $0<h \leq h_{0}$ and $0<\tau <$ $\min
\left\{1, \right.\\
\left.\gamma \frac{\pi(b-a)^{\alpha/2} \tau_{0}^{\alpha/2}}{2 \sqrt{\beta\tau_{0}^{\alpha}(b-a)^{\alpha}+2^{\alpha} \pi^{\alpha}\left(1+\tau_{0}\right)^{\alpha}}}\right\}$, where $\gamma \in(0,1)$ is a fixed constant, for any $0<\varepsilon \leq 1$,  we have the following improved uniform error estimate
 \begin{equation}
\label{eq3.5}
\begin{aligned}
\left\|\psi\left(\cdot, t_{n}\right)-I_{N} \psi^{n}\right\|_{\alpha/2}+\left\|\partial_{t} \psi\left(\cdot, t_{n}\right)-I_{N} \eta^{n}\right\|
\lesssim  h^{m}+\varepsilon^{2} \tau^{2}+\tau_{0}^{m+\alpha/2}, \quad 0 \leq n \leq T_{\varepsilon}/{\tau}.
\end{aligned}
\end{equation}
In particular, the last term $\tau_{0}^{m+\alpha/2}$ can be ignored for small enough $\tau_{0}$, then the improved uniform error bound for sufficiently small $\tau$ is
\begin{equation}
\label{eq3.6}
\begin{aligned}
\left\|\psi\left(\cdot, t_{n}\right)-I_{N} \psi^{n}\right\|_{\alpha/2}+\left\|\partial_{t} u\left(\cdot, t_{n}\right)-I_{N} \eta^{n}\right\| \lesssim h^{m}+\varepsilon^{2} \tau^{2}, \quad 0 \leq n \leq T_{\varepsilon}/{\tau}.
\end{aligned}
\end{equation}

\end{thm}

\subsection{Proof of Theorem \ref{thm3.1}}
\label{sec:3.1}
For small enough $0<\tau\leq \tau_{c}$ ($\tau_{c}$ is a constant), under the assumption $(A)$, there exists a constant $K > 0$ depending on $T$, $\|\psi_{0}\|_{H^{m+\alpha/2}}$, $\|\psi_{1}\|_{H^{m}}$, $\|\psi\|_{C^{2}([0,T_{\varepsilon}]; H^{m+\alpha/2})}$ and $\|\partial_{t}\psi\|_{C^{2}([0,T_{\varepsilon}]; H^{m})}$ such that
\begin{equation}\label{eq3.1.1}
\|\psi^{[n]}\|^{2}_{m+\alpha/2}+\|\eta^{[n]}\|^{2}_{m}\leq K, i.e.,  \|\varphi^{[n]}\|^{2}_{m+\alpha/2}\leq K, \quad 0\leq n\leq T_{\varepsilon}/{\tau}.
\end{equation}
See the Appendix for the specific proof of Eq. \eqref{eq3.1.1}.

Denote the numerical error function $e^{[n]}:=e^{[n]}(x)(n=0,1,\ldots, T_{\varepsilon}/{\tau})$ by
\begin{equation}\label{eq3.1.2}
e^{[n]}:=\varphi^{[n]}-\varphi(t_{n}).
\end{equation}
Subtract (\ref{eq2.2.1}) from (\ref{eq2.2.4}), we get the following error equation,
\begin{equation}\label{eq3.1.3}
\begin{aligned}
e^{[n+1]} =\textrm{e}^{i\tau\langle\nabla\rangle_{\alpha}}e^{[n]}+Q^{n}+\mathcal{E}^{n},\quad 0\leq n\leq T_{\varepsilon}/\tau-1,
\end{aligned}
\end{equation}
where
\begin{equation}\label{eq3.1.4}
\begin{aligned}
Q^{n}=\frac{\varepsilon^2\tau}{2}\left[G\left(\varphi^{[n+1]}\right)+ \textrm{e}^{i\tau\langle\nabla\rangle_{\alpha}}G\left(\varphi^{[n]}\right)\right]-\left[G\left(\varphi\left(t_{n+1}\right)\right)+ \textrm{e}^{i\tau\langle\nabla\rangle_{\alpha}}G\left(\varphi\left(t_n\right)\right)\right],
\end{aligned}
\end{equation}

\begin{equation}\label{eq3.1.5}
\begin{aligned}
\mathcal{E}^{n}=\frac{\varepsilon^2\tau}{2}\left[G\left(\varphi\left(t_{n+1}\right)\right)+ \textrm{e}^{i\tau\langle\nabla\rangle_{\alpha}}G\left(\varphi\left(t_n\right)\right)\right]-\varepsilon^2\int_0^\tau \textrm{e}^{i(\tau-s)\langle\nabla\rangle_{\alpha}} G\left(\psi\left(t_n+s\right)\right) \textrm{d} s.
\end{aligned}
\end{equation}
Then,
\begin{equation}\label{eq3.1.6}
\begin{split}
e^{[n+1]}= \textrm{e}^{i(n+1)\tau\langle\nabla\rangle_{\alpha}}e^{[0]}+\sum_{k=0}^{n}\textrm{e}^{i(n-k)
\tau\langle\nabla\rangle_{\alpha}}(Q^{k}+\mathcal{E}^{k})
=\sum_{k=0}^{n}\textrm{e}^{i(n-k)
\tau\langle\nabla\rangle_{\alpha}}(Q^{k}+\mathcal{E}^{k}).
\end{split}
\end{equation}
Under the assumption $(A)$, by \eqref{eq2.2.2} and \eqref{eq3.1.1}, we have
\begin{equation}\label{eq3.1.7}
\begin{aligned}
\|Q^{n}\|_{\alpha/2}\lesssim \varepsilon^2\tau(\|e^{[n+1]}\|_{\alpha/2}+\|e^{[n]}\|_{\alpha/2}).
\end{aligned}
\end{equation}
Thus, when $\tau$ is sufficiently small and $0<\tau<1/\varepsilon^{2}$,
\begin{equation}\label{eq3.1.8}
\begin{aligned}
\|e^{[n+1]}\|_{\alpha/2}\lesssim \varepsilon^2\tau\sum_{k=0}^{n}\|e^{[k]}\|_{\alpha/2}+\|\sum_{k=0}^{n}\textrm{e}^{i(n-k)
\tau\langle\nabla\rangle_{\alpha}}\mathcal{E}^{k}\|_{\alpha/2}.
\end{aligned}
\end{equation}
Next, the RCO technique \cite{Bao2021b} will be employed to analysis the last term on the right hand side (RHS) of (\ref{eq3.1.8}) and gain the order $O(\varepsilon^2)$.

Since $\textrm{e}^{i\tau\langle\nabla\rangle_{\alpha}}$ preserves the $H^{m}-$norm$(m>0)$, multiplying the last term in (\ref{eq3.1.8}) by $\textrm{e}^{i(n-1)\tau\langle\nabla\rangle_{\alpha}}$, then the last term becomes
\begin{equation}\label{eq3.1.9}
\|\sum_{k=0}^{n}\textrm{e}^{-i(k+1)
\tau\langle\nabla\rangle_{\alpha}}\mathcal{E}^{k}\|_{\alpha/2}.
\end{equation}
From \eqref{eq2.1.4}, we find that $\partial_{t}\varphi(x,t)-i\langle\nabla\rangle_{\alpha}\varphi(x,t)=\varepsilon^2G(\varphi(x,t))=O(\varepsilon^2)$, consider the `twisted variable' as
\begin{equation}
\label{eq3.1.10}
\xi(x,t)= \textrm{e}^{-it\langle\nabla\rangle_{\alpha}}\varphi(x,t),\quad t\geq 0,
\end{equation}
which satisfies the equation $\partial_{t}\xi(x,t)=\varepsilon^2\textrm{e}^{-it\langle\nabla\rangle_{\alpha}}G(\textrm{e}^{it\langle\nabla\rangle_{\alpha}}\xi(x,t))$.
According to the assumption (A), we have $\|\xi\|_{C^{2}([0,T_{\varepsilon}]; H^{m+\alpha/2})}\lesssim 1$ and $\|\partial_{t}\xi\|_{C^{2}([0,T_{\varepsilon}]; H^{m+\alpha/2})}\lesssim \varepsilon^2$ with
\begin{equation}
\label{eq3.1.11}
\|\xi(t_{n+1})-\xi(t_{n})\|_{m+\alpha/2}\lesssim \varepsilon^2\tau, \quad 0\leq n\leq T_{\varepsilon}/\tau-1.
\end{equation}

\textbf{Step 1.}\quad Choose $\tau_{0}\in(0,1)(\textbf{Cut-off parameter})$, and let $N_{0}=2\lceil1/\tau_{0}\rceil\in \mathbb{Z}^{+}$ ($\lceil\cdot\rceil$ is the ceiling function) with $1/\tau_{0}\leq N_{0}/2 <1+1/\tau_{0}$, then only those Fourier modes with $-\frac{N_{0}}{2}\leq l\leq \frac{N_{0}}{2}-1$ ($|l|\leq \frac{1}{\tau_{0}}$, low frequency modes) in a spectral projection would be considered. Next, we analysis $\|\mathcal{E}^{k}\|_{\alpha/2}$ carefully.

\begin{equation*}
\begin{aligned}
G(\varphi)=G(\textrm{e}^{it\langle\nabla\rangle_{\alpha}}\xi)
=\frac{i}{8}\langle\nabla\rangle_{\alpha}^{-1}(\textrm{e}^{it\langle\nabla\rangle_{\alpha}}\xi+\textrm{e}^{-it\langle\nabla\rangle_{\alpha}}\overline{\xi})^{3},
\end{aligned}
\end{equation*}
by (A), the properties of operators $\textrm{e}^{it\langle\nabla\rangle_{\alpha}}$ and $\langle\nabla\rangle_{\alpha}^{-1}$, we obtain
\begin{equation}\label{eq3.1.12}
\|G(\textrm{e}^{it_{k+1}\langle\nabla\rangle_{\alpha}}\xi(t_{k+1}))\|_{m+\alpha}\lesssim \|\xi(t_{k+1})\|_{m+\alpha/2}^{3}\lesssim 1,\quad 0\leq k\leq T_{\varepsilon}/{\tau}-1.
\end{equation}
Since
\begin{equation}\label{eq3.1.13}
\begin{aligned}
&G(\varphi(t_{k+1}))
=G(\textrm{e}^{it_{k+1}\langle\nabla\rangle_{\alpha}}\xi(t_{k+1}))\\
=&G(\textrm{e}^{it_{k+1}\langle\nabla\rangle_{\alpha}}\xi(t_{k+1}))-
P_{N_{0}}G(\textrm{e}^{it_{k+1}\langle\nabla\rangle_{\alpha}}\xi(t_{k+1}))\\
&+P_{N_{0}}G(\textrm{e}^{it_{k+1}\langle\nabla\rangle_{\alpha}}\xi(t_{k+1}))-
P_{N_{0}}G(\textrm{e}^{it_{k+1}\langle\nabla\rangle_{\alpha}}(P_{N_{0}}\xi(t_{k+1})))\\
&+P_{N_{0}}G(\textrm{e}^{it_{k+1}\langle\nabla\rangle_{\alpha}}(P_{N_{0}}\xi(t_{k+1}))),
\end{aligned}
\end{equation}
by \eqref{eq3.1.12} and the properties of projection operator, then,
\begin{equation}\label{eq3.1.14}
\begin{aligned}
\|G(\varphi(t_{k+1}))\|_{\alpha/2}\lesssim \tau_{0}^{m+\alpha/2}+\|P_{N_{0}}G(\textrm{e}^{it_{k+1}\langle\nabla\rangle_{\alpha}}(P_{N_{0}}\xi(t_{k+1})))\|_{\alpha/2}.
\end{aligned}
\end{equation}
Make similar analysis for each term in $\|\mathcal{E}^{k}\|_{\alpha/2}$, for $0\leq n \leq T_{\varepsilon}/\tau-1$,
\begin{equation}\label{eq3.1.15}
\begin{aligned}
\|e^{[n+1]}\|_{\alpha/2}\lesssim \tau_{0}^{m+\alpha/2}+ \varepsilon^2\tau\sum_{k=0}^{n}\|e^{[k]}\|_{\alpha/2}+\| {R}^{n}\|_{\alpha/2},
\end{aligned}
\end{equation}
where
\begin{equation}
\label{eq3.1.16}
\begin{aligned}
 {R}^{n}=&\sum_{k=0}^{n}\textrm{e}^{-i(k+1)
\tau\langle\nabla\rangle_{\alpha}}\left\{\frac{\varepsilon^2\tau}{2}\left[P_{N_{0}}G\left(\textrm{e}^{it_{k+1}\langle\nabla\rangle_{\alpha}}\left(P_{N_{0}}\xi(t_{k+1})\right)\right)\right.\right.\\
&+ \left. \left.\textrm{e}^{i\tau\langle\nabla\rangle_{\alpha}}P_{N_{0}}G\left(\textrm{e}^{it_{k}\langle\nabla\rangle_{\alpha}}\left(P_{N_{0}}\xi(t_{k})\right)\right)\right]\right.\\
&\left.-\varepsilon^2\int_0^\tau \textrm{e}^{i(\tau-s)\langle\nabla\rangle_{\alpha}}P_{N_{0}} G\left(\textrm{e}^{i(t_{k}+s)\langle\nabla\rangle_{\alpha}}\left(P_{N_{0}}\xi(t_{k}+s)\right)\right) \textrm{d} s \right\}.
\end{aligned}
\end{equation}

\textbf{Step 2.}\quad Next, we focus on the term $ {R}^{n}$. We present the following decomposition for the nonlinear function $G(\cdot)$,
\begin{equation}
\label{eq3.1.17}
G(\xi)=\sum_{q=1}^{4} G^{q}(\xi), \quad G^{q}(\xi)=i\langle\nabla\rangle^{-1}_{\alpha} g^{q}(\xi), \quad q=1,2,3,4,
\end{equation}
with $g^{1}(\xi)=\frac{1}{8} \xi^{3}, g^{2}(\xi)=\frac{3}{8} \bar{\xi} \xi^{2}, g^{3}(\xi)=\frac{3}{8} \bar{\xi}^{2} \xi,  g^{4}(\xi)=\frac{1}{8} \bar{\xi}^{3}$. Then $ {R}^{n}=\sum_{q=1}^{4} {R}_{q}^{n}$, where
\begin{equation}
\label{eq3.1.18}
\begin{aligned}
 {R}_{q}^{n}=&\sum_{k=0}^{n}\textrm{e}^{-i(k+1)
\tau\langle\nabla\rangle_{\alpha}}\left\{\frac{\varepsilon^2\tau}{2}\left[P_{N_{0}}G^{q}\left(\textrm{e}^{it_{k+1}\langle\nabla\rangle_{\alpha}}\left(P_{N_{0}}\xi(t_{k+1})\right)\right)\right.\right.\\
&+ \left.\left. \textrm{e}^{i\tau\langle\nabla\rangle_{\alpha}}P_{N_{0}}G^{q}\left(\textrm{e}^{it_{k}\langle\nabla\rangle_{\alpha}}\left(P_{N_{0}}\xi(t_{k})\right)\right)\right]\right.\\
&\left.-\varepsilon^2\int_0^\tau \textrm{e}^{i(\tau-s)\langle\nabla\rangle_{\alpha}}P_{N_{0}} G^{q}\left(\textrm{e}^{i(t_{k}+s)\langle\nabla\rangle_{\alpha}}\left(P_{N_{0}}\xi(t_{k}+s)\right)\right) \textrm{d} s \right\},\quad 1 \leq q \leq 4.
\end{aligned}
\end{equation}
Since the analysis of $ {R}_{q}^{n}(q=1,2,3,4)$ are similarly, we only show the case of $ {R}_{1}^{n}$ $\left(0 \leq n \leq T_{\varepsilon} / \tau-1\right)$. Define the index set $ {I}_{l}^{N_{0}}$ as
\begin{equation}
\label{eq3.1.19}
 {I}_{l}^{N_{0}}=\left\{\left(l_{1}, l_{2}, l_{3}\right) \mid l_{1}+l_{2}+l_{3}=l,  l_{1}, l_{2}, l_{3} \in {T}_{N_{0}}\right\}, \quad l \in {T}_{N_{0}}.
\end{equation}
In view of $P_{N_{0}} \xi\left(t_{k}\right)=\sum_{l \in  {T}_{N_{0}}} \widehat{\xi}_{l}\left(t_{k}\right) \textrm{e}^{i \mu_{l}(x-a)}$, the following expansion holds,
\begin{equation}
\label{eq3.1.20}
\begin{aligned}
&\textrm{e}^{-i(k+1)\tau\langle\nabla\rangle_{\alpha}}P_{N_{0}}G^{1}\left(\textrm{e}^{it_{k+1}\langle\nabla\rangle_{\alpha}}\left(P_{N_{0}}\xi(t_{k+1})\right)\right)\\
=&\textrm{e}^{-i(k+1)\tau\langle\nabla\rangle_{\alpha}}P_{N_{0}}\left[ \frac{i\langle\nabla\rangle_{\alpha}^{-1}}{8}\left( \textrm{e}^{t_{k+1}\langle\nabla\rangle_{\alpha}}P_{N_{0}}\xi(t_{k+1})\right)^{3}\right]\\
=&\sum_{l \in  {T}_{N_{0}}} \sum_{\left(l_{1}, l_{2}, l_{3}\right) \in  {I}_{l}^{N_{0}}} \frac{i}{8 \delta_{l}}  {G}_{l, l_{1}, l_{2}, l_{3}}^{k}(\tau) \textrm{e}^{i \mu_{l}(x-a)},
\end{aligned}
\end{equation}
where
\begin{equation}
\label{eq3.1.21}
 {G}_{l, l_{1}, l_{2}, l_{3}}^{k}(\zeta)=\textrm{e}^{-i(t_{k}+\zeta)\delta_{l,l_{1},l_{2}, l_{3}}}\widehat{\xi}_{l_{1}}\left(t_{k}+\zeta\right)\widehat{\xi}_{l_{2}}\left(t_{k}+\zeta\right)\widehat{\xi}_{l_{3}}\left(t_{k}+\zeta\right),\quad \zeta \in \mathbb{R},
\end{equation}
with $\delta_{l,l_{1},l_{2}, l_{3}}=\delta_{l}-\delta_{l_{1}}-\delta_{l_{2}}-\delta_{l_{3}}$ and $\delta_{l}=\sqrt{\beta+|\mu_{l}|^{\alpha}}$ for $l\in  {T}_{N_{0}}$. Thus,
\begin{equation}
\label{eq3.1.22}
 {R}_{1}^{n}=\frac{i\varepsilon^2}{8}\sum_{k=0}^{n}\sum_{l\in {T}_{N_{0}}}\sum_{\left(l_{1}, l_{2}, l_{3}\right) \in  {I}_{l}^{N_{0}}}\frac{1}{\delta_{l}}\Theta_{l,l_{1},l_{2},l_{3}}^{k}\textrm{e}^{i \mu_{l}(x-a)},
\end{equation}
where
\begin{equation}
\label{eq3.1.23}
\begin{split}
\Theta_{l,l_{1},l_{2},l_{3}}^{k}=\frac{\tau}{2}\left( {G}_{l, l_{1}, l_{2}, l_{3}}^{k}(\tau)+ {G}_{l, l_{1}, l_{2}, l_{3}}^{k}(0)\right)-\int_{0}^{\tau}G_{l, l_{1}, l_{2}, l_{3}}^{k}(s)\textrm{d}s.
\end{split}
\end{equation}
For a general function $f(s)\in C^{2}$, the following equation holds \cite{Dong2014,Germund2008},
\begin{equation}
\label{eq3.1.24}
\begin{split}
\frac{\tau}{2}(f(0)+f(\tau))-\int_{0}^{\tau}f(s)\textrm{d}s=\frac{\tau^{3}}{2}\int_{0}^{1}\mu(1-\mu)f''(\mu\tau)\textrm{d}\mu.
\end{split}
\end{equation}
Taking $f(s)=G_{l, l_{1}, l_{2}, l_{3}}^{k}(s)$ in \eqref{eq3.1.24}, expand function $f''(s)$ specifically, which yields
\begin{equation}
\label{eq3.1.25}
\begin{split}
\Theta_{l,l_{1},l_{2},l_{3}}^{k}= r_{l, l_{1}, l_{2}, l_{3}}\textrm{e}^{-it_{k}\delta_{l,l_{1},l_{2}, l_{3}}}c_{l, l_{1}, l_{2}, l_{3}}^{k},
\end{split}
\end{equation}
with
\begin{equation}
\label{eq3.1.26}
c_{l, l_{1}, l_{2}, l_{3}}^{k}=\widehat{\xi}_{l_{1}}\left(t_{k}\right) \widehat{\xi}_{l_{2}}\left(t_{k}\right) \widehat{\xi}_{l_{3}}\left(t_{k}\right),
\end{equation}

\begin{equation}
\label{eq3.1.27}
\begin{split}
r_{l, l_{1}, l_{2}, l_{3}}=O\left(\tau^{3}\left(\delta_{l, l_{1}, l_{2}, l_{3}}\right)^{2}\right).
\end{split}
\end{equation}
Assume $\delta_{l, l_{1}, l_{2}, l_{3}} \neq 0$  ($r_{l, l_{1}, l_{2}, l_{3}}=0$ if $\delta_{l, l_{1}, l_{2}, l_{3}}=0$). For $l \in {T}_{N_{0}}$ and $\left(l_{1}, l_{2}, l_{3}\right) \in {I}_{l}^{N_{0}}$, the following inequality holds,
\begin{equation}
\label{eq3.1.28}
\left|\delta_{l, l_{1}, l_{2}, l_{3}}\right| \leq 4 \delta_{N_{0} / 2}=4 \sqrt{\beta+|\mu_{N_{0} / 2}|^{\alpha}}<4 \sqrt{\beta+\frac{2^{\alpha} \pi^{\alpha}\left(1+\tau_{0}\right)^{\alpha}}{\tau_{0}^{\alpha}(b-a)^{\alpha}}}
\end{equation}
which implies
\begin{equation}
\label{eq3.1.29}
\frac{\tau}{2}\left|\delta_{l, l_{1}, l_{2}, l_{3}}\right| \leq \gamma \pi,
\end{equation}
if $0<\tau \leq \gamma \frac{\pi(b-a)^{\alpha/2} \tau_{0}^{\alpha/2}}{2 \sqrt{\beta\tau_{0}^{\alpha}(b-a)^{\alpha}+2^{\alpha} \pi^{\alpha}\left(1+\tau_{0}\right)^{\alpha}}}:=\tau_{0}^{\gamma}\left(0<\tau_{0}, \gamma<1\right)$. Denoting $S_{l, l_{1}, l_{2}, l_{3}}^{n}=$ \\
$\sum_{k=0}^{n}\textrm{e}^{-it_{k}\delta_{l, l_{1}, l_{2}, l_{3}}}(n\geq0)$, for $0<\tau\leq \tau_{0}^{\gamma}$, based on the fact $\frac{sin(x)}{x}$ is bounded and decreasing for $x\in[0, \gamma\pi)$, then the following inequality holds
\begin{equation}
\label{eq3.1.30}
\quad\left|S_{l, l_{1}, l_{2}, l_{3}}^{n}\right| \leq \frac{1}{|\sin(\tau\delta_{l, l_{1}, l_{2}, l_{3}}/2)|}\leq \frac{C}{\tau|\delta_{l, l_{1}, l_{2}, l_{3}}|}, \quad C=\frac{2\gamma\pi}{\sin(\gamma\pi)}.
\end{equation}
Using summation by parts, we derive from (\ref{eq3.1.25}) that
\begin{equation}
\label{eq3.1.31}
\sum_{k=0}^{n} \Theta_{l, l_{1}, l_{2}, l_{3}}^{k}=r_{l, l_{1}, l_{2}, l_{3}}\left[\sum_{k=0}^{n-1} S_{l, l_{1}, l_{2}, l_{3}}^{k}\left(c_{l, l_{1}, l_{2}, l_{3}}^{k}-c_{l, l_{1}, l_{2}, l_{3}}^{k+1}\right)+S_{l, l_{1}, l_{2}, l_{3}}^{n} c_{l, l_{1}, l_{2}, l_{3}}^{n}\right],
\end{equation}
with
\begin{equation}
\label{eq3.1.32}
\begin{aligned}
&c_{l, l_{1}, l_{2}, l_{3}}^{k}-c_{l, l_{1}, l_{2}, l_{3}}^{k+1} \\
=&\left(\widehat{\xi}_{l_{1}}\left(t_{k}\right)-\widehat{\xi}_{l_{1}}\left(t_{k+1}\right)\right) \widehat{\xi}_{l_{2}}\left(t_{k}\right) \widehat{\xi}_{l_{3}}\left(t_{k}\right)+\widehat{\xi}_{l_{1}}\left(t_{k+1}\right)\left(\widehat{\xi}_{l_{2}}\left(t_{k}\right)-\widehat{\xi}_{l_{2}}\left(t_{k+1}\right)\right) \widehat{\xi}_{l_{3}}\left(t_{k}\right)\\
&+\widehat{\xi}_{l_{1}}\left(t_{k+1}\right) \widehat{\xi}_{l_{2}}\left(t_{k+1}\right)\left(\widehat{\xi}_{l_{3}}\left(t_{k}\right)-\widehat{\xi}_{l_{3}}\left(t_{k+1}\right)\right) .
\end{aligned}
\end{equation}
Combining (\ref{eq3.1.27}), (\ref{eq3.1.30}), (\ref{eq3.1.31}) and (\ref{eq3.1.32}), we obtain
\begin{equation}
\label{eq3.1.33}
\begin{aligned}
\left|\sum_{k=0}^{n} \Theta_{l, l_{1}, l_{2}, l_{3}}^{k}\right| \lesssim & \tau^{2}\left|\delta_{l, l_{1}, l_{2}, l_{3}}\right| \sum_{k=0}^{n-1}\left(\left|\widehat{\xi}_{l_{1}}\left(t_{k}\right)-\widehat{\xi}_{l_{1}}\left(t_{k+1}\right)\right|\left|\widehat{\xi}_{l_{2}}\left(t_{k}\right)\right|\left|\widehat{\xi}_{l_{3}}\left(t_{k}\right)\right|\right.\\
&+\left|\widehat{\xi}_{l_{1}}\left(t_{k+1}\right)\right|\left|\widehat{\xi}_{l_{2}}\left(t_{k}\right)-\widehat{\xi}_{l_{2}}\left(t_{k+1}\right)\right|\left|\widehat{\xi}_{l_{3}}\left(t_{k}\right)\right| \\
&\left.+\left|\widehat{\xi}_{l_{1}}\left(t_{k+1}\right)\right|\left|\widehat{\xi}_{l_{2}}\left(t_{k+1}\right)\right|\left|\widehat{\xi}_{l_{3}}\left(t_{k}\right)-\widehat{\xi}_{l_{3}}\left(t_{k+1}\right)\right|\right) \\
&+\tau^{2}\left|\delta_{l, l_{1}, l_{2}, l_{3}}\right|\left|\widehat{\xi}_{l_{1}}\left(t_{n}\right)\right|\left|\widehat{\xi}_{l_{2}}\left(t_{n}\right)\right|\left|\widehat{\xi}_{l_{3}}\left(t_{n}\right)\right| .
\end{aligned}
\end{equation}
For $l \in {T}_{N_{0}}$, $\left(l_{1}, l_{2}, l_{3}\right) \in {I}_{l}^{N_{0}}$ and $\beta>0$, we know that
\begin{equation}
\label{eq3.1.34}
\begin{aligned}
\left|\delta_{l, l_{1}, l_{2}, l_{3}}\right| \lesssim\left(1+\left|\sum_{j=1}^{3} \mu_{l_{j}}\right|^{\alpha}\right)^{1 / 2}+\sum_{j=1}^{3} \sqrt{1+|\mu_{l_{j}}|^{\alpha}} \lesssim \prod_{j=1}^{3} \sqrt{1+|\mu_{l_{j}}|^{\alpha}}.
\end{aligned}
\end{equation}
Based on (\ref{eq3.1.22}),(\ref{eq3.1.33}),(\ref{eq3.1.34}), then
\begin{equation}
\label{eq3.1.35}
\begin{aligned}
&\left\|{R}_{1}^{n}\right\|_{\alpha/2}^{2}\lesssim \varepsilon^{4} \sum_{l \in  {T}_{N_{0}}}\left|\sum_{\left(l_{1}, l_{2}, l_{3}\right) \in  {I}_{l}^{N_{0}}} \sum_{k=0}^{n} \Theta_{l, l_{1}, l_{2}, l_{3}}^{k}\right|^{2}\\
 \lesssim & \varepsilon^{4} \tau^{4}\left\{\sum_{l \in  {T}_{N_{0}}}\left(\sum_{\left(l_{1}, l_{2}, l_{3}\right) \in {I}_{l}^{N_{0}}}\left|\widehat{\xi}_{l_{1}}\left(t_{n}\right)\right|\left|\widehat{\xi}_{l_{2}}\left(t_{n}\right)\right|\left|\widehat{\xi}_{l_{3}}\left(t_{n}\right)\right| \prod_{j=1}^{3} \sqrt{1+|\mu_{l_{j}}|^{\alpha}}\right)^{2}\right.\\
&+n \sum_{k=0}^{n-1} \sum_{l \in  {T}_{N_{0}}}\left[\left(\sum_{\left(l_{1}, l_{2}, l_{3}\right) \in {I}_{l}^{N_{0}}}\left|\widehat{\xi}_{l_{1}}\left(t_{k}\right)-\widehat{\xi}_{l_{1}}\left(t_{k+1}\right)\right|\left|\widehat{\xi}_{l_{2}}\left(t_{k}\right)\right|\left|\widehat{\xi}_{l_{3}}\left(t_{k}\right)\right| \prod_{j=1}^{3} \sqrt{1+|\mu_{l_{j}}|^{\alpha}}\right)^{2}\right.\\
&+\left(\sum_{\left(l_{1}, l_{2}, l_{3}\right) \in {I}_{l}^{N_{0}}}\left|\widehat{\xi}_{l_{1}}\left(t_{k+1}\right)\right|\left|\widehat{\xi}_{l_{2}}\left(t_{k}\right)-\widehat{\xi}_{l_{2}}\left(t_{k+1}\right)\right|\left|\widehat{\xi}_{l_{3}}\left(t_{k}\right)\right| \prod_{j=1}^{3} \sqrt{1+|\mu_{l_{j}}|^{\alpha}}\right)^{2}\\
&+\left.\left.\left(\sum_{\left(l_{1}, l_{2}, l_{3}\right) \in {I}_{l}^{N_{0}}}\left|\widehat{\xi}_{l_{1}}\left(t_{k+1}\right)\right|\left|\widehat{\xi}_{l_{2}}\left(t_{k+1}\right)\right|\left|\widehat{\xi}_{l_{3}}\left(t_{k}\right)-\widehat{\xi}_{l_{3}}\left(t_{k+1}\right)\right| \prod_{j=1}^{3} \sqrt{1+|\mu_{l_{j}}|^{\alpha}}\right)^{2}\right]\right\}.
\end{aligned}
\end{equation}
We get the estimate on the RHS of (\ref{eq3.1.35}) with the aid of the following auxiliary functions,
\begin{equation*}
\theta(x)=\sum_{l \in \mathbb{Z}} \sqrt{1+|\mu_{l_{j}}|^{\alpha}}\left|\widehat{\xi}_{l}\left(t_{n}\right)\right| \textrm{e}^{i \mu_{l}(x-a)},
\end{equation*}
\begin{equation*}
\theta_{1}(x)=\sum_{l \in \mathbb{Z}} \sqrt{1+|\mu_{l_{j}}|^{\alpha}}\left|\widehat{\xi}_{l}\left(t_{k}\right)-\widehat{\xi}_{l}\left(t_{k+1}\right)\right| \textrm{e}^{i \mu_{l}(x-a)},
\end{equation*}
\begin{equation*}
\theta_{2}(x)=\sum_{l \in \mathbb{Z}} \sqrt{1+|\mu_{l_{j}}|^{\alpha}}\left|\widehat{\xi}_{l}\left(t_{k}\right)\right| \textrm{e}^{i \mu_{l}(x-a)},
\end{equation*}
and
\begin{equation*}
\theta_{3}(x)=\sum_{l \in \mathbb{Z}} \sqrt{1+|\mu_{l_{j}}|^{\alpha}}\left|\widehat{\xi}_{l}\left(t_{k+1}\right)\right| \textrm{e}^{i \mu_{l}(x-a)}.
\end{equation*}
By assumption $(\mathrm{A})$ we know $\|\theta\|_{s} \lesssim\left\|\xi\left(t_{n}\right)\right\|_{s+\alpha/2}(s \leq m)$.
Thus,
\begin{equation}
\label{eq3.1.38}
\begin{aligned}
\left\|{R}_{1}^{n}\right\|_{\alpha/2}^{2} \lesssim & \varepsilon^{4} \tau^{4}\left[\left\|\xi\left(t_{n}\right)\right\|_{m+\alpha/2}^{6}+n \sum_{k=0}^{n-1}\left\|\xi\left(t_{k}\right)-\xi\left(t_{k+1}\right)\right\|_{m+\alpha/2}^{2}\right.\\
&\left.\left(\left\|\xi\left(t_{k}\right)\right\|_{m+\alpha/2}+\left\|\xi\left(t_{k+1}\right)\right\|_{m+\alpha/2}\right)^{4}\right]\\
\lesssim & \varepsilon^{4} \tau^{4}+n^{2} \varepsilon^{4} \tau^{4}\left(\varepsilon^{2} \tau\right)^{2} \lesssim \varepsilon^{4} \tau^{4}, \quad 0 \leq n \leq T_{\varepsilon} / \tau-1.
\end{aligned}
\end{equation}
The similar estimates can be established for ${R}_{q}^{n}(q=2,3,4)$, from (\ref{eq3.1.15}) and (\ref{eq3.1.17}), the following inequality is obtained,
\begin{equation}
\label{eq2.52}
\begin{aligned}
\left\|e^{[n+1]}\right\|_{\alpha/2} \lesssim \tau_{0}^{m+\alpha/2}+\varepsilon^{2} \tau^{2}+\varepsilon^{2} \tau \sum_{k=0}^{n}\left\|e^{[k]}\right\|_{\alpha/2}, \quad 0 \leq n \leq T_{\varepsilon} / \tau-1.
\end{aligned}
\end{equation}
The discrete Gronwall's implies
\begin{equation}
\label{eq2.53}
\begin{aligned}
\left\|e^{[n+1]}\right\|_{\alpha/2} \lesssim \varepsilon^{2} \tau^{2}+ \tau_{0}^{m+\alpha/2}, \quad 0 \leq n \leq T_{\varepsilon} / \tau-1,
\end{aligned}
\end{equation}
the error bound (\ref{eq2.15}) follows in view of (\ref{eq2.1.5}) and (\ref{eq2.2.5}).

\subsection{Proof of Theorem \ref{thm3.2}}
\label{sec:3.2}
Let $\psi^{n}$, $\eta^{n}$ and $\varphi^{n}$ be the numerical approximations. Under the assumption $(\mathrm{A})$, for $0<\tau \leq \tau_{c}, 0<h \leq h_{c}$, where $\tau_{c}, h_{c}$ are constants independent of $\varepsilon$, there exists a constant $K>0$ depending on $T,\left\|\psi_{0}\right\|_{m+\alpha/2},\left\|\psi_{1}\right\|_{m}$, $\|\psi\|_{C^{2}\left(\left[0, T_{\varepsilon}\right] ; H^{m+\alpha/2}\right)}$ and $\left\|\partial_{t} \psi\right\|_{C^{2}\left(\left[0, T_{\varepsilon}\right] ; H^{m}\right)}$ such that the numerical solution satisfies
\begin{equation}
\label{eq3.2.1}
\begin{aligned}
\left\|I_{N} \psi^{n}\right\|_{m+\alpha/2}^{2}+\left\|I_{N} \eta^{n}\right\|_{m}^{2} \leq K, i.e., \left\|I_{N} \varphi^{n}\right\|_{m+\alpha/2}^{2} \leq K, \quad 0 \leq n \leq T_{\varepsilon}/{\tau}.
\end{aligned}
\end{equation}

Since $\varphi\left(\cdot, t_{n}\right)-I_{N} \varphi^{n}=\varphi\left(\cdot, t_{n}\right)-\varphi^{[n]}+\varphi^{[n]}-P_{N} \varphi^{[n]}+P_{N} \varphi^{[n]}-I_{N} \varphi^{n}$, we derive that
\begin{equation}
\label{eq3.2.2}
\begin{aligned}
\left\|\varphi\left(\cdot, t_{n}\right)-I_{N} \varphi^{n}\right\|_{\alpha/2} \leq\left\|P_{N} \varphi^{[n]}-I_{N} \varphi^{n}\right\|_{\alpha/2}+C_{1}\left(\varepsilon^{2} \tau^{2}+\tau_{0}^{m+\alpha/2}+h^{m}\right)
\end{aligned}
\end{equation}

Define the error function $e^{n}:=e^{n}(x) \in$ $Y_{N}$,
$$
e^{n}:=P_{N} \varphi^{[n]}-I_{N} \varphi^{n}, \quad 0 \leq n \leq T_{\varepsilon}/{\tau}.
$$
From \eqref{eq2.2.4} and \eqref{eq2.3.1}, we get the following equations,
$$
\begin{aligned}
&P_{N} \varphi^{[n+1]}=\textrm{e}^{i \tau\langle\nabla\rangle_{\alpha}} P_{N} \varphi^{[n]}+ \varepsilon^{2} \tau/2 \left(P_{N}G\left( \varphi^{[n+1]}\right)+e^{i \tau\langle\nabla\rangle_{\alpha}}P_{N}G\left( \varphi^{[n]}\right)\right),\\
&I_{N} \varphi^{n+1}=\textrm{e}^{i \tau\langle\nabla\rangle_{\alpha}} I_{N} \varphi^{n}+\varepsilon^{2} \tau/2 \left(I_{N}G\left( I_{N} \varphi^{n+1}\right)+e^{i \tau\langle\nabla\rangle_{\alpha}}I_{N}G\left( I_{N} \varphi^{n}\right)\right) ,
\end{aligned}
$$
which lead to
\begin{equation}
\label{eq3.2.3}
\begin{aligned}
e^{n+1}=&\textrm{e}^{i \tau\langle\nabla\rangle_{\alpha}} e^{n}+\varepsilon^{2} \tau/2 \left(P_{N}G\left( \varphi^{[n+1]}\right)-I_{N}G\left( I_{N} \varphi^{n+1}\right)\right.\\
&\left.+\textrm{e}^{i \tau\langle\nabla\rangle_{\alpha}}\left(P_{N}G\left( \varphi^{[n]}\right)-I_{N}G\left( I_{N} \varphi^{n}\right)\right)\right),
\end{aligned}
\end{equation}
when $\tau$ is sufficiently small and $0<\tau<1/{\varepsilon^2}$, then
\begin{equation}
\label{eq3.2.4}
\begin{aligned}
\|e^{n+1}\|_{\alpha/2}\lesssim \|e^{n}\|_{\alpha/2}+\varepsilon^{2} \tau h^{m+\alpha/2},\quad 0 \leq n \leq T_{\varepsilon} / \tau.
\end{aligned}
\end{equation}
Since $e^{0}=P_{N} u_{0}-I_{N} u_{0}-$ $i\langle\nabla\rangle_{\alpha}^{-1}\left(P_{N} u_{1}-I_{N} u_{1}\right)$, then $\left\|e^{0}\right\|_{\alpha/2} \lesssim h^{m}$, the discrete Gronwall's inequality implies $\left\|e^{n+1}\right\|_{\alpha/2} \lesssim h^{m}\left(0 \leq n \leq T_{\varepsilon} / \tau-1\right)$. Combining the above estimtates with (\ref{eq3.2.2}), we derive
$$
\left\|\varphi\left(\cdot, t_{n}\right)-I_{N} \varphi^{n}\right\|_{\alpha/2} \lesssim h^{m}+\varepsilon^{2} \tau^{2}+\tau_{0}^{m+\alpha/2}, \quad 0 \leq n \leq T_{\varepsilon} / \tau.
$$
Recalling (\ref{eq2.3.2}), the error bounds for $\psi^{n}$ and $\eta^{n}\left(0 \leq n \leq T_{\varepsilon} / \tau\right)$ are
$$
\begin{aligned}
\left\|\psi\left(\cdot, t_{n}\right)-I_{N} \psi^{n}\right\|_{\alpha/2} &=\frac{1}{2}\left\|\varphi\left(\cdot, t_{n}\right)+\overline{\varphi\left(\cdot, t_{n}\right)}-I_{N} \varphi^{n}-I_{N} \overline{\varphi^{n}}\right\|_{\alpha/2} \\
& \leq\left\|\varphi\left(\cdot, t_{n}\right)-I_{N} \varphi^{n}\right\|_{\alpha/2} \lesssim \varepsilon^{2} \tau^{2}+\tau_{0}^{m+\alpha/2}+h^{m} \\
\left\|\eta\left(\cdot, t_{n}\right)-I_{N} \eta^{n}\right\| &=\frac{1}{2}\left\|\langle\nabla\rangle_{\alpha}\left(\varphi\left(\cdot, t_{n}\right)-\overline{\varphi\left(\cdot, t_{n}\right)}\right)-\langle\nabla\rangle_{\alpha}\left(I_{N} \varphi^{n}-I_{N} \overline{\varphi^{n}}\right)\right\| \\
&\lesssim\left\|\varphi\left(\cdot, t_{n}\right)-I_{N} \varphi^{n}\right\|_{\alpha/2}\lesssim \varepsilon^{2} \tau^{2}+\tau_{0}^{m+\alpha/2}+h^{m},
\end{aligned}
$$
which show (\ref{eq3.5}) is valid and we complete the proof of Theorem 2.

\section{Extensions}
\label{sec:4}
In this section, we discuss the error estimates of the complex NSFKGE and oscillatory complex NSFKGE with nonlinear terms of general power exponents.
\subsection{The complex NSFKGE with nonlinear terms of general power exponents}
\label{sec:4.1}
Consider the following dimensionless nonlinear complex NSFKGE with nonlinear terms of general power exponents,
\begin{equation}\label{eq4.1.1}
\left\{
\begin{aligned}&\partial_{t t} \psi(\textbf{x}, t)+(-\Delta)^{\frac{\alpha}{2}} \psi(\textbf{x}, t)+\beta \psi(\textbf{x}, t)+\varepsilon^{2p} |\psi(\textbf{x}, t)|^{2p}\psi(\textbf{x}, t)=0,  \quad\textbf{x} \in \Omega, \quad t>0,  \\
 &\psi(\textbf{x}, 0)=u_{0}(\textbf{x}), \quad \partial_{t} \psi(\textbf{x}, 0)=\psi_{1}(\textbf{x}),  \quad\textbf{x} \in \Omega,\\
\end{aligned}
\right.
\end{equation}
with periodic boundary equations. $\psi:=\psi(\textbf{x}, t)$ is a complex-valued function, $p\in \mathbb{N}^{+}$. $\psi_{0}(\textbf{x})$ and $\psi_{1}(\textbf{x})$ are two known complex-valued functions independent of $\varepsilon$. When $\alpha=2, \beta=1$, the analysis results show that the life-span of NKGE smooth solution is at least $O(\varepsilon^{-2p})$, see \cite{Delort2009,Delort2004,Fang2010} and references therein.

Here, we only show the theoretical result in 1D. Introducing $\eta(x,t)=\partial_t\psi(x,t)$ and
\begin{equation}\label{eq4.1.22}
\varphi_{\pm}(x, t)=\psi(x, t) \mp i\langle\nabla\rangle^{-1} \eta(x, t), \quad a \leq x \leq b, \quad t \geq 0,
\end{equation}
and introducing $f(\varphi)=|\varphi|^{2 p} \varphi$, then Eq. \eqref{eq4.1.1} can be transformed into the following coupled relativistic NLSFSEs:
\begin{equation}\label{eq4.1.33}
\left\{\begin{array}{l}
i \partial_t \varphi_{\pm} \pm\langle\nabla\rangle_{\alpha} \varphi_{\pm} \pm \varepsilon^{2 p}\langle\nabla\rangle^{-1}_{\alpha} f\left(\frac{1}{2} \varphi_{+}+\frac{1}{2} \varphi_{-}\right)=0, \\
\varphi_{\pm}(t=0)=u_0 \mp i\langle\nabla\rangle^{-1} v_0 .
\end{array}\right.
\end{equation}
Suppose that there exists an exact solution $\psi:=\psi(x,t)$ of the NSFKGE \eqref{eq4.1.1} up to the time $T_{\varepsilon,p}=T/\varepsilon^{2p}$ and

(B)\quad $\|\psi\|_{C^{2}([0,T_{\varepsilon,p}]; H^{m+\alpha/2})}\lesssim 1, \quad \|\partial_{t}\psi\|_{C^{2}([0,T_{\varepsilon,p}]; H^{m})}\lesssim 1, \quad m>1.$

We then give the following improved uniform error bounds.
\setcounter{thm}{0}
\begin{thm}\label{thm4.1.1}
Under the assumption (B), there exist $h_{0}>0$ and $0<\tau_{0}<1$ are small enough and independent of $\varepsilon$, such that when $0<h \leq h_{0}$, $0<\tau<\gamma \tau_{0}$ where $\gamma >0 $ is a fixed constant, for any $0<\varepsilon \leq 1$,  we have the following improved uniform error estimate
 \begin{equation}
\label{eqeq4.1.2}
\begin{aligned}
&\left\|\psi\left(\cdot, t_{n}\right)-I_{N} \psi^{n}\right\|_{\alpha/2}+\left\|\partial_{t} \psi\left(\cdot, t_{n}\right)-I_{N} \eta^{n}\right\|\\
&\lesssim h^{m}+\varepsilon^{2p} \tau^{2}+\tau_{0}^{m+\alpha/2}, \quad 0 \leq n \leq T_{\varepsilon,p}/{\tau}.
\end{aligned}
\end{equation}
In particular, if the exact solution is sufficiently smooth, the improved uniform error bounds for sufficiently small $\tau$ is
\begin{equation}
\label{eq4.1.3}
\begin{aligned}
&\left\|\psi\left(\cdot, t_{n}\right)-I_{N} \psi^{n}\right\|_{\alpha/2}+\left\|\partial_{t} \psi\left(\cdot, t_{n}\right)-I_{N} \eta^{n}\right\|\\
& \lesssim h^{m}+\varepsilon^{2p} \tau^{2}, \quad 0 \leq n \leq T_{\varepsilon,p}/{\tau}.
\end{aligned}
\end{equation}
\end{thm}

\begin{remark}\label{rem4.1}
The NSFKGE \eqref{eq4.1.1} conserves the energy as
\begin{equation}\label{eq4.1.4}
\begin{aligned}
E(t) &:=\int_{\Omega}\left[|\partial_{t} \psi(x, t)|^{2}+|(-\Delta)^{\frac{\alpha}{4}} \psi(x, t)|^{2}+\beta |\psi(x, t)|^{2}+\frac{\varepsilon^{2p}}{p+1}|\psi(x, t)|^{2p+2}\right] \textrm{d}x \\
& \equiv \int_{\Omega}\left[|\psi_{1}(x)|^{2}+|(-\Delta)^{\frac{\alpha}{4}} \psi_{0}(x)|^{2}+\beta|\psi_{0}(x)|^{2}+\frac{\varepsilon^{2p}}{p+1}|\psi_{0}(x)|^{2p+2}\right] \textrm{d}x\\
&=E(0), \quad t \geq 0 .
\end{aligned}
\end{equation}
\end{remark}

\subsection{An oscillatory complex NSFKGE}
\label{sec:4.2}
Rescale in time
\begin{equation}\label{eq4.2.1}
\begin{aligned}
t=\frac{r}{\varepsilon^{2p}}\Longleftrightarrow r=\varepsilon^{2p} t, \Phi(\textbf{x}, r)=\psi(\textbf{x}, t),
\end{aligned}
\end{equation}
then the Eq. \eqref{eq4.1.1} can be reformulated as the following oscillatory complex NSFKGE
\begin{equation}\label{eq4.2.2}
\left\{
\begin{aligned}&\varepsilon^{2p}\partial_{r r} \Phi(\textbf{x}, r)+\frac{1}{\varepsilon^{2p}}(-\Delta)^{\frac{\alpha}{2}} \Phi(\textbf{x}, r)+\frac{\beta}{\varepsilon^{2p}} \Phi(\textbf{x}, r)+ |\Phi(\textbf{x}, r)|^{2p}\Phi(\textbf{x}, r)=0,    \\
 &\Phi(\textbf{x}, 0)=\psi_{0}(\textbf{x}), \quad \partial_{r} \Phi(\textbf{x}, 0)=\frac{1}{\varepsilon^{2p}}\psi_{1}(\textbf{x}),  \quad\textbf{x} \in \Omega,\\
\end{aligned}
\right.
\end{equation}
where $\textbf{x} \in \Omega, r>0$. Denote $\Upsilon(\textbf{x}, r)=\partial_{r} \Phi(\textbf{x}, r)$, taking the time step $\lambda=\varepsilon^{2p}\tau$. Extend the improved error bounds of the long-time problem to the oscillatory complex NSFKGE \eqref{eq4.2.2} up to the fixed time $T$. For convenience, we only consider 1D problem, and assume existence of the exact solution $\Phi:=\Phi(x,t)$ of the NSFKGE \eqref{eq4.2.2}, and

(C)\quad $\|\Phi\|_{C^{2}([0,T]; H^{m+\alpha/2})}\lesssim 1, \quad \|\partial_{r}\Phi\|_{C^{2}([0,T]; H^{m})}\lesssim \frac{1}{\varepsilon^{2p}}, \quad m>1.$

\begin{thm}\label{thm4.2.1}
Under the assumption (C), there exist $h_{0}>0$ and $0<\lambda_{0}<1$ are small enough and independent of $\varepsilon$ such that, when $0<h \leq h_{0}$, $0<\lambda<\gamma\lambda_{0}\varepsilon^{2p}$ where $\gamma >0 $ is a fixed constant, for any $0<\varepsilon \leq 1$, we have the following improved uniform error estimate
 \begin{equation}
\label{eq4.2.3}
\begin{aligned}
&\left\|\Phi\left(\cdot, r_{n}\right)-I_{N} \Phi^{n}\right\|_{\alpha/2}+\varepsilon^{2p}\left\|\partial_{r} \Phi\left(\cdot, t_{n}\right)-I_{N} \Upsilon^{n}\right\|\\
&\lesssim h^{m}+\frac{\lambda^{2}}{\varepsilon^{2p}}+\lambda_{0}^{m+\alpha/2}, \quad 0 \leq n \leq T/{\lambda}.
\end{aligned}
\end{equation}
In particular, if the exact solution is sufficiently smooth, the improved uniform error bounds for sufficiently small time step $\lambda$ is
\begin{equation}
\label{eq4.2.4}
\begin{aligned}
&\left\|\Phi\left(\cdot, r_{n}\right)-I_{N} \Phi^{n}\right\|_{\alpha/2}+\varepsilon^{2p}\left\|\partial_{r} \Phi\left(\cdot, t_{n}\right)-I_{N} \mu^{n}\right\|\\
&\lesssim h^{m}+\frac{\lambda^{2}}{\varepsilon^{2p}}, \quad 0 \leq n \leq T/{\lambda}.
\end{aligned}
\end{equation}
\end{thm}

\section{Numerical results}
\label{sec:5}
We demonstrate a few numerical examples in 1D and 2D in this part to illustrate the improved uniform error bounds for the long-time dynamics of the NSFKGE and oscillating NSFKGE.
\subsection{The long-time dynamics in 1D}
\label{sec:5.1}
First, we study the long-time errors for the \eqref{eq4.1.1} in 1D with $p=2$ and real-valued initial data as
\begin{equation}
\label{eq5.1.1}
\begin{aligned}
\psi_0(x)=\frac{3}{2+\cos ^2(x)}, \quad \psi_1(x)=\frac{3}{4+\cos ^2(x)}, \quad x \in \Omega=(0,2 \pi).
\end{aligned}
\end{equation}
For no exact solution to \eqref{eq4.1.1}, we use the numerical solution calculated with very fine space step ($h_{e}=\pi/64(i.e. N=128)$) and time step ($\tau_{e}=10^{-4}$) as the `exact' solution. We define the following error functions to quantify the error:
\begin{equation}\label{eq5.1.2}
e_1\left(t_n\right)=\left\|\psi\left(x, t_n\right)-I_N \psi^n\right\|_{\alpha/2}, \quad e_{1, \max }\left(t_n\right)=\max _{0 \leq m \leq n} e_1\left(t_m\right).
\end{equation}

\begin{figure}[htb]
\centering
\includegraphics[scale=.34]{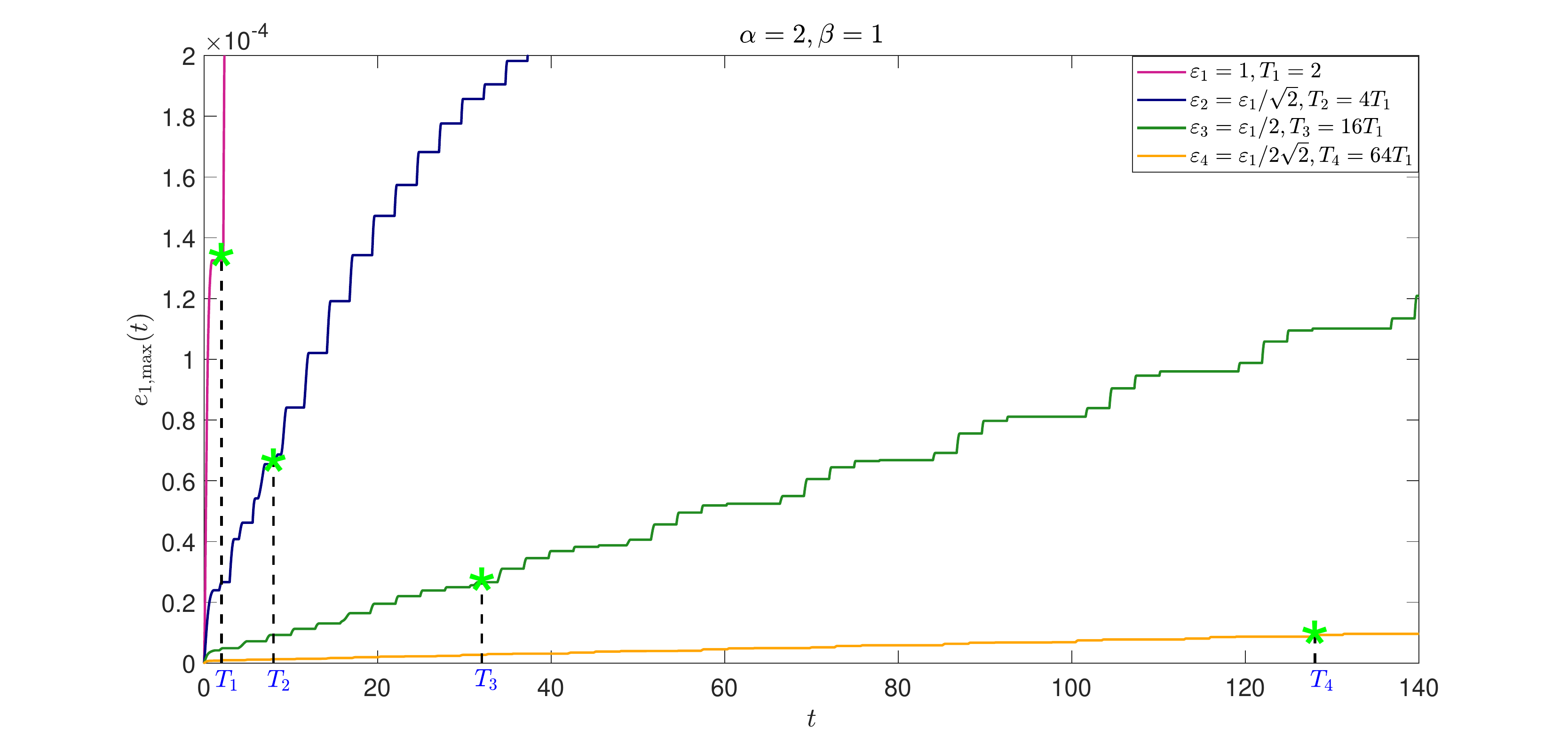}
\includegraphics[scale=.34]{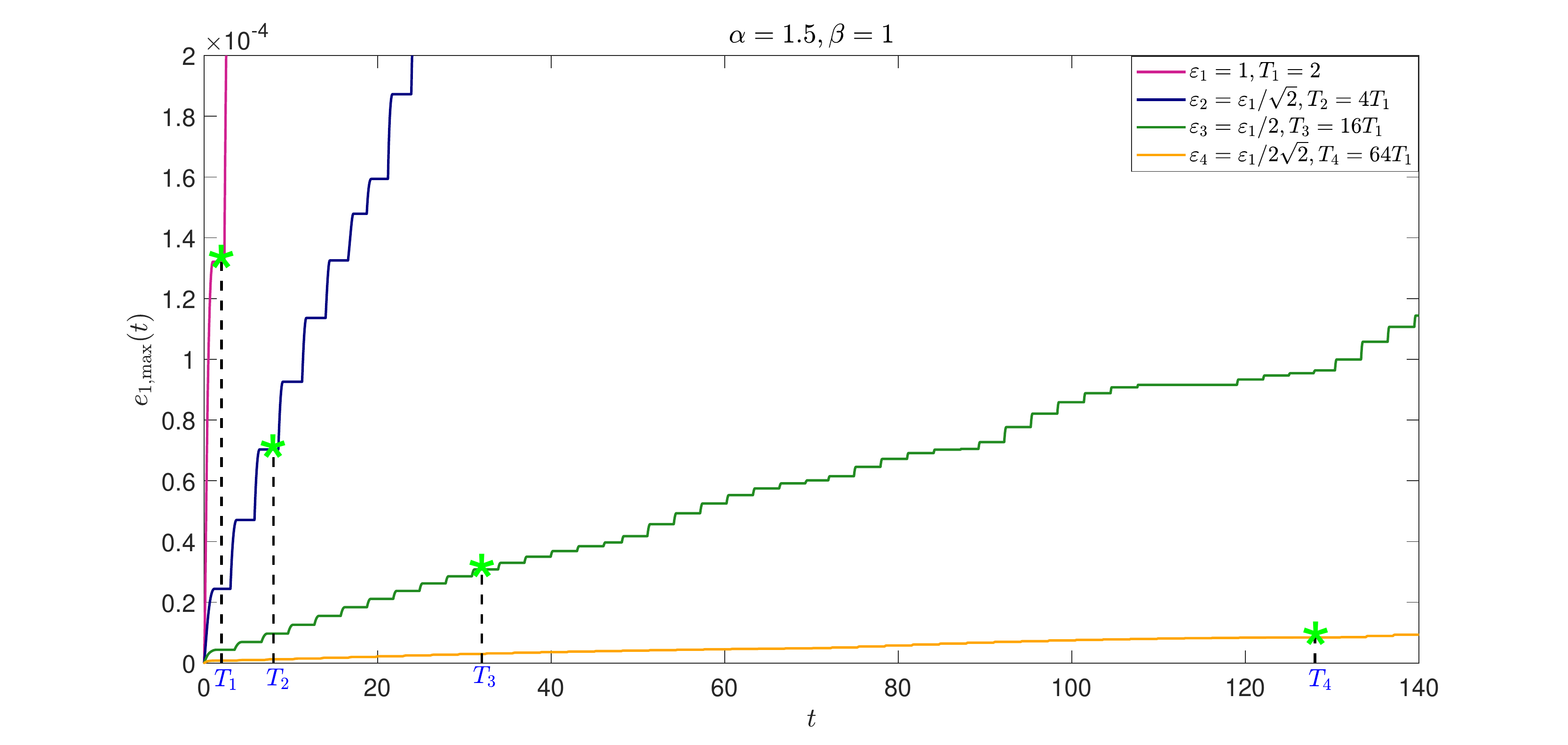}
\includegraphics[scale=.34]{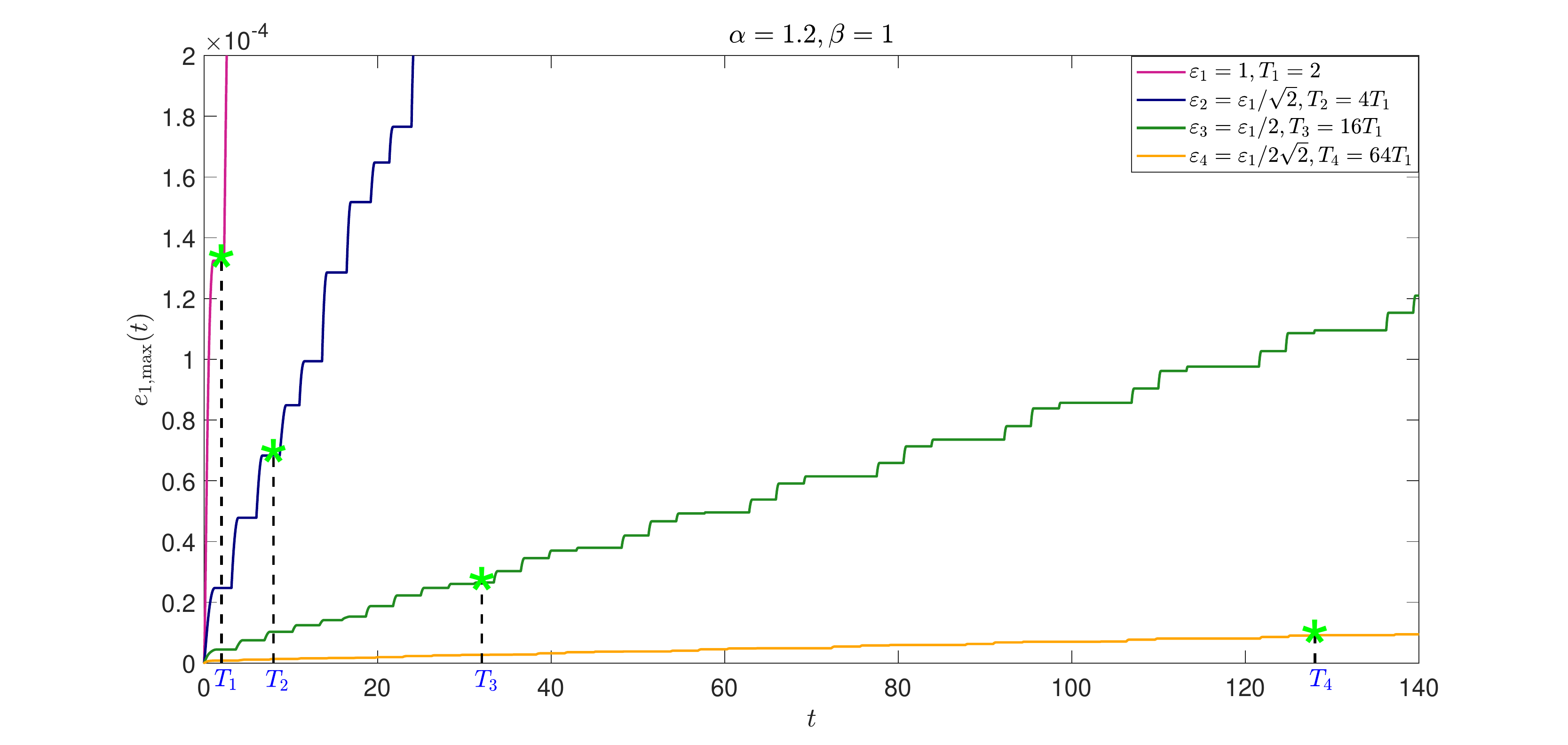}
\caption{Long-time temporal errors for the NSFKGE \eqref{eq4.1.1} with $p=2$ and different $\varepsilon$ in 1D when $\alpha$ is taken as 2, 1.5 and 1.2 respectively.}
\label{fig5.1}
\end{figure}

\clearpage
\begin{figure}[htb]
\centering
\includegraphics[scale=.34]{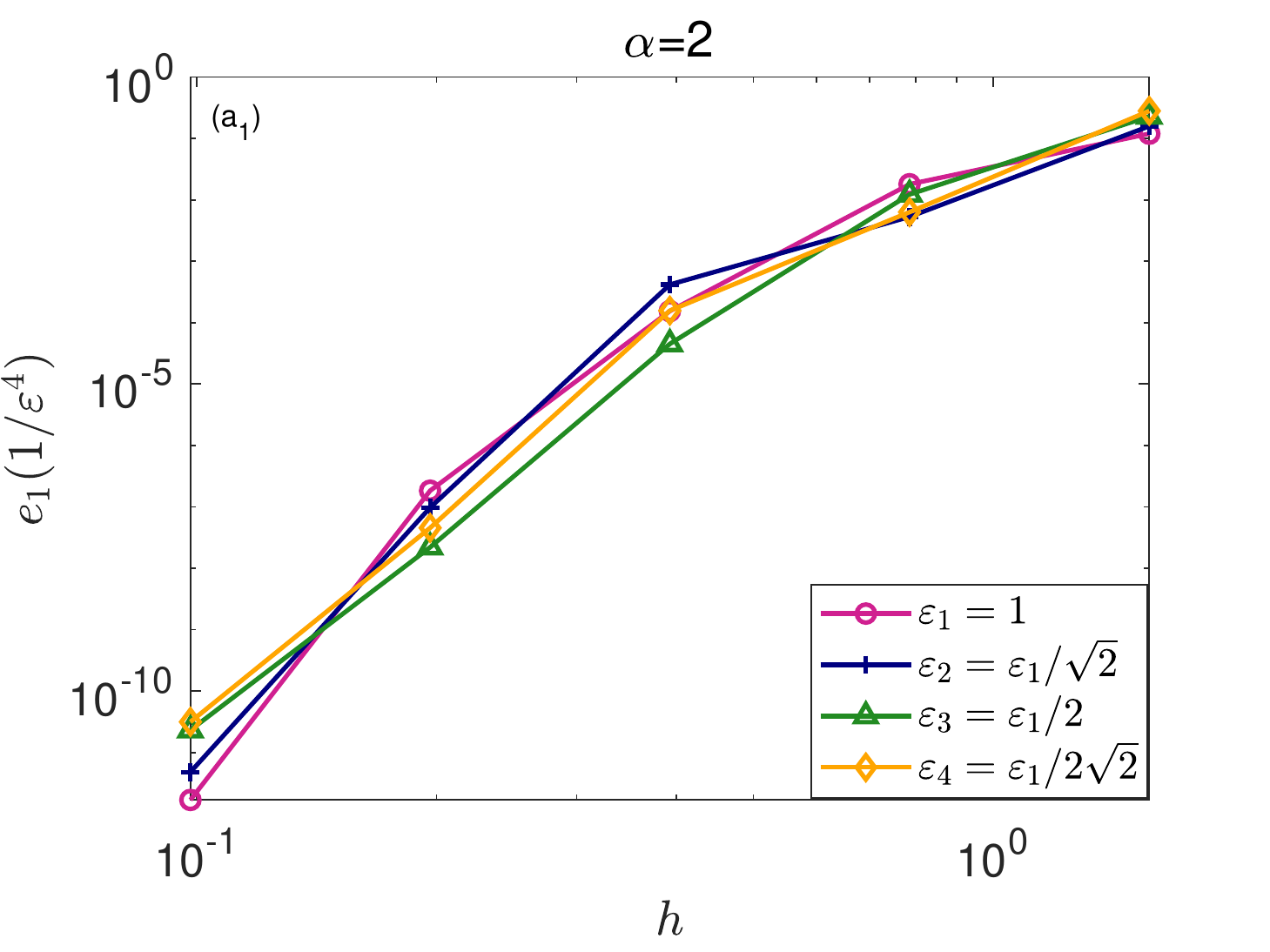}
\includegraphics[scale=.34]{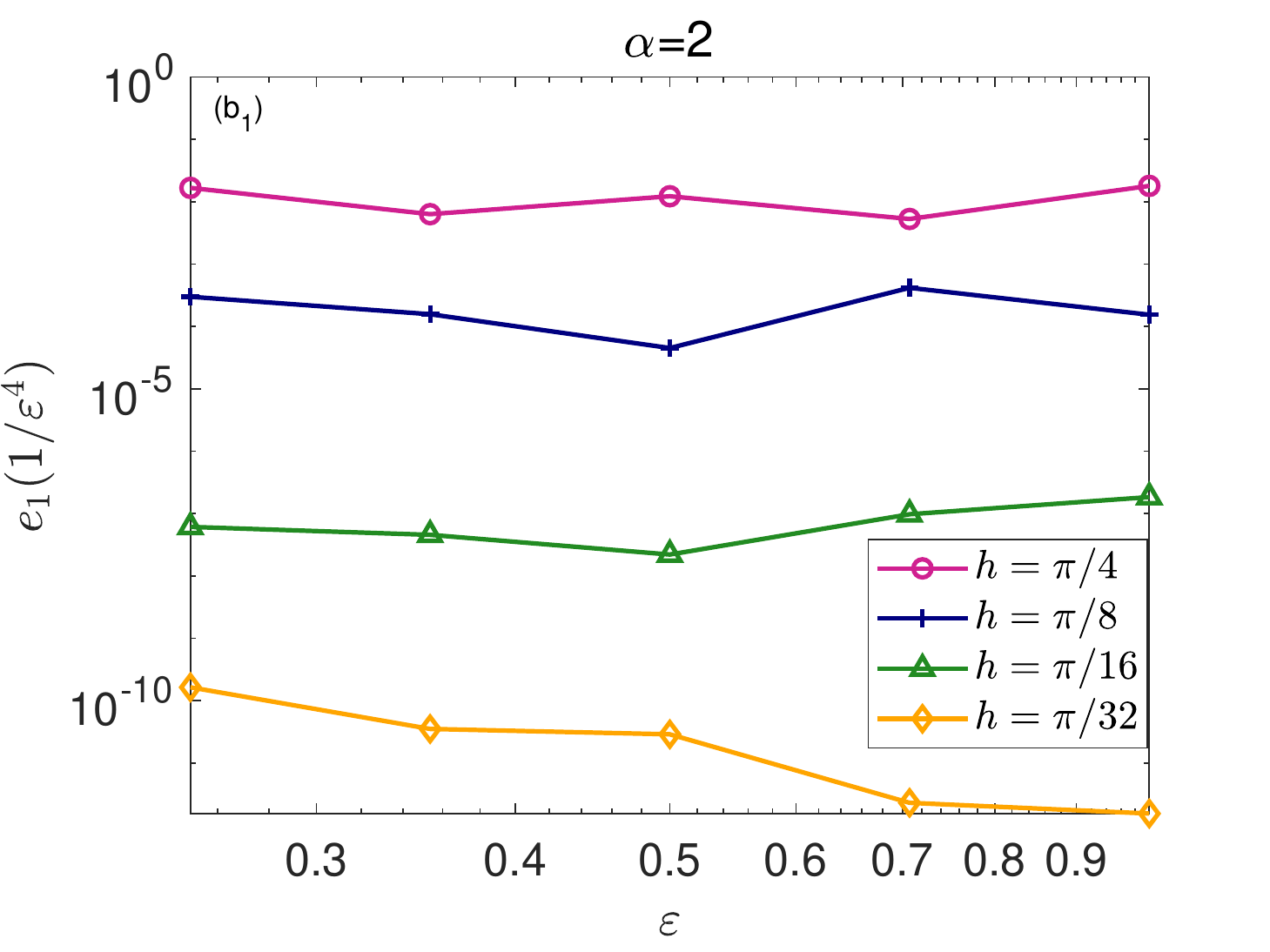}
\includegraphics[scale=.34]{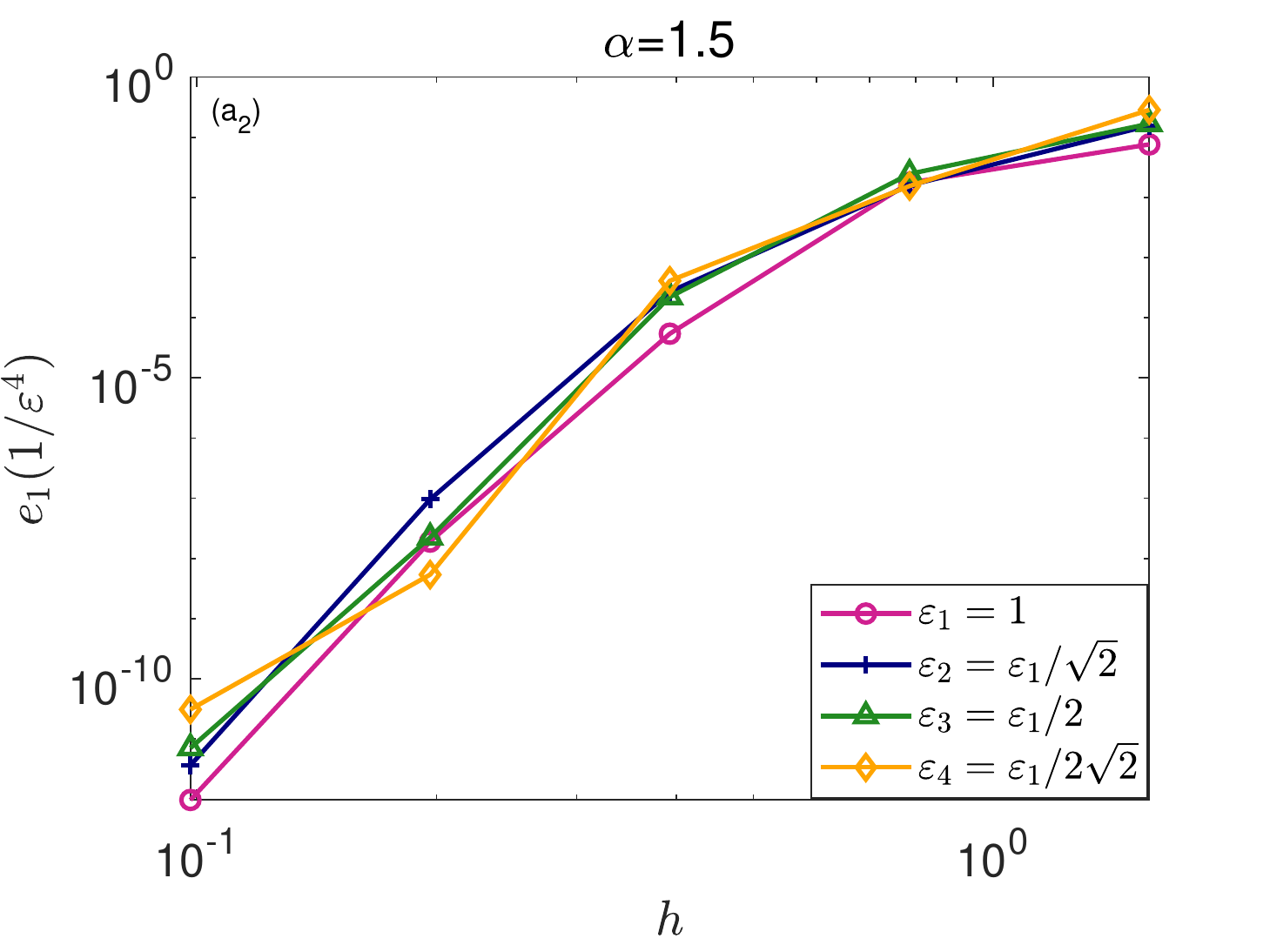}
\includegraphics[scale=.34]{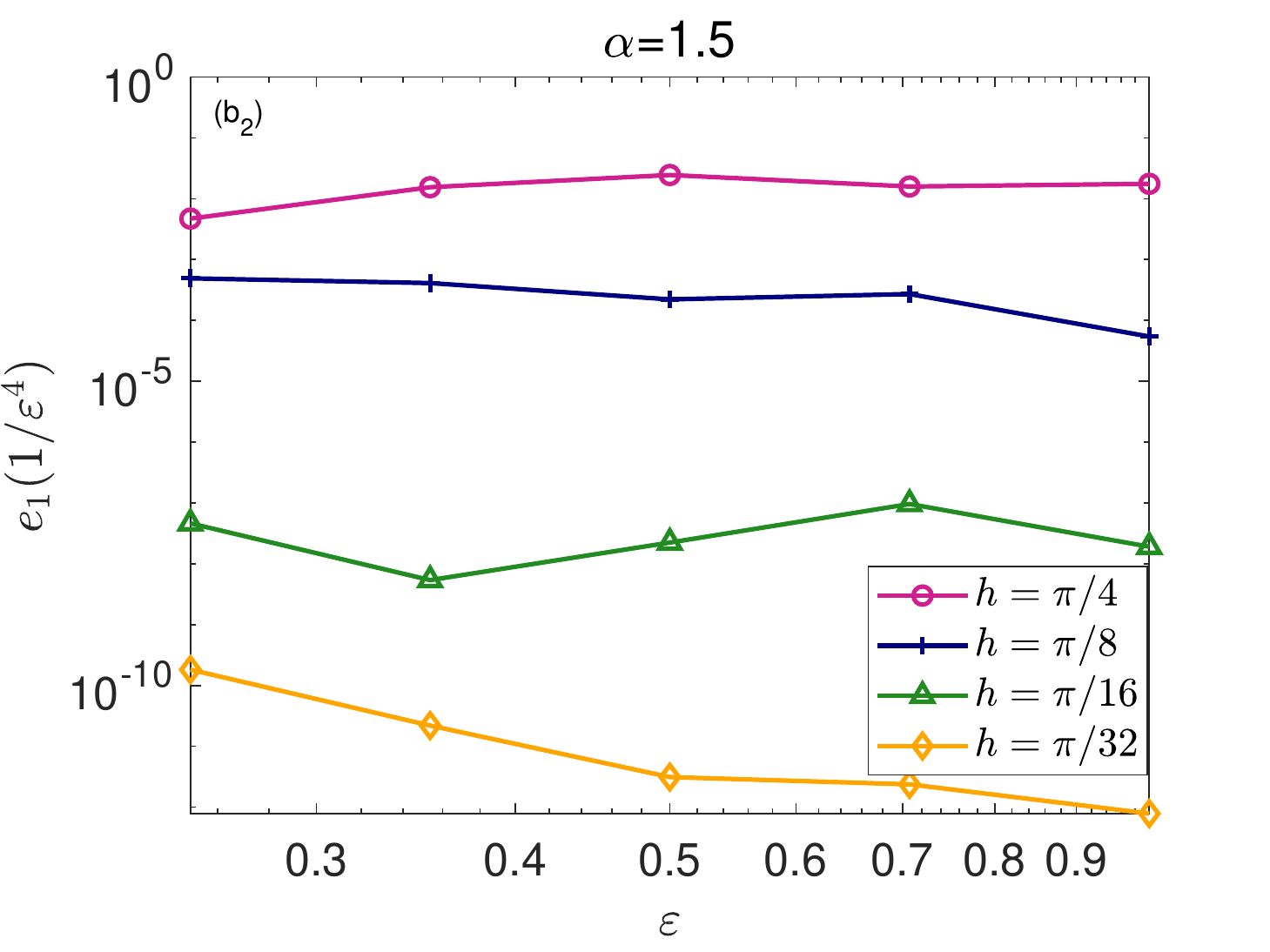}
\includegraphics[scale=.34]{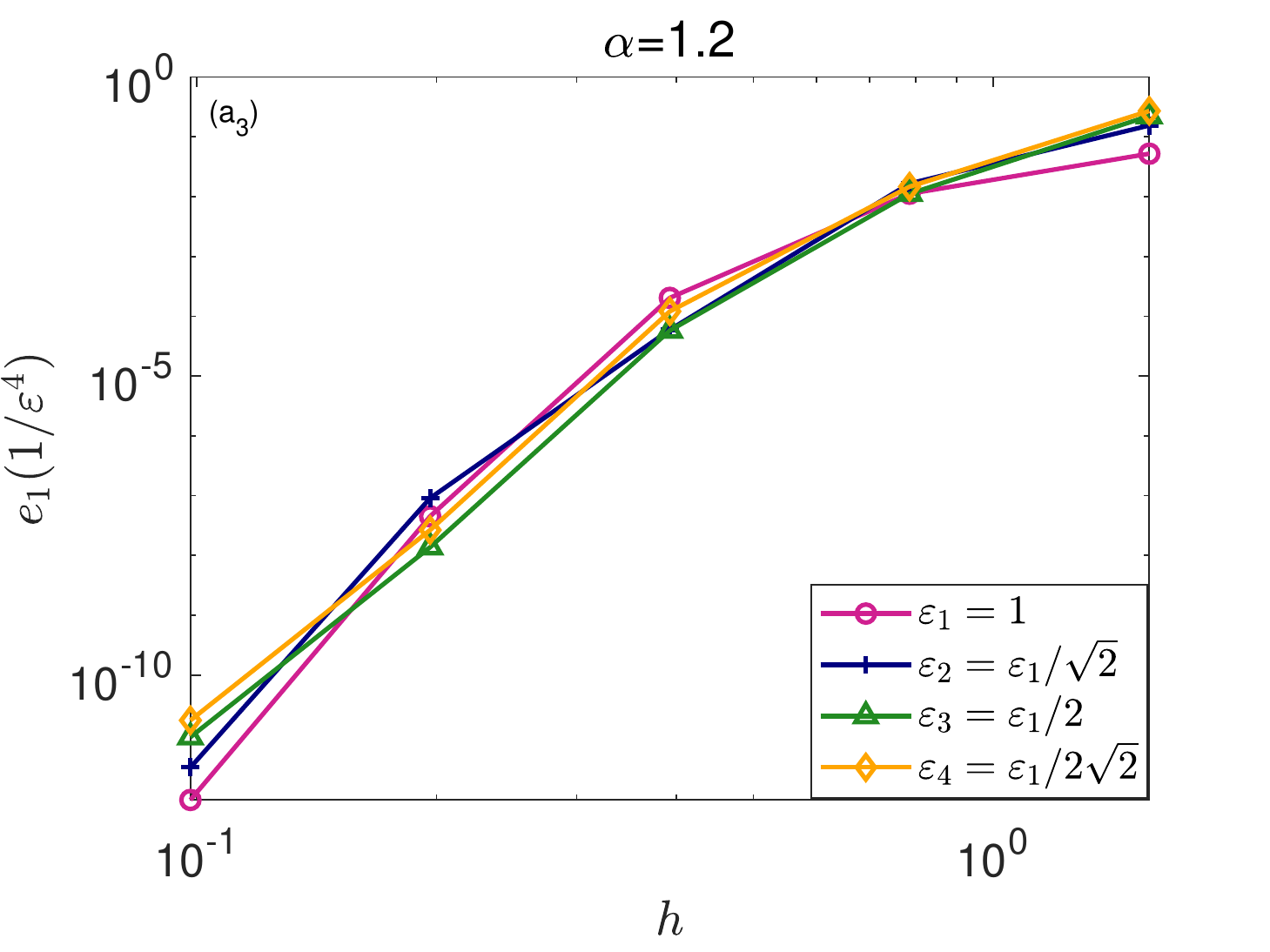}
\includegraphics[scale=.34]{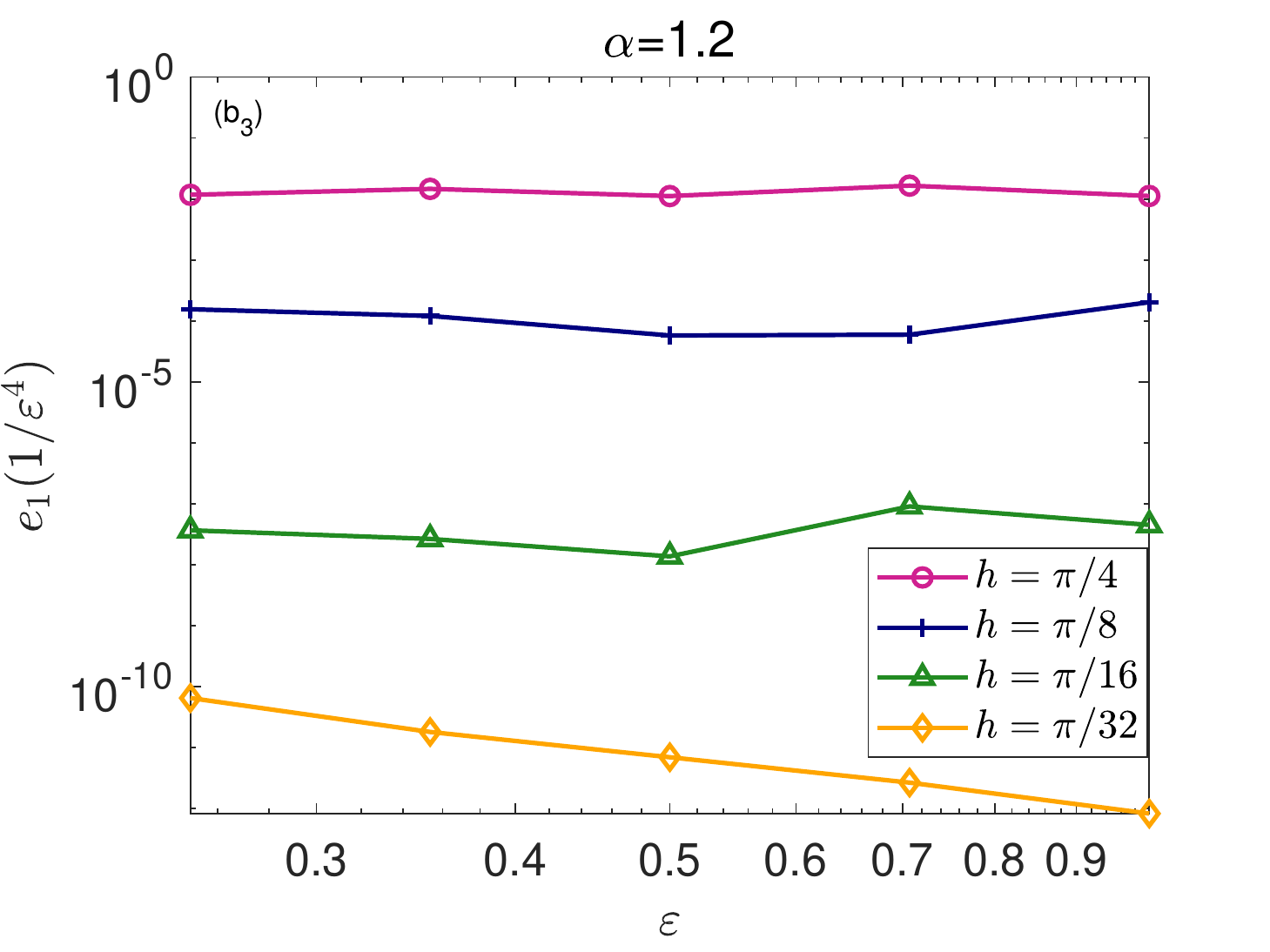}
\caption{Long-time spatial errors for the NSFKGE \eqref{eq4.1.1} with $p=2$ in 1D when $\alpha$ is taken as 2, 1.5 and 1.2 at $t=1/{\varepsilon^{4}}$ respectively.}
\label{fig5.2}
\end{figure}

\begin{figure}[htb]
\centering
\includegraphics[scale=.34]{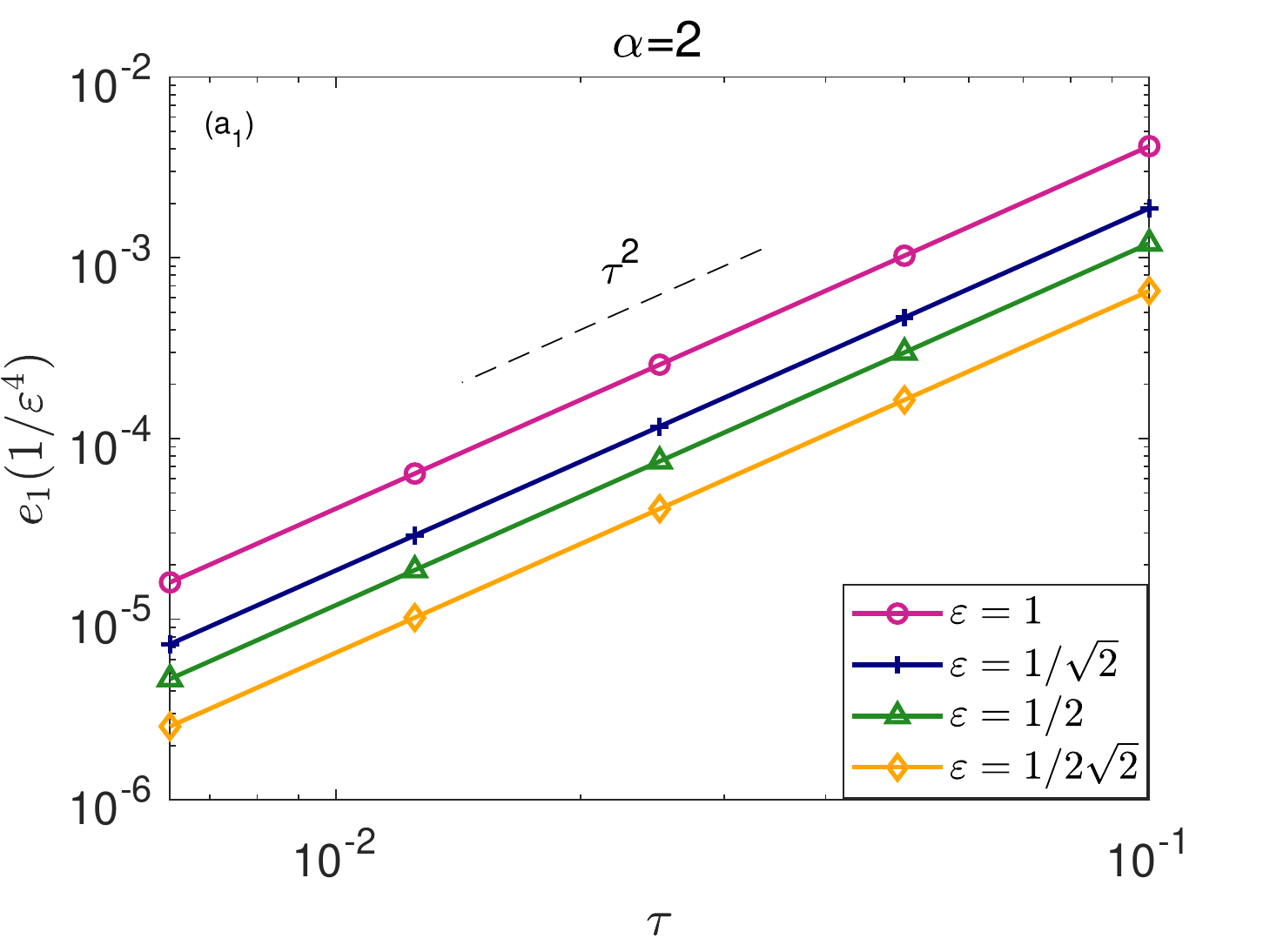}
\includegraphics[scale=.34]{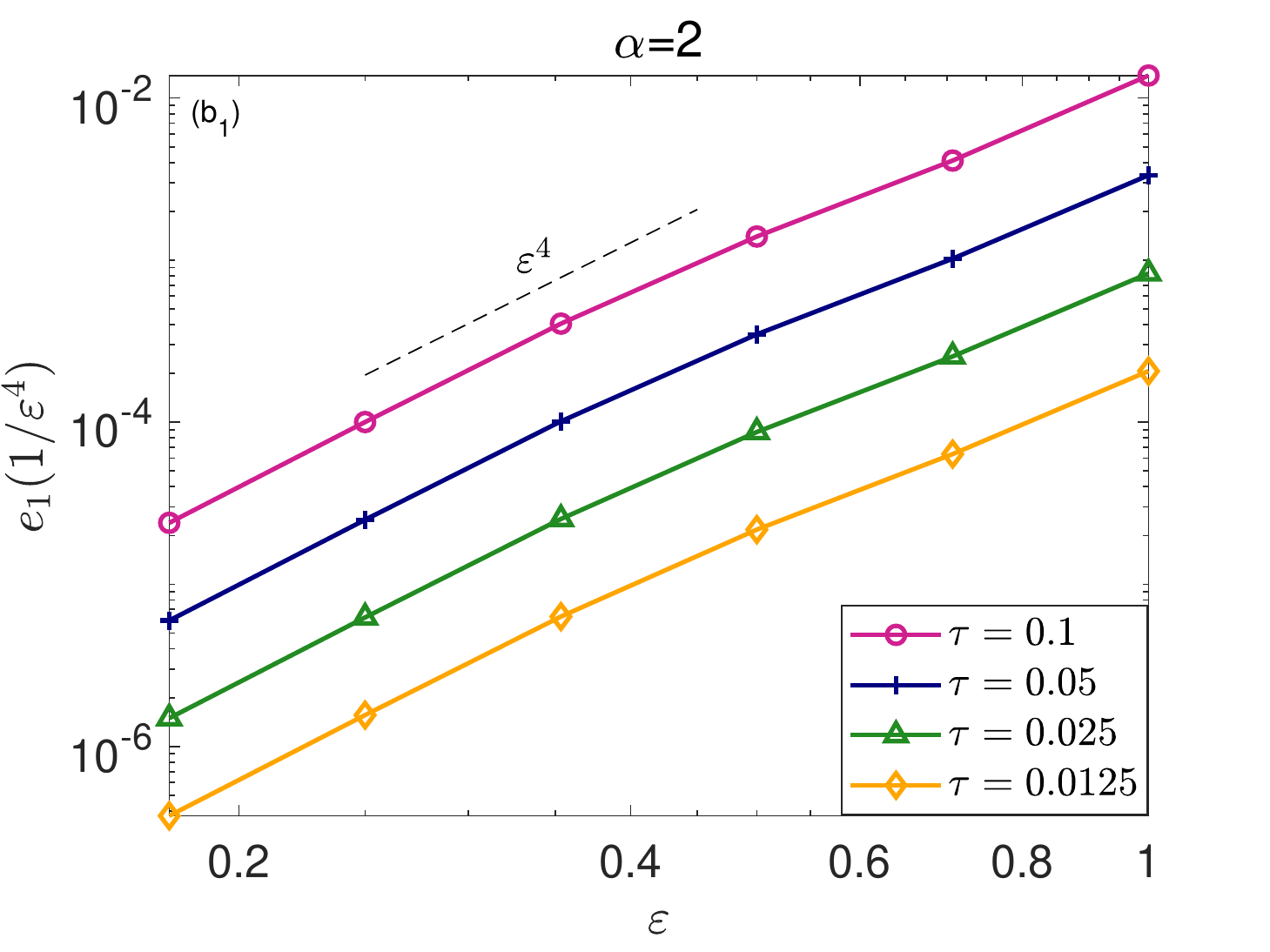}
\includegraphics[scale=.34]{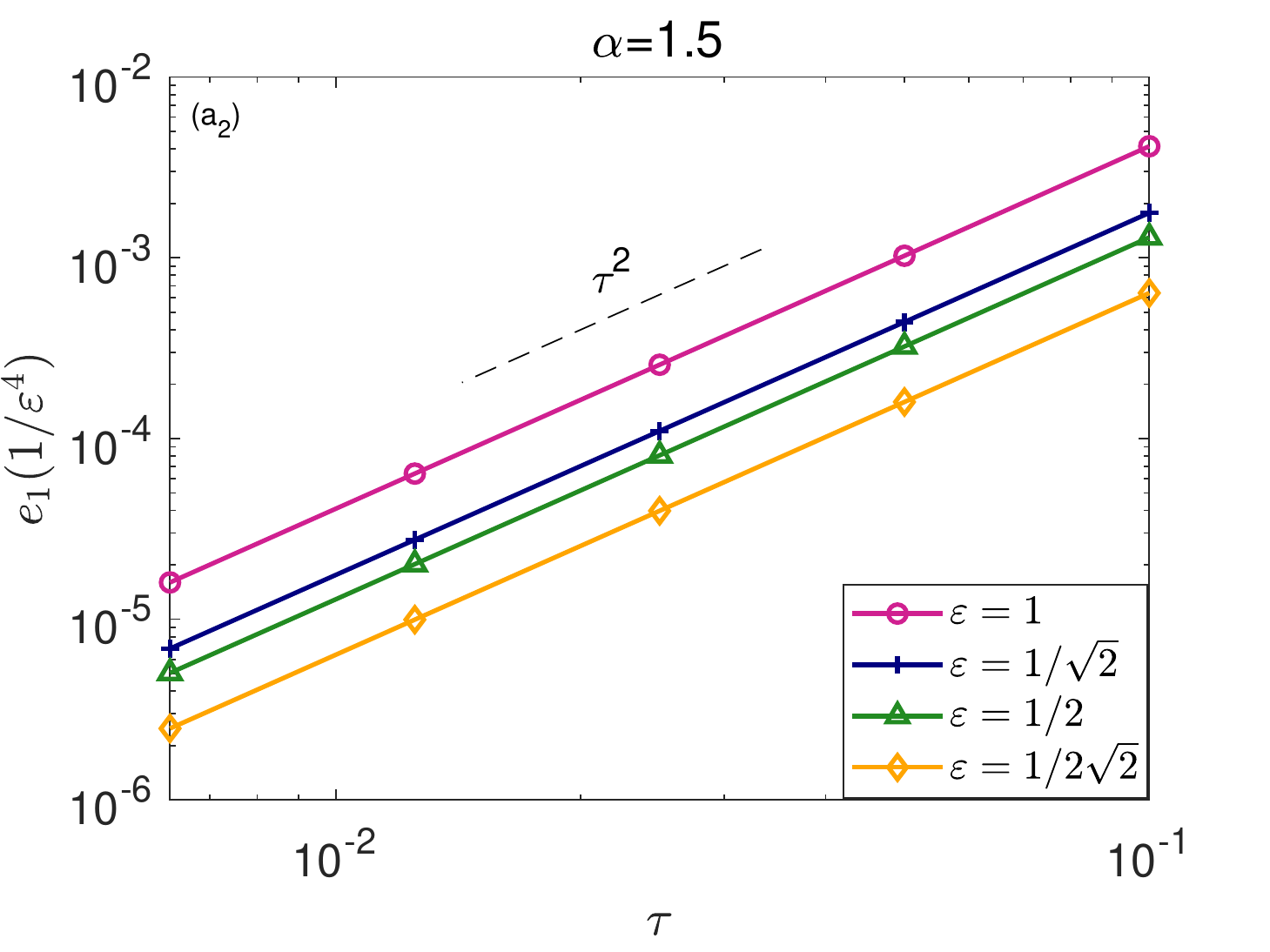}
\includegraphics[scale=.34]{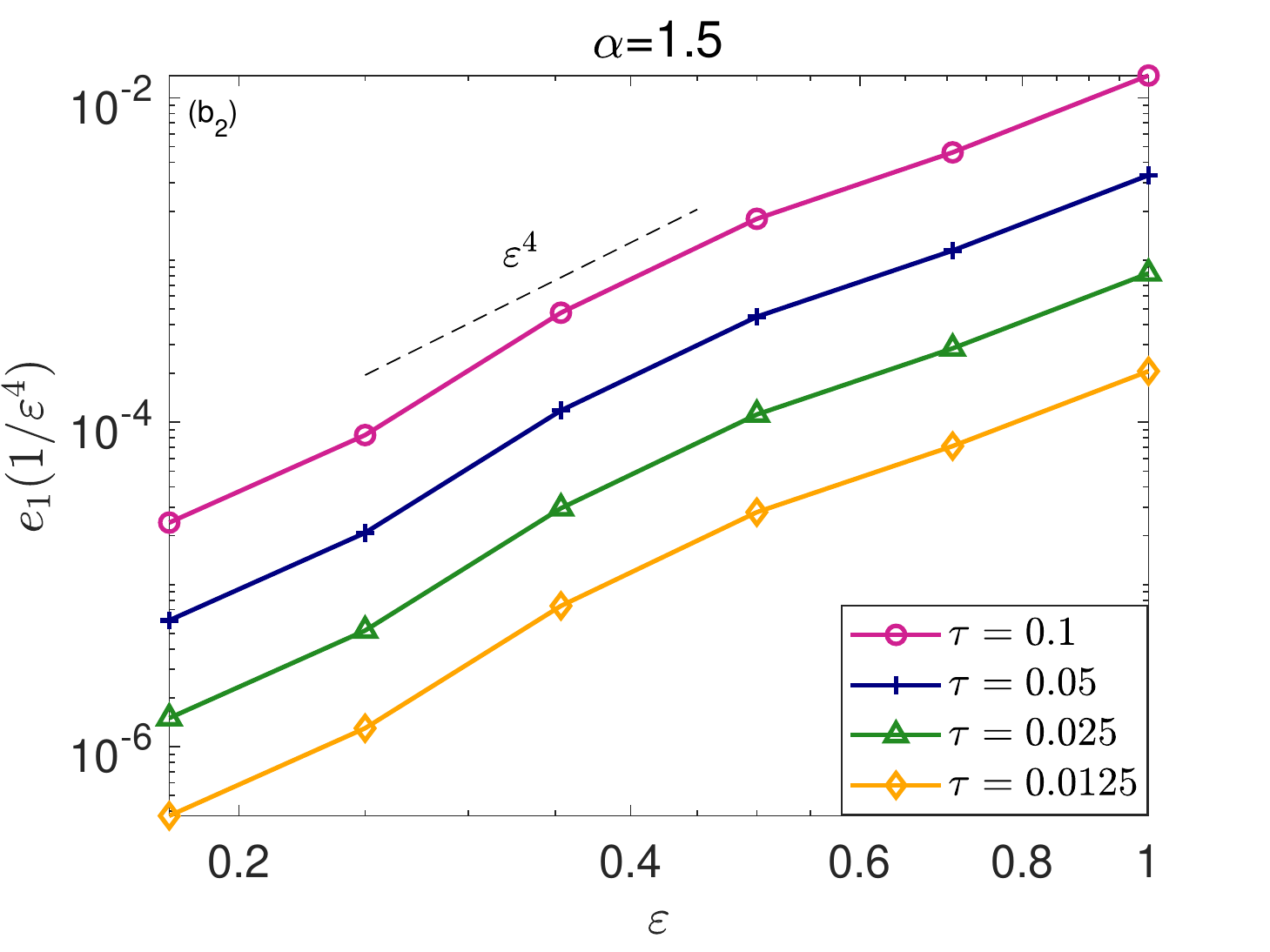}
\includegraphics[scale=.34]{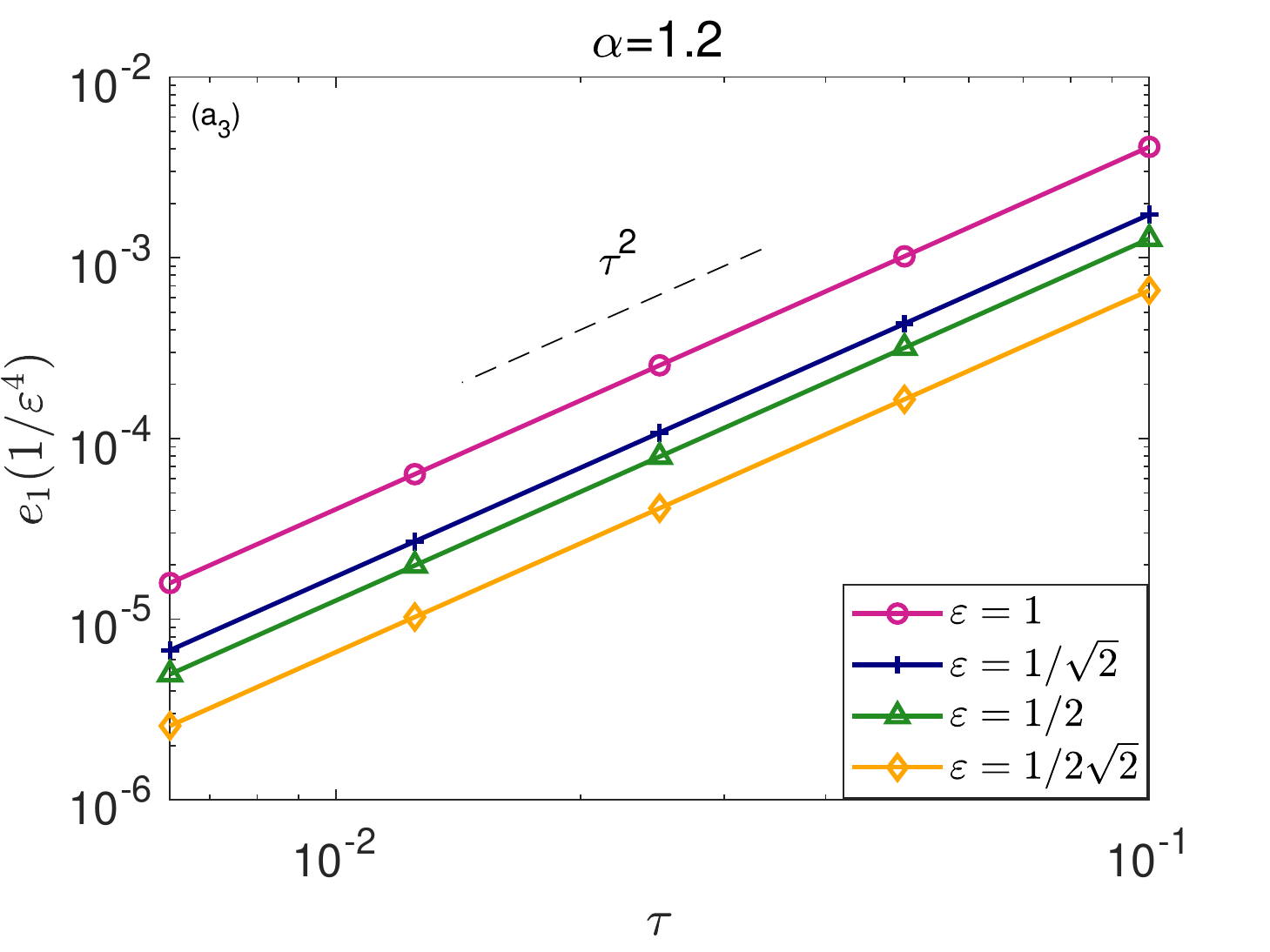}
\includegraphics[scale=.34]{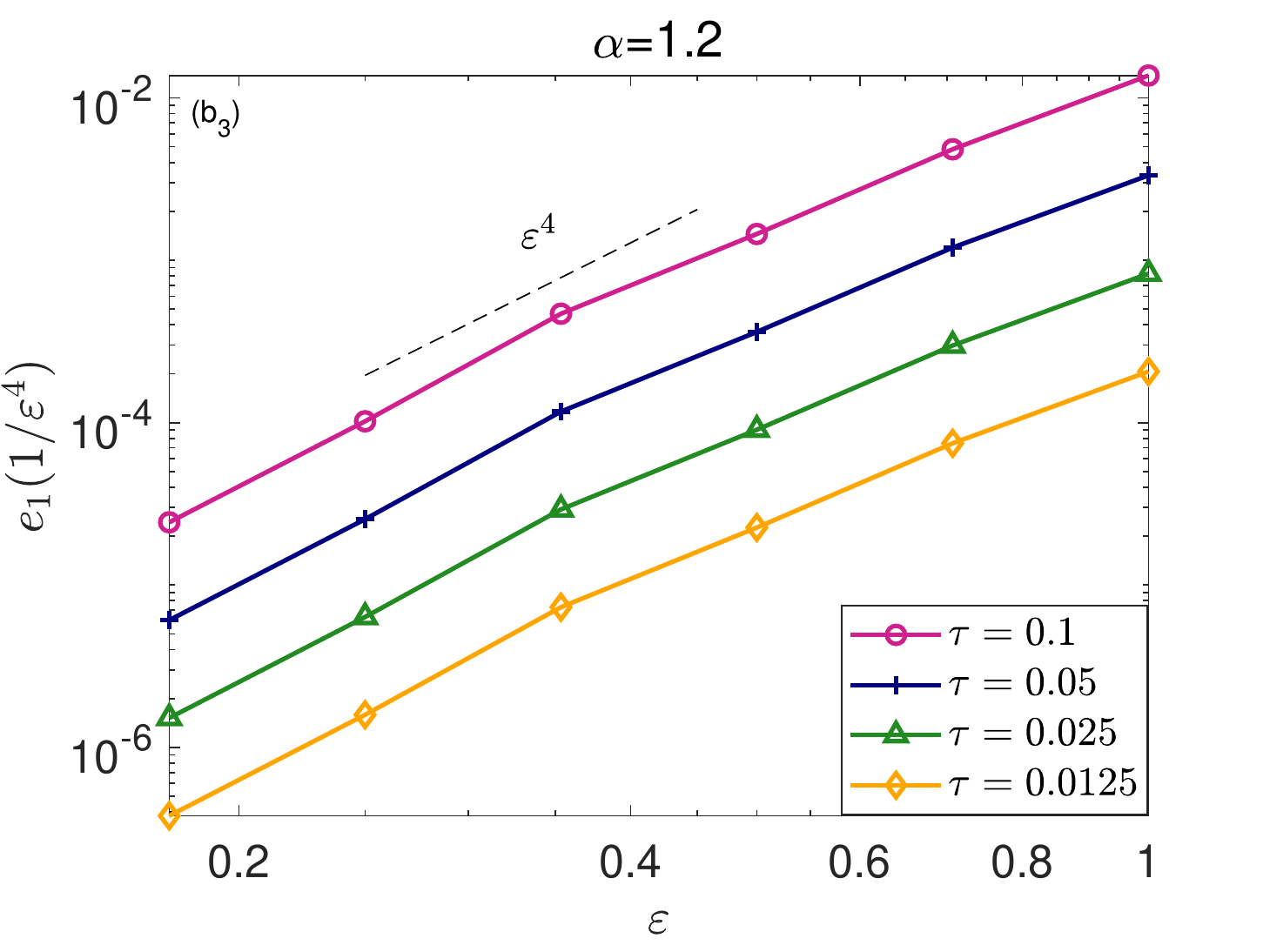}
\caption{Long-time temporal errors for the NSFKGE \eqref{eq4.1.1} with $p=2$ in 1D when $\alpha$ is taken as 2, 1.5 and 1.2 at $t=1/{\varepsilon^{4}}$ respectively.}
\label{fig5.3}
\end{figure}

\clearpage
\begin{figure}[htb]
\centering
\includegraphics[scale=.35]{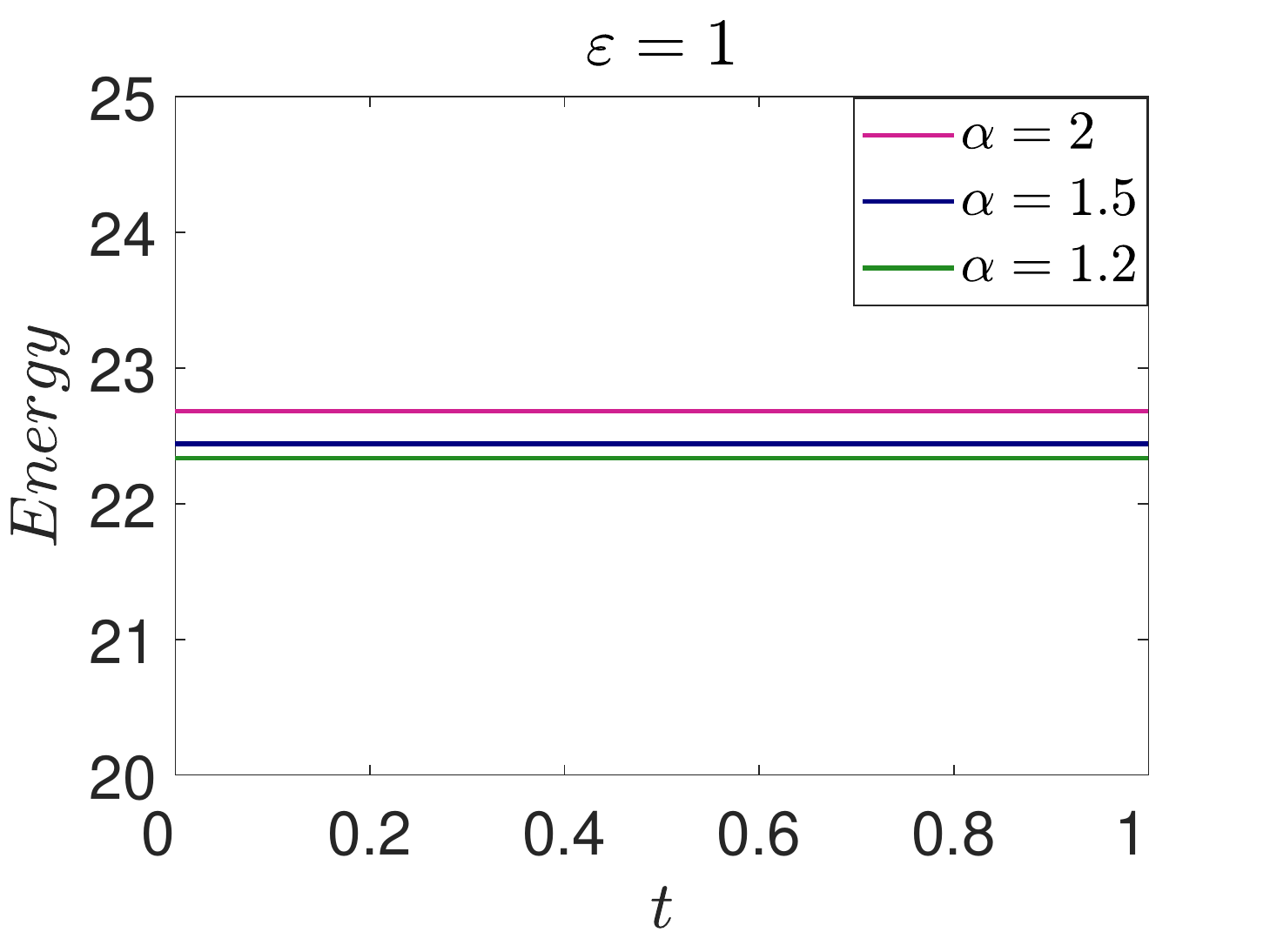}
\includegraphics[scale=.35]{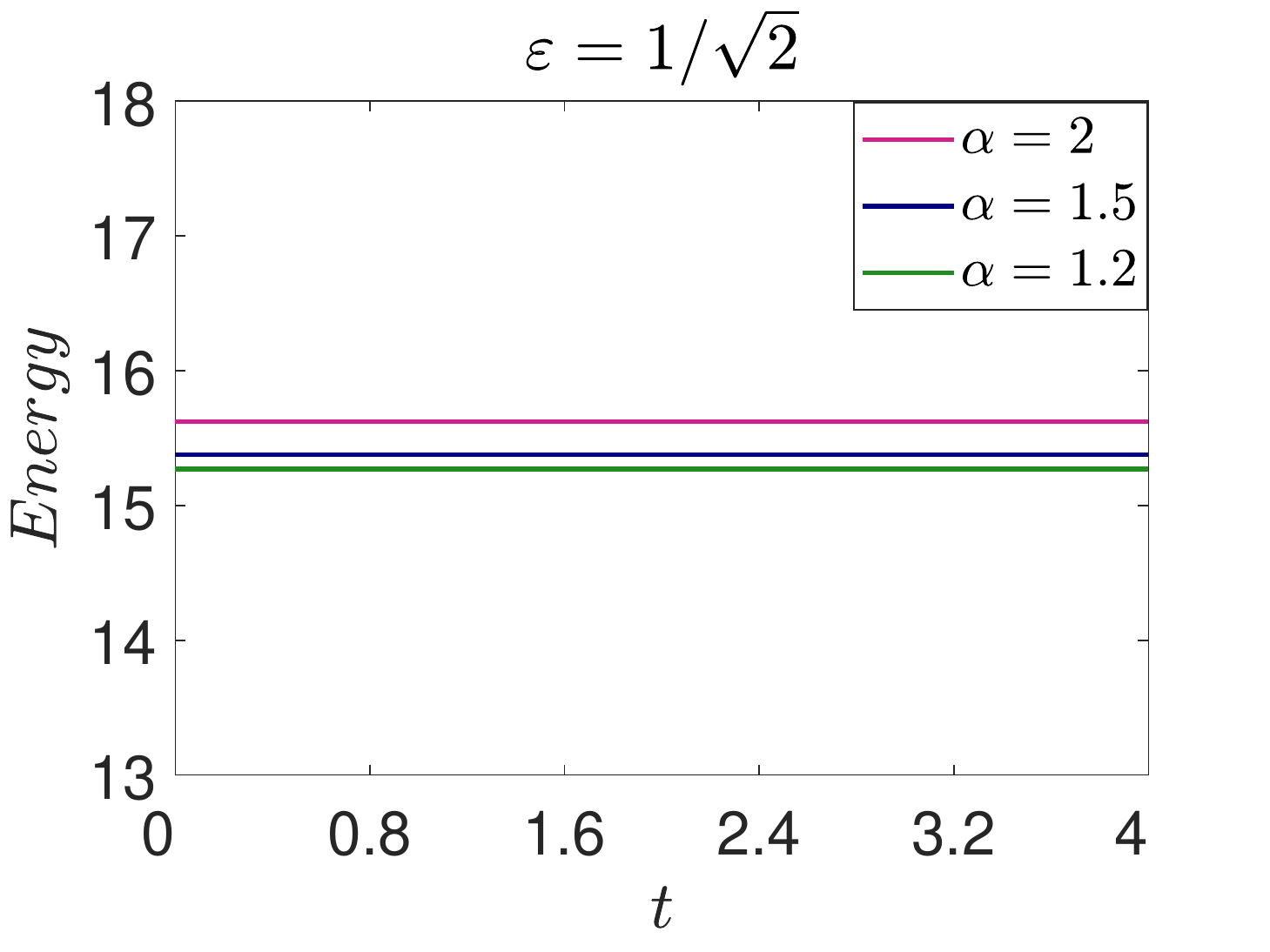}
\includegraphics[scale=.35]{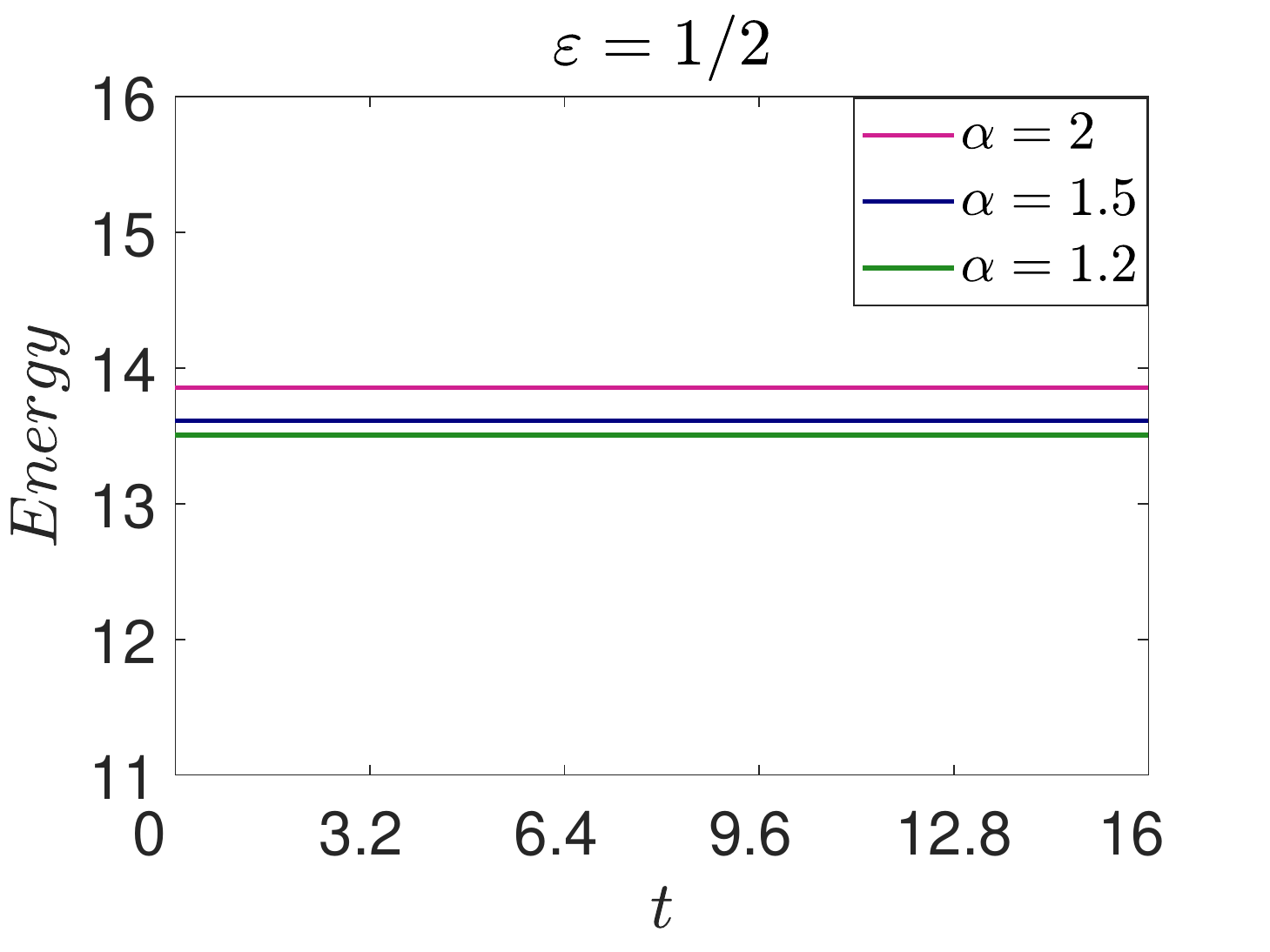}
\includegraphics[scale=.35]{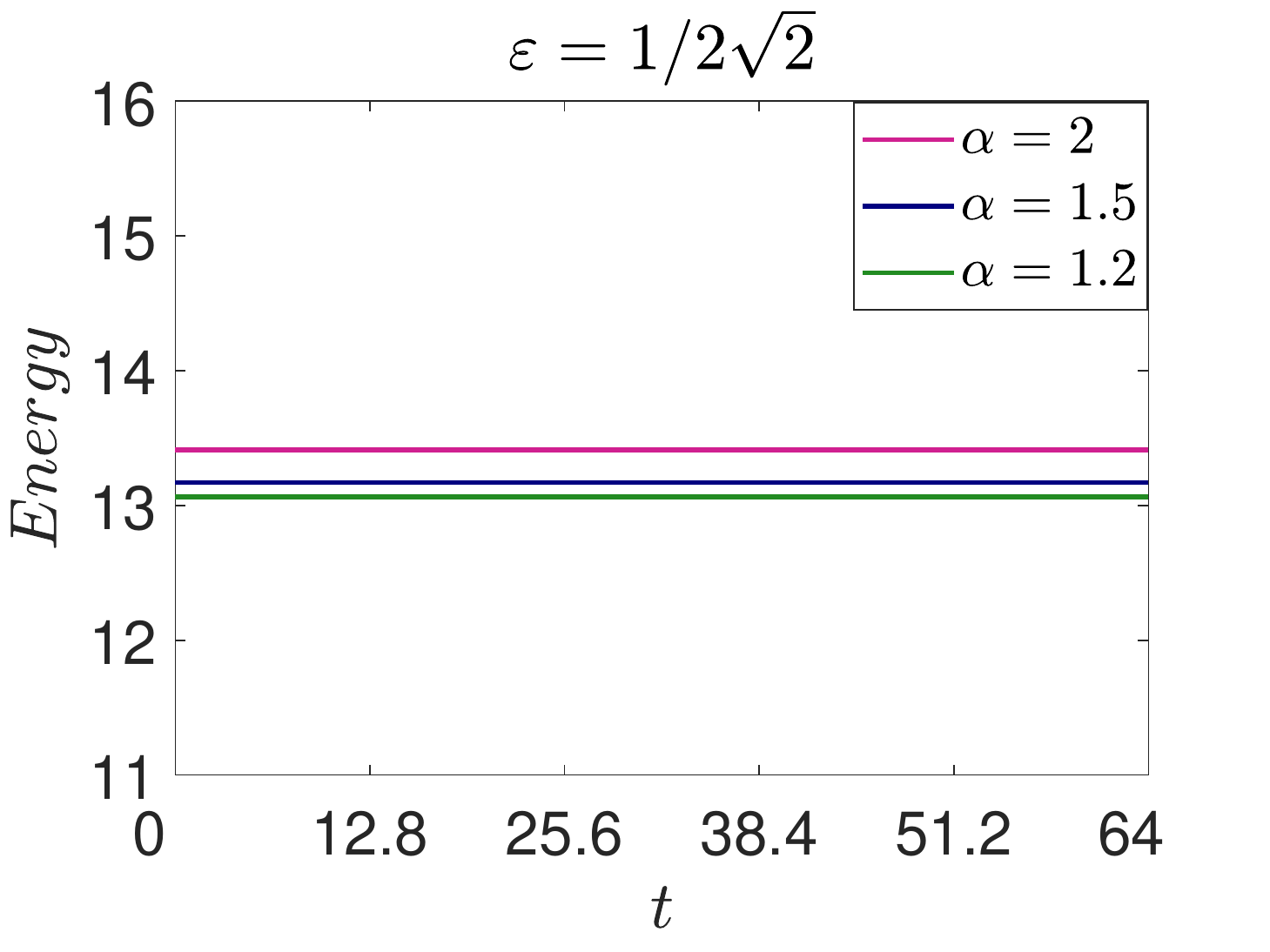}
\caption{Energy evolutions for the NSFKGE \eqref{eq4.1.1} with $p=2$ in 1D when $\alpha$ is taken as 2, 1.5 and 1.2.}
\label{fig5.4}
\end{figure}

Fig. \ref{fig5.1} displays the long-time temporal errors for the NSFKGE \eqref{eq4.1.1} with  fixed time step $\tau$ (here we take $\tau=10^{-2}$) and different $\varepsilon$ in 1D when $\alpha$ is taken as 2, 1.5 and 1.2 respectively. It indicates that the improved uniform error bounds in $H^{\alpha/2}-$norm are $O(\varepsilon^{4}\tau^{2})$ up to time $O(1/\varepsilon^{4})$, and the fractional index $\alpha$ will not affect the result. Fig. \ref{fig5.2} presents long-time spatial errors for the NSFKGE \eqref{eq4.1.1} in 1D when $\alpha$ is taken as 2, 1.5 and 1.2 at $t=1/{\varepsilon^{4}}$ respectively. Fig. \ref{fig5.2}$(a_{1})-(a_{3})$  indicates the spectral accuracy for the NSFKGE \eqref{eq4.1.1} in space, and Fig. \ref{fig5.2}$(b_{1})-(b_{3})$ illustrates that
the small parameter $\varepsilon$ has no effect on the spatial errors. Fig. \ref{fig5.3}$(a_{1})-(a_{3})$  shows that there is second-order convergence accuracy in time, Fig. \ref{fig5.3}$(b_{1})-(b_{3})$  further validates that the $H^{\alpha/2}$-norm errors behave like $O\left(\varepsilon^4 \tau^2\right)$ up to the time at $t=1/{\varepsilon^{4}}$. From Fig. \ref{fig5.4}, we can observe that the numerical method is energy conservation.

\subsection{The long-time dynamics in 2D}
\label{sec:5.2}
In this subsection, we show a $2 \mathrm{D}$ example with $p=1$, the domain $(x, y) \in \Omega=(0,1) \times(0,2 \pi)$ and the initial conditions are
\begin{equation}
\label{eq5.2.1}
\begin{aligned}
\psi_0(x, y)=\frac{2}{1+\cos ^2(2 \pi x+y)}, \quad \psi_1(x, y)=\frac{3}{2+2 \cos ^2(2 \pi x+y)}
\end{aligned}
\end{equation}

\begin{figure}[htb]
	\centering
    \subfigure[$T=0,\alpha=2$]{
		\begin{minipage}[t]{0.22\linewidth}
			\centering
			\includegraphics[width=1.2in]{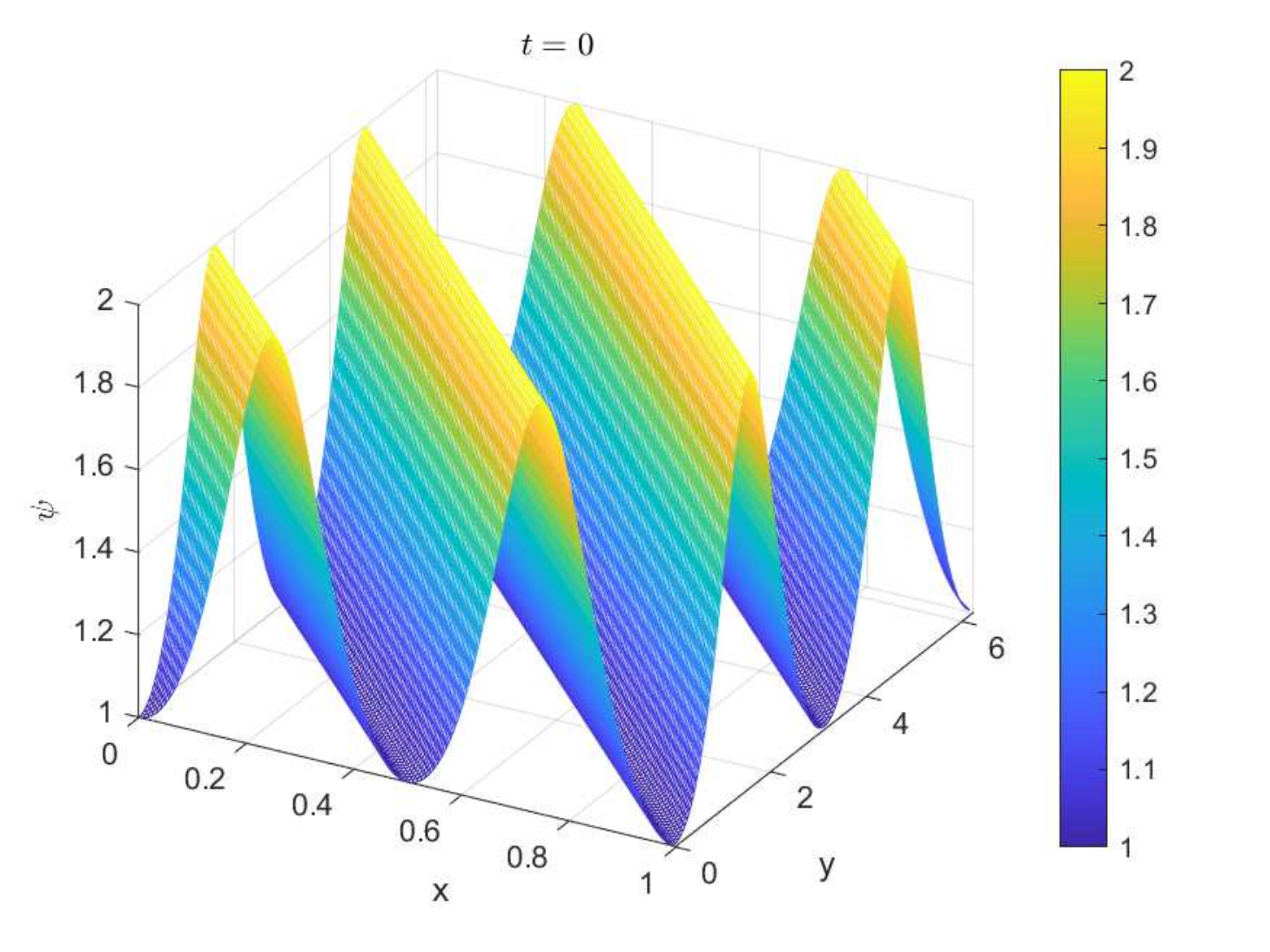}
	\end{minipage}}
	\subfigure[$T=0,\alpha=1.7$]{
		\begin{minipage}[t]{0.22\linewidth}
			\centering
			\includegraphics[width=1.2in]{f_2D_t-eps-converted-to}
	\end{minipage}}
	\subfigure[$T=0,\alpha=1.4$]{
		\begin{minipage}[t]{0.22\linewidth}
			\centering
			\includegraphics[width=1.2in]{f_2D_t-eps-converted-to}
	\end{minipage}}
	\subfigure[$T=0,\alpha=1.1$]{
		\begin{minipage}[t]{0.22\linewidth}
			\centering
			\includegraphics[width=1.2in]{f_2D_t-eps-converted-to}
	\end{minipage}}
	\subfigure[$T=2,\alpha=2$]{
		\begin{minipage}[t]{0.22\linewidth}
			\centering
			\includegraphics[width=1.2in]{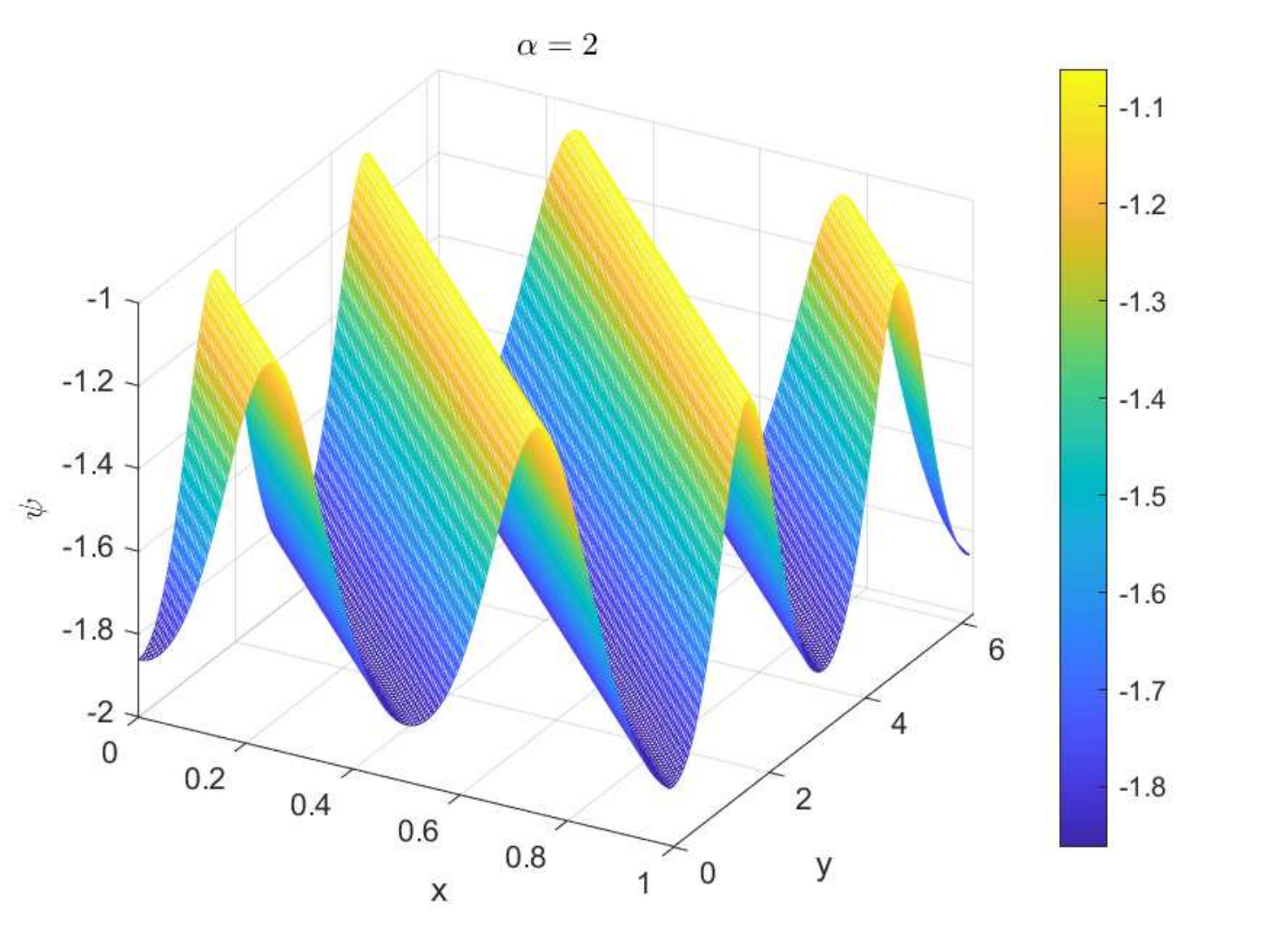}
	\end{minipage}}
	\subfigure[$T=2,\alpha=1.7$]{
		\begin{minipage}[t]{0.22\linewidth}
			\centering
			\includegraphics[width=1.2in]{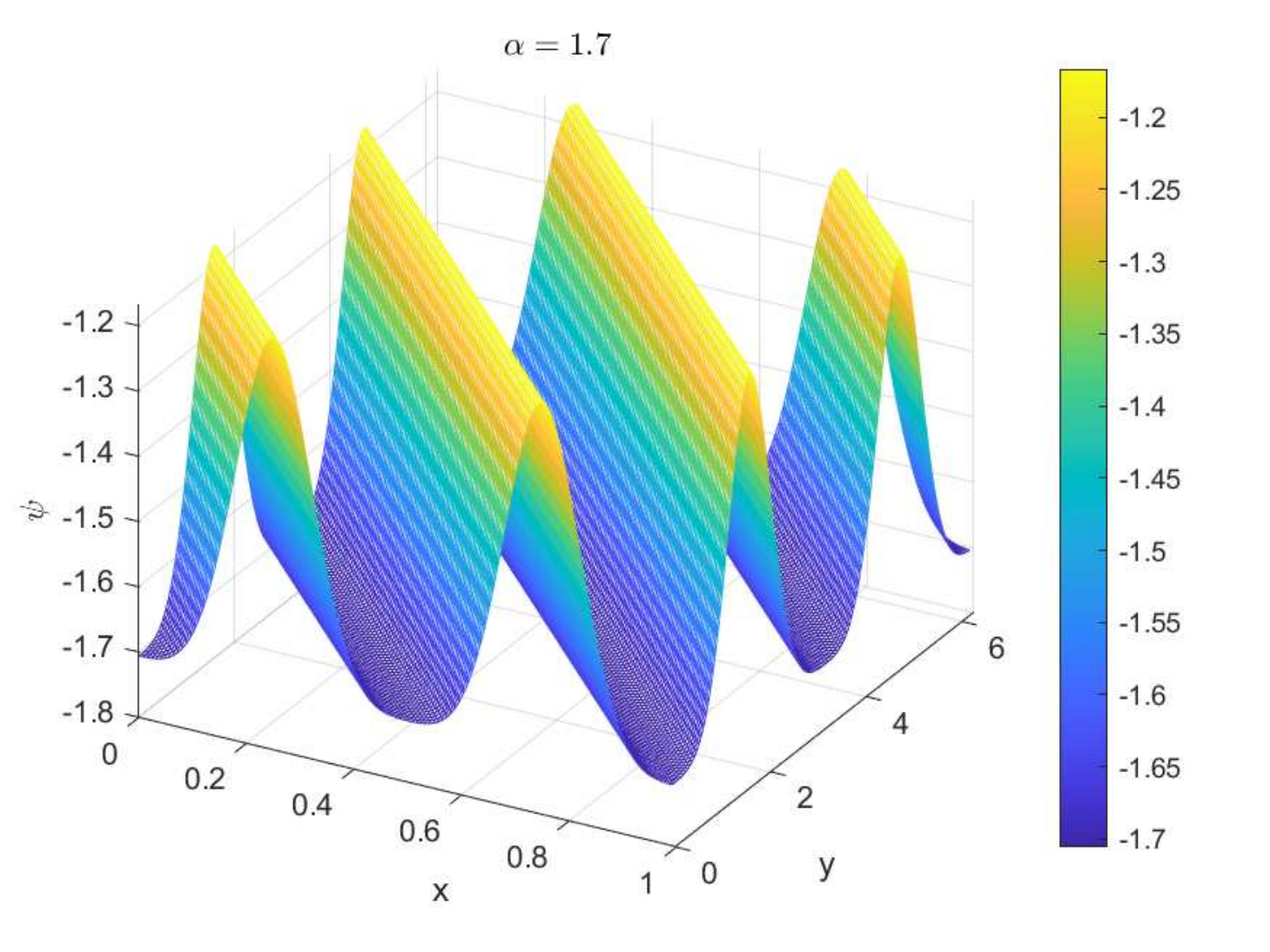}
	\end{minipage}}
	\subfigure[$T=2,\alpha=1.4$]{
		\begin{minipage}[t]{0.22\linewidth}
			\centering
			\includegraphics[width=1.2in]{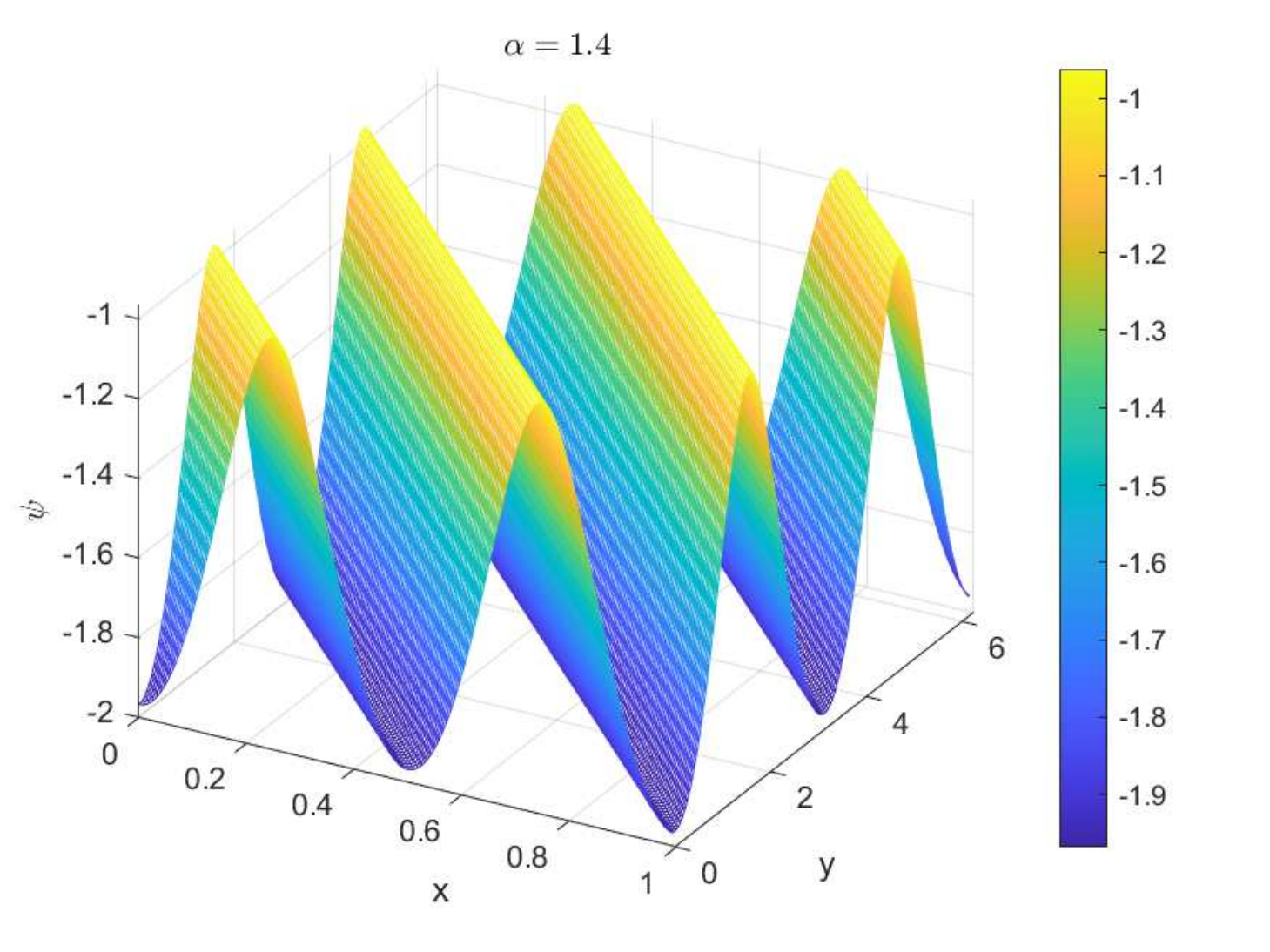}
	\end{minipage}}
	\subfigure[$T=2,\alpha=1.1$]{
		\begin{minipage}[t]{0.22\linewidth}
			\centering
			\includegraphics[width=1.2in]{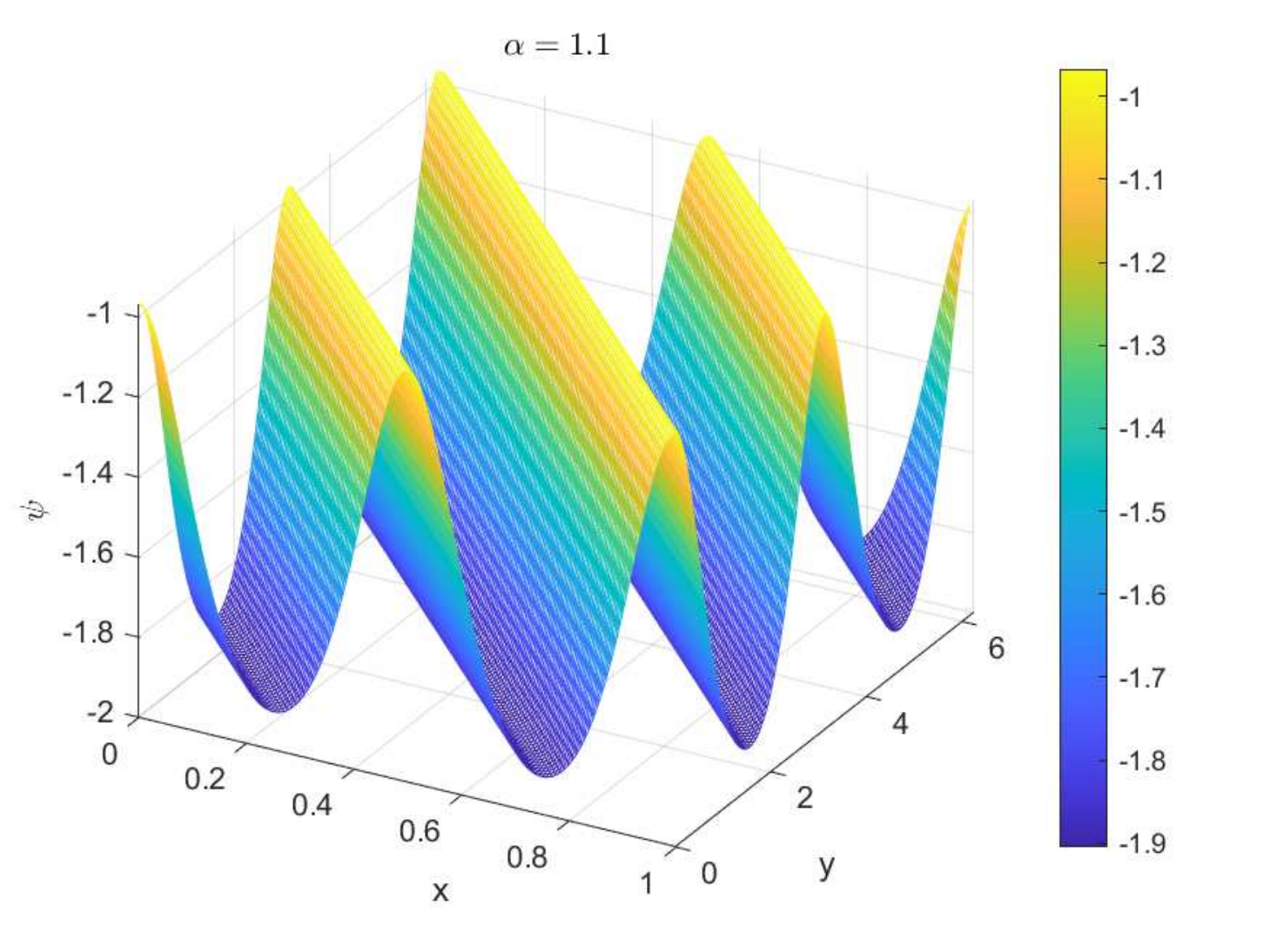}
	\end{minipage}}
	\subfigure [$T=8,\alpha=2$]{
		\begin{minipage}[t]{0.22\linewidth}
			\centering
			\includegraphics[width=1.2in]{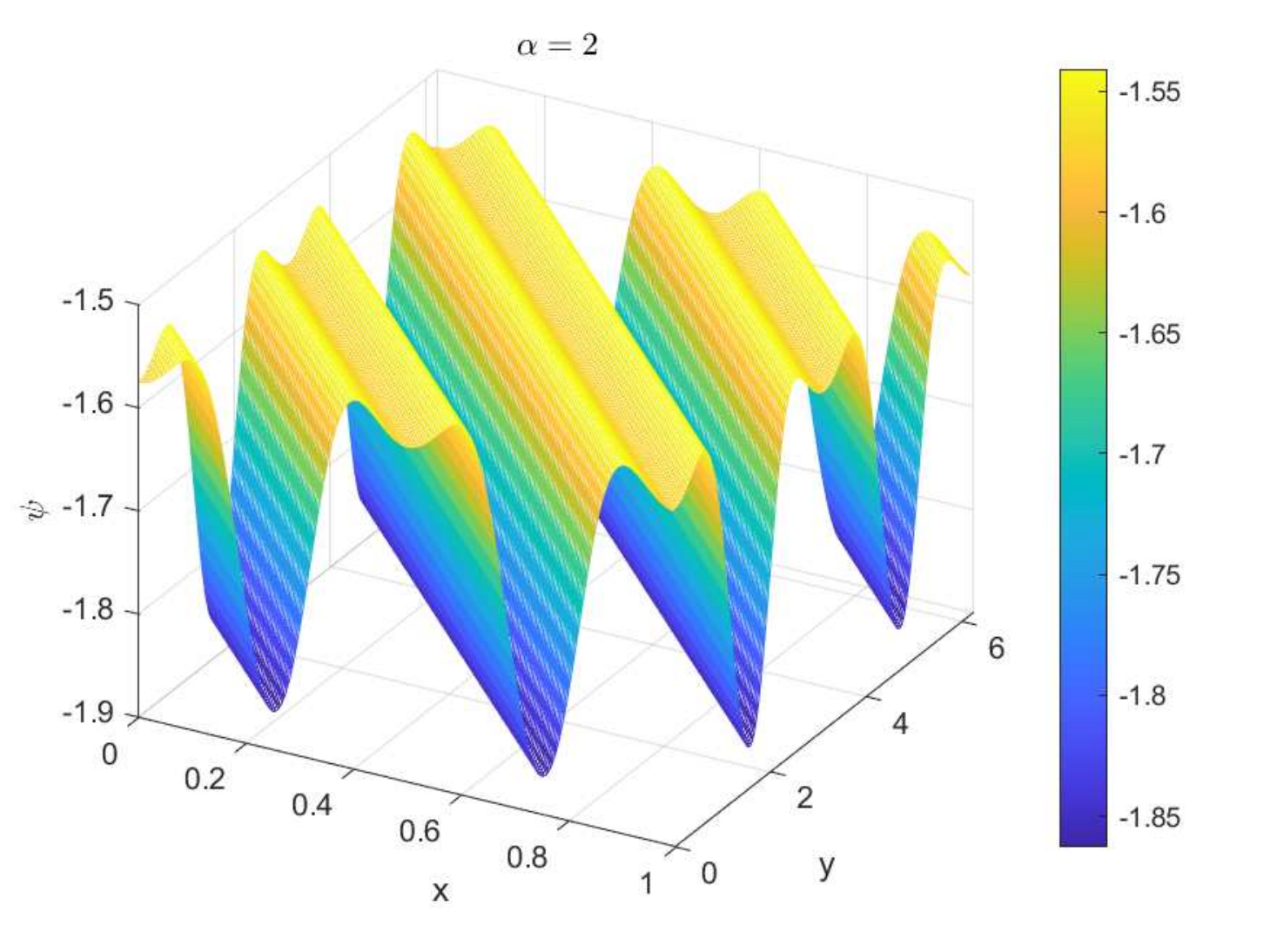}
	\end{minipage}}
	\subfigure[$T=8,\alpha=1.7$]{
		\begin{minipage}[t]{0.22\linewidth}
			\centering
			\includegraphics[width=1.2in]{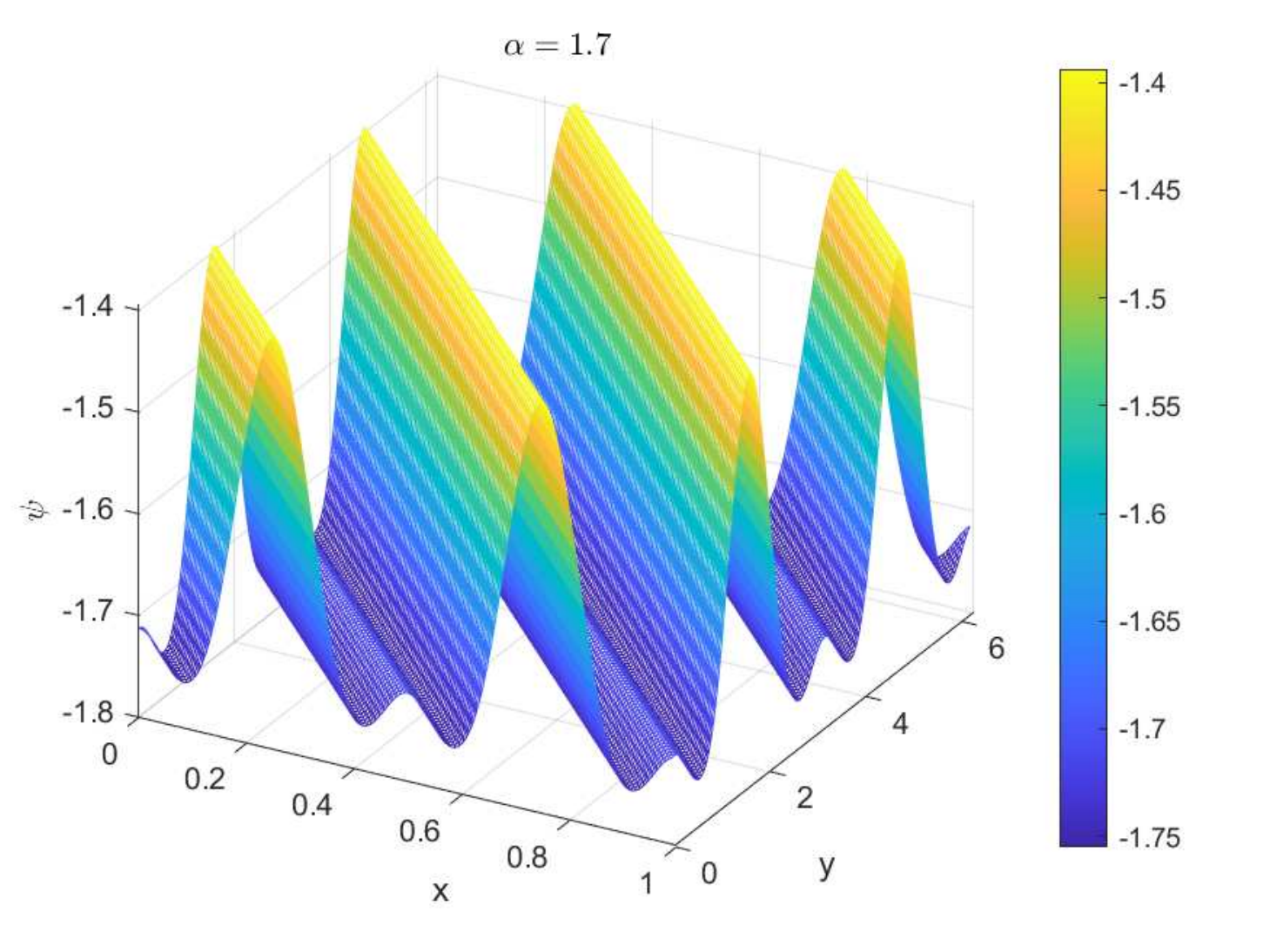}
	\end{minipage}}
	\subfigure[$T=8,\alpha=1.4$]{
		\begin{minipage}[t]{0.22\linewidth}
			\centering
			\includegraphics[width=1.2in]{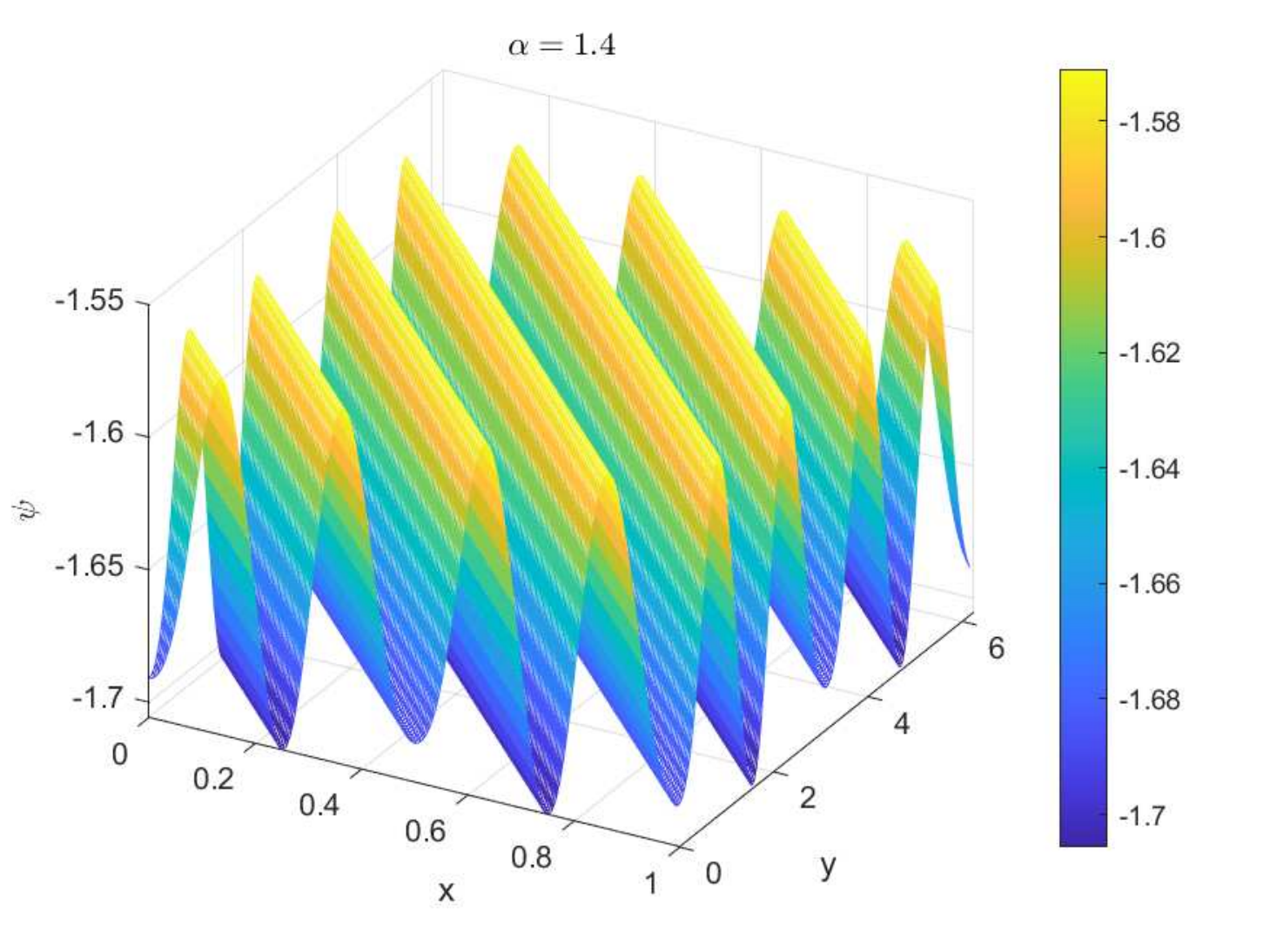}
	\end{minipage}}
	\subfigure[$T=8,\alpha=1.1$]{
		\begin{minipage}[t]{0.22\linewidth}
			\centering
			\includegraphics[width=1.2in]{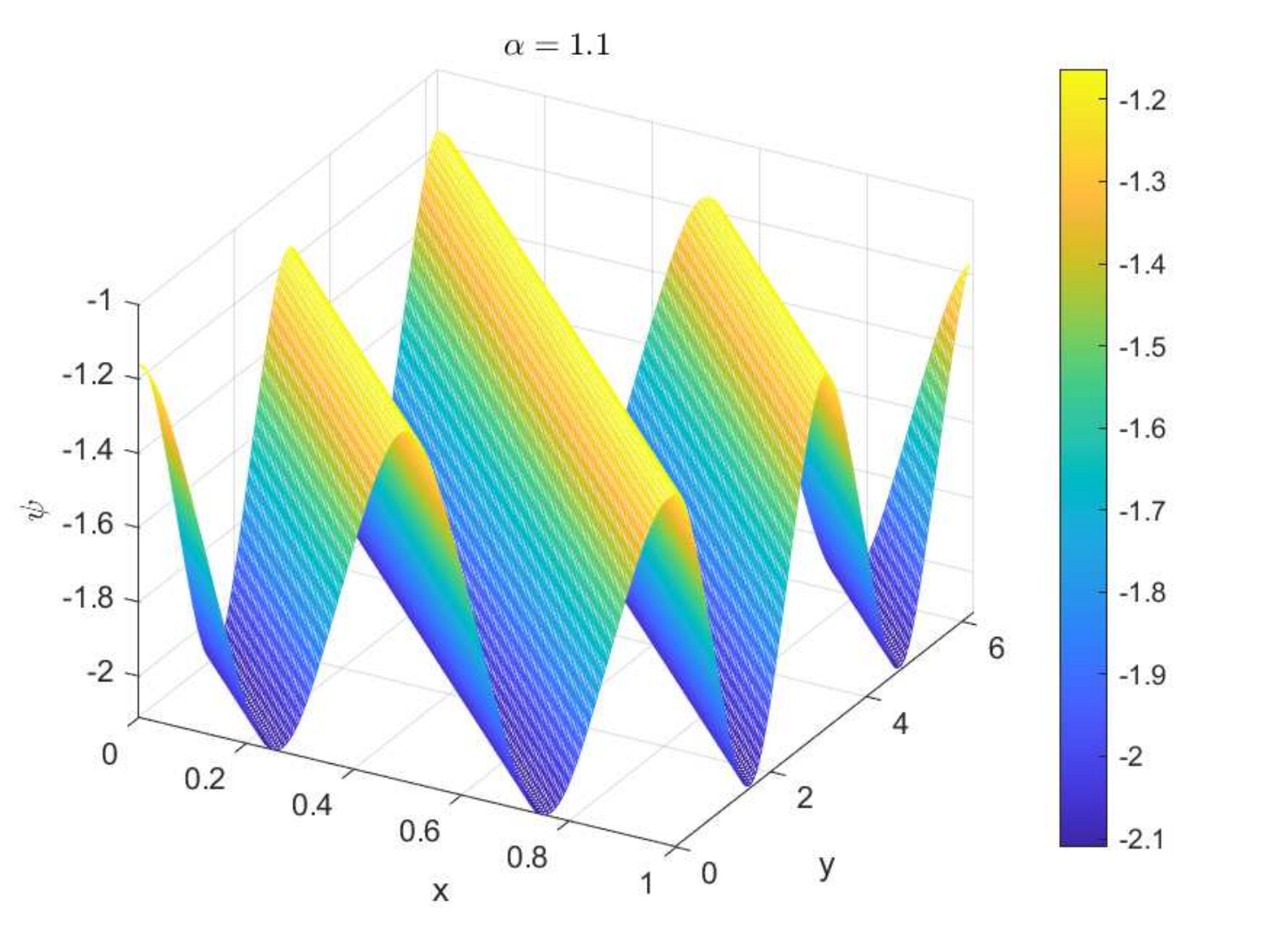}
	\end{minipage}}
	\subfigure[$T=32,\alpha=2$]{
		\begin{minipage}[t]{0.22\linewidth}
			\centering
			\includegraphics[width=1.2in]{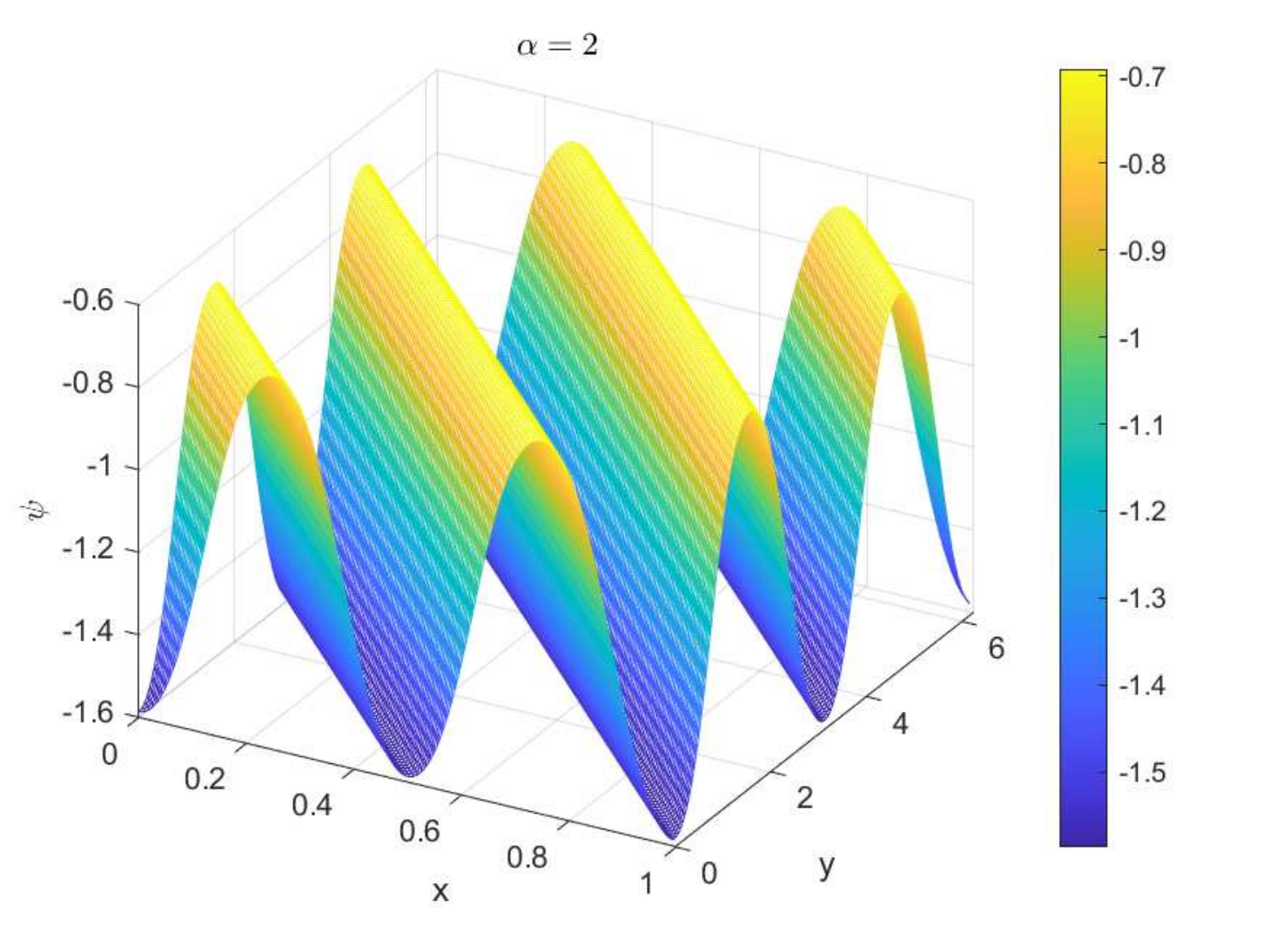}
	\end{minipage}}
	\subfigure[$T=32,\alpha=1.7$]{
		\begin{minipage}[t]{0.22\linewidth}
			\centering
			\includegraphics[width=1.2in]{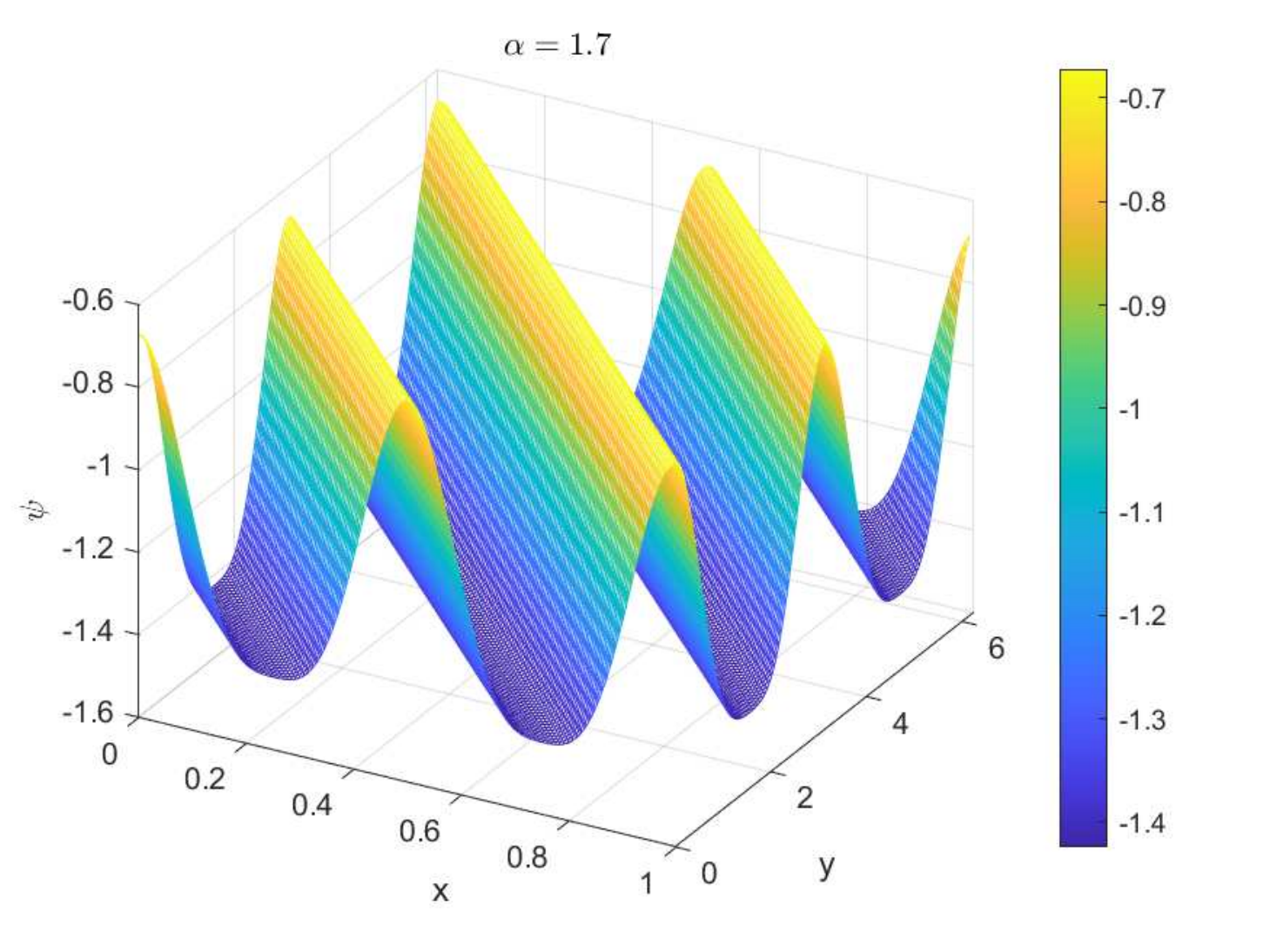}
	\end{minipage}}
	\subfigure[$T=32,\alpha=1.4$]{
		\begin{minipage}[t]{0.22\linewidth}
			\centering
			\includegraphics[width=1.2in]{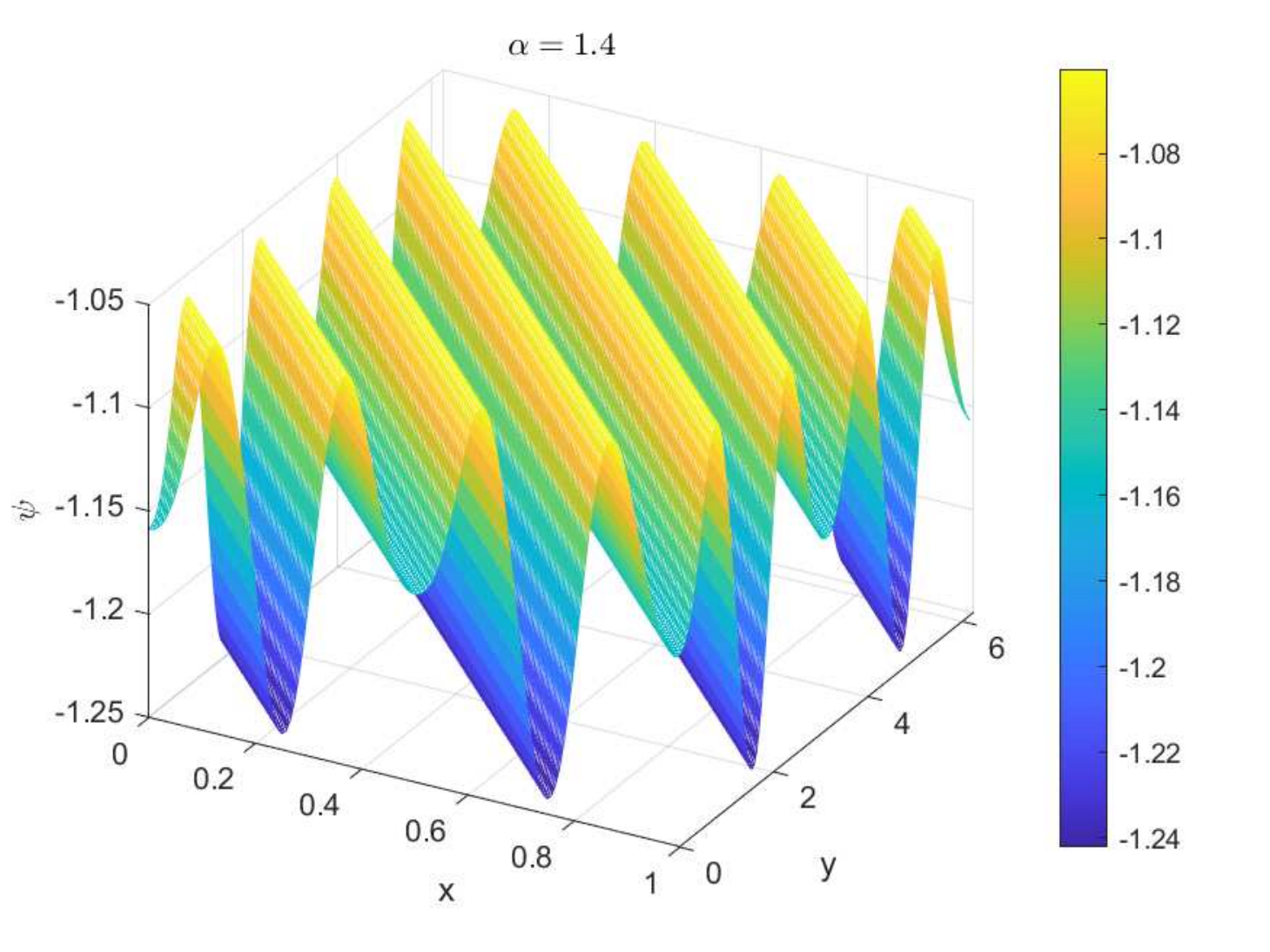}
	\end{minipage}}
	\subfigure[$T=32,\alpha=1.1$]{
		\begin{minipage}[t]{0.22\linewidth}
			\centering
			\includegraphics[width=1.2in]{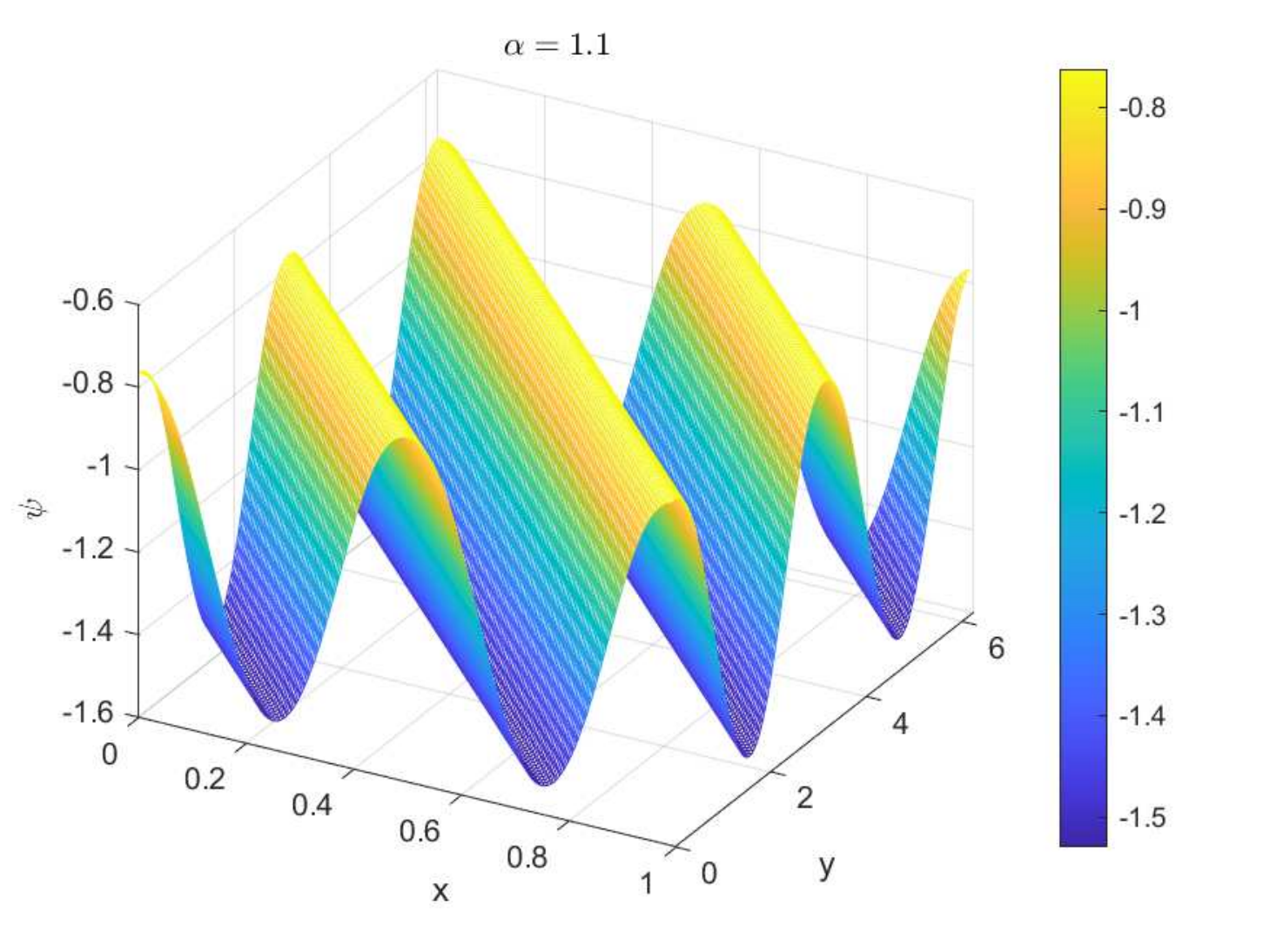}
	\end{minipage}}
\subfigure[$T=128,\alpha=2$]{
		\begin{minipage}[t]{0.22\linewidth}
			\centering
			\includegraphics[width=1.2in]{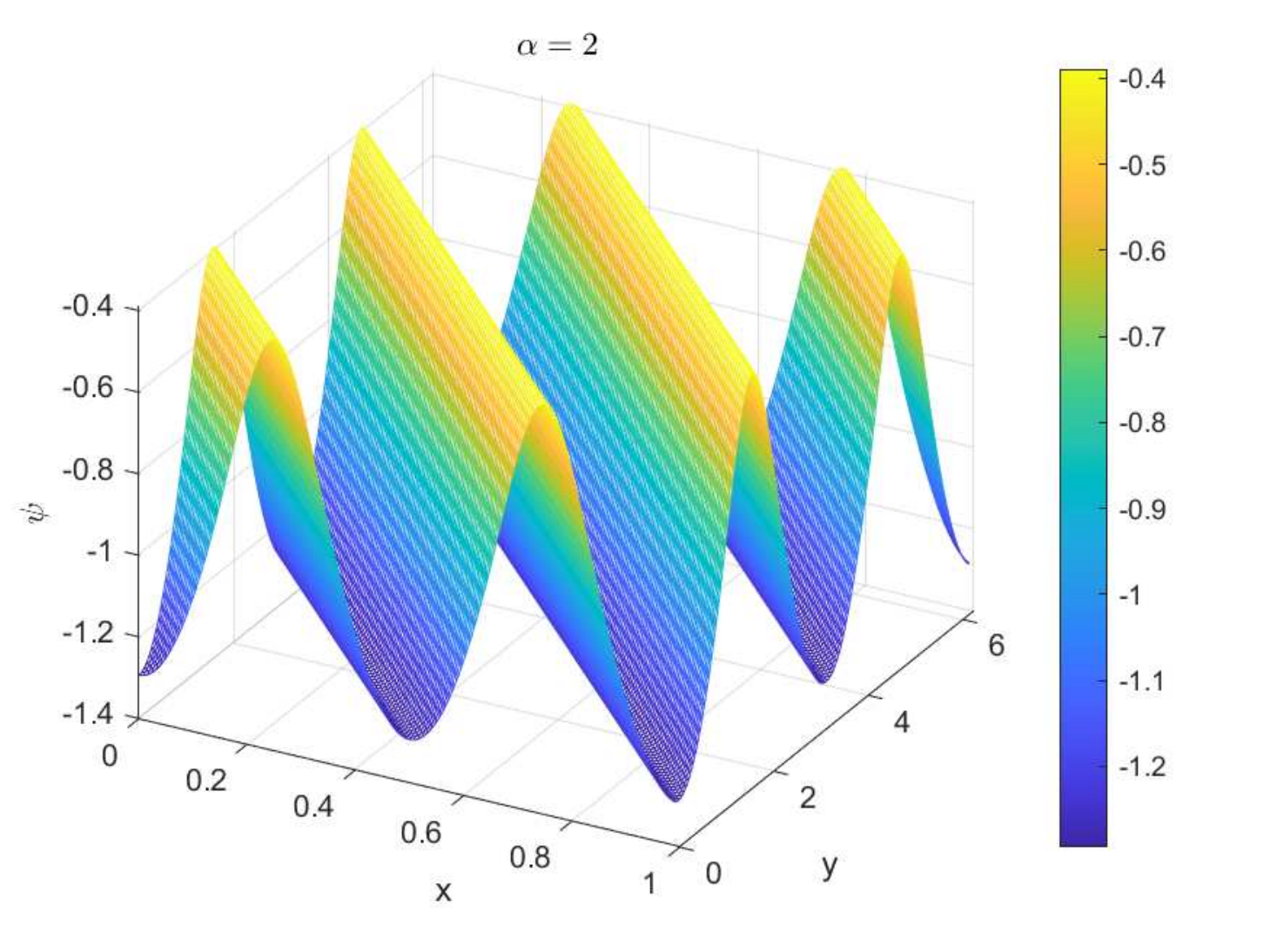}
	\end{minipage}}
	\subfigure[$T=128,\alpha=1.7$]{
		\begin{minipage}[t]{0.22\linewidth}
			\centering
			\includegraphics[width=1.2in]{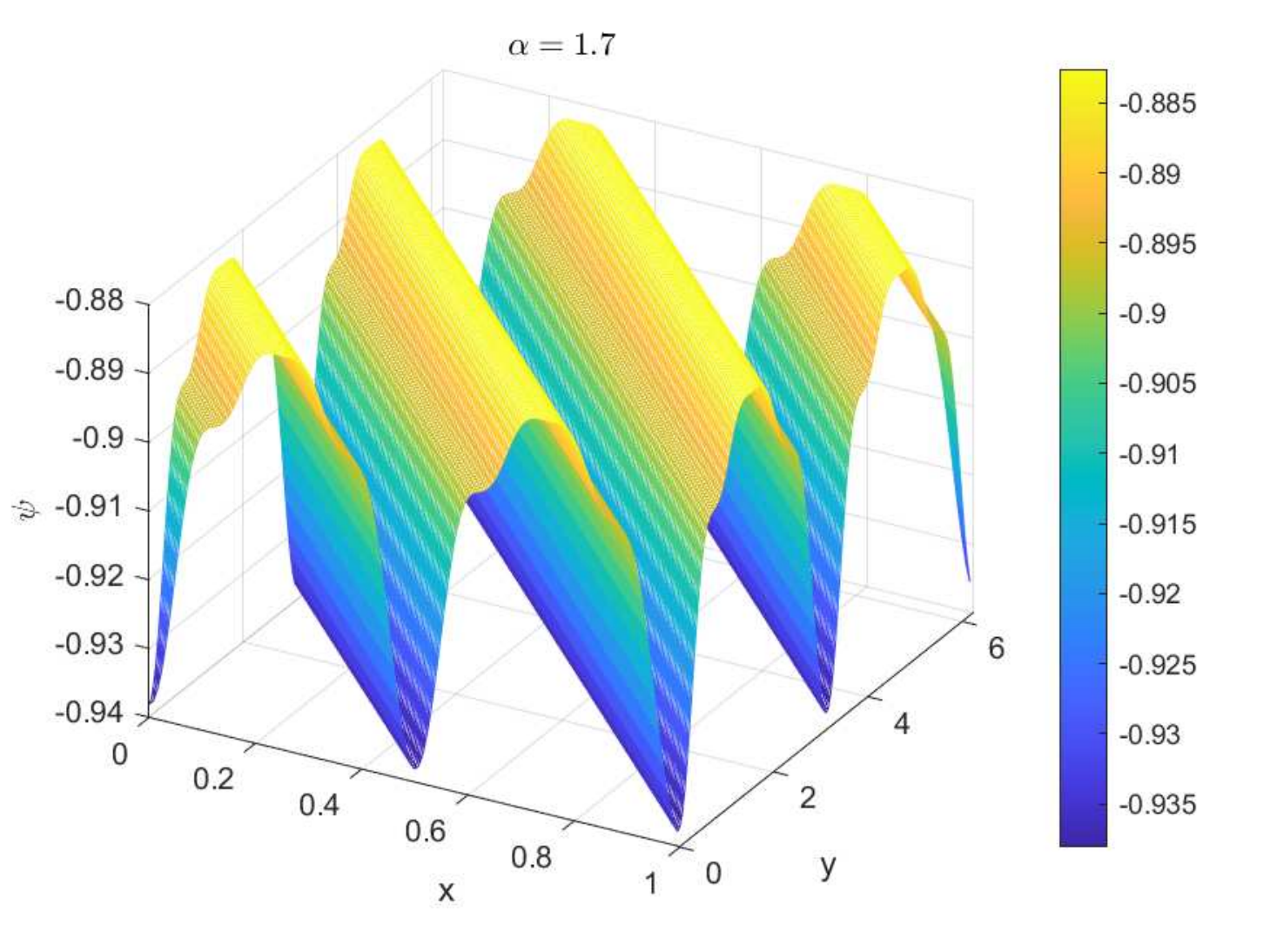}
	\end{minipage}}
	\subfigure[$T=128,\alpha=1.4$]{
		\begin{minipage}[t]{0.22\linewidth}
			\centering
			\includegraphics[width=1.2in]{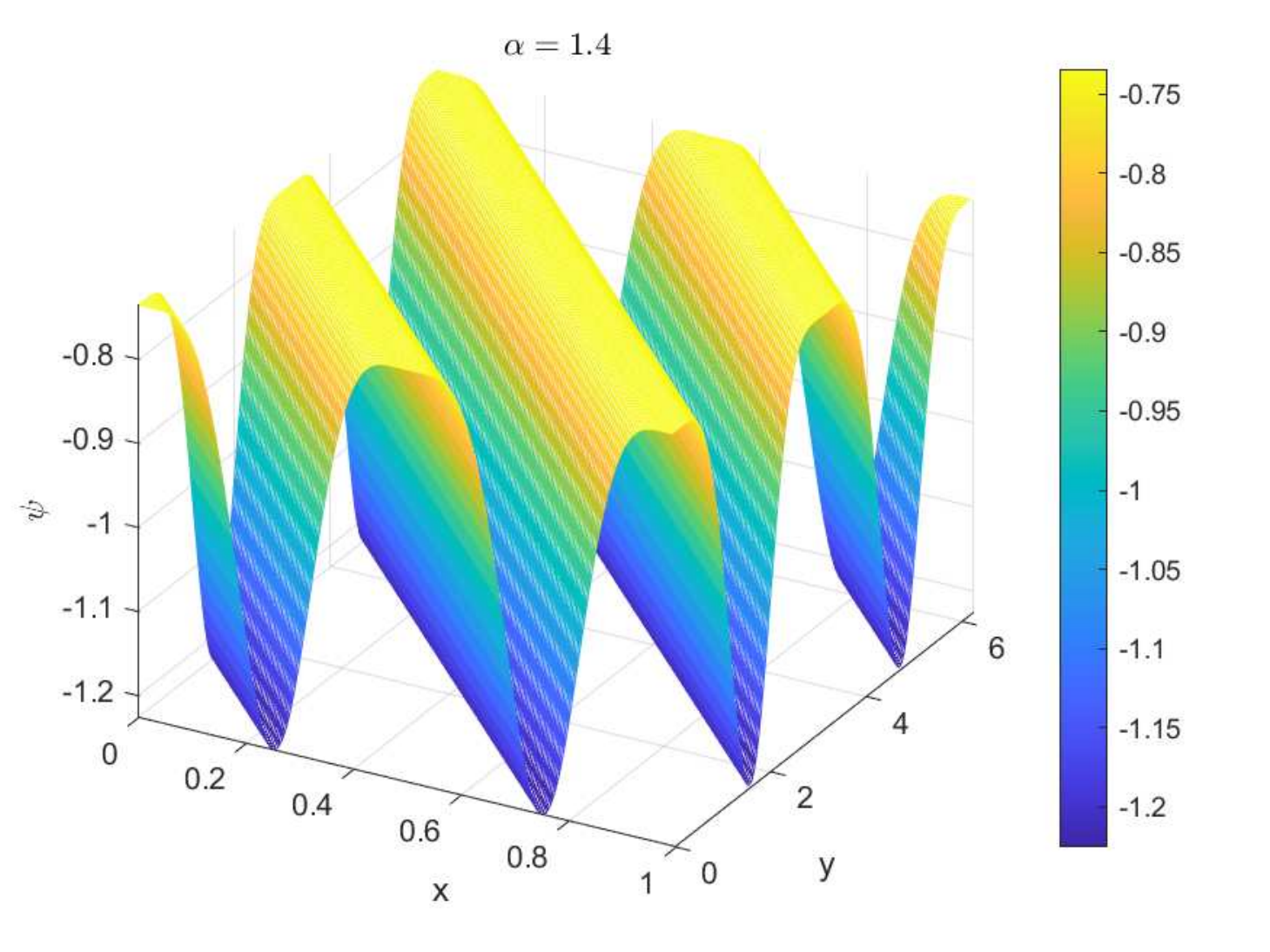}
	\end{minipage}}
	\subfigure[$T=128,\alpha=1.1$]{
		\begin{minipage}[t]{0.22\linewidth}
			\centering
			\includegraphics[width=1.2in]{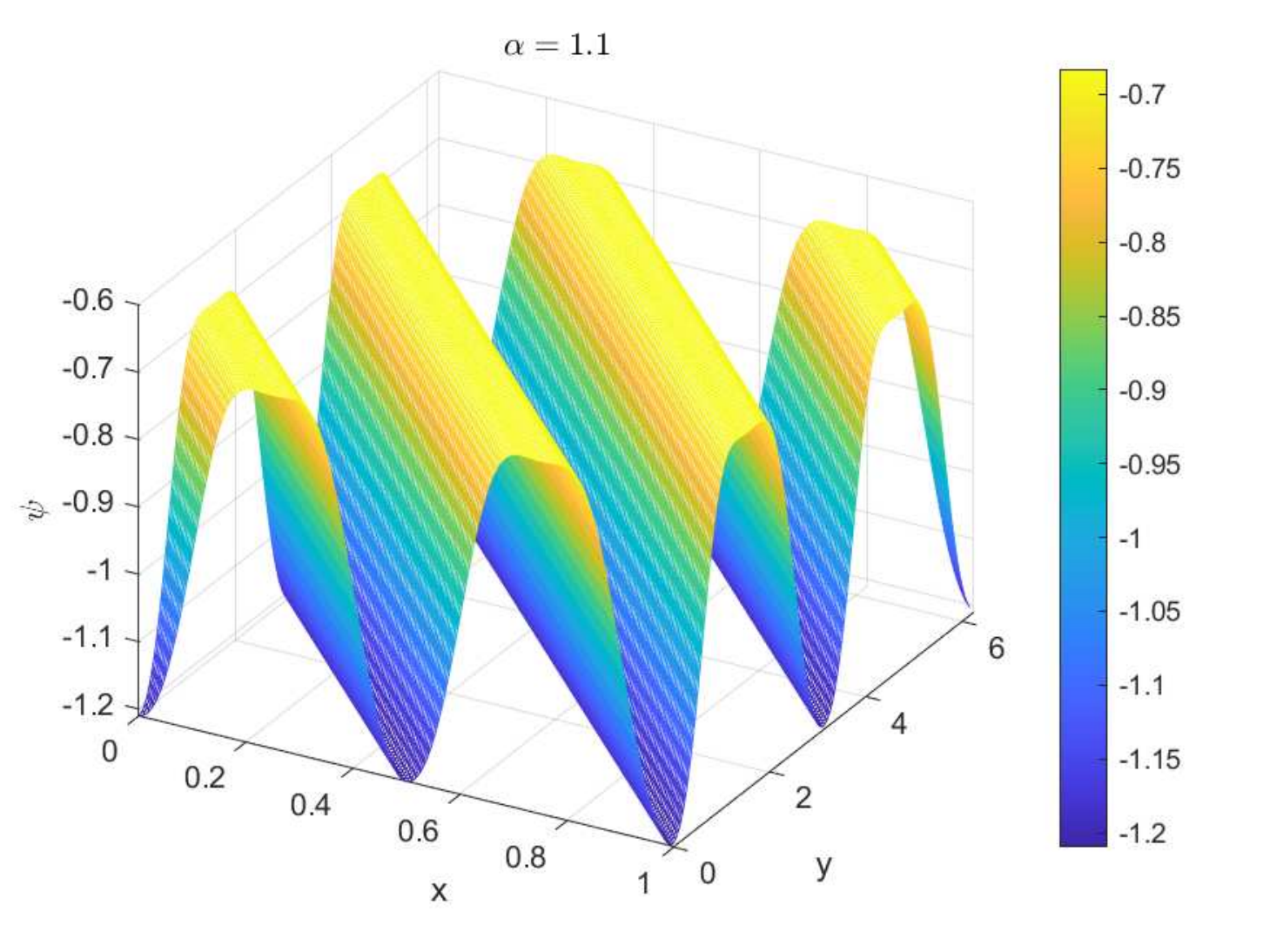}
	\end{minipage}}
	\caption{The numerical solutions for the NSFKGE \eqref{eq4.1.1} with $p=1$ in 2D}
	\label{fig5.41}
\end{figure}

\clearpage

\begin{figure}[htb]
	\centering
    \subfigure[$T=0,\alpha=2$]{
		\begin{minipage}[t]{0.22\linewidth}
			\centering
			\includegraphics[width=1.2in]{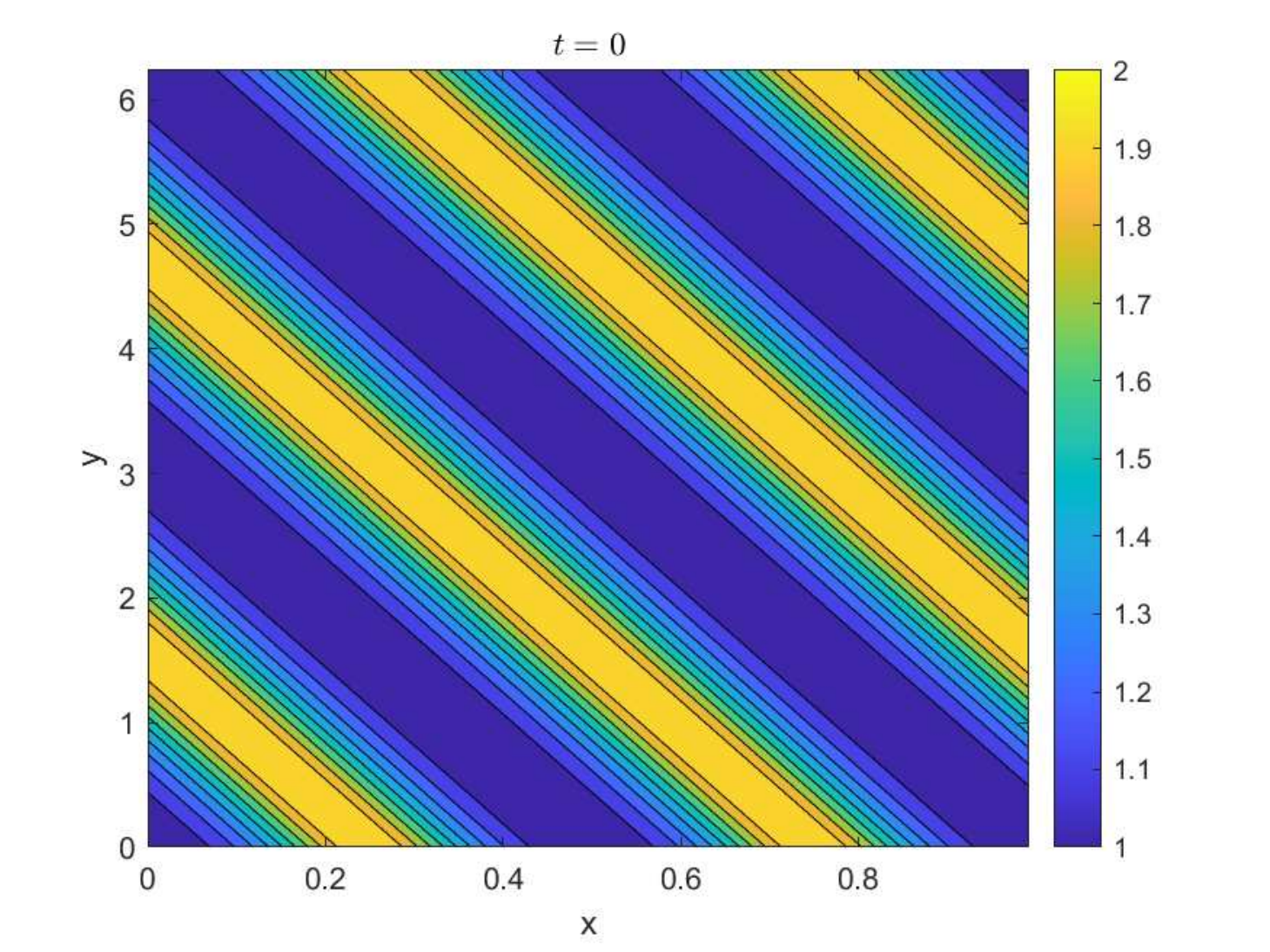}
	\end{minipage}}
	\subfigure[$T=0,\alpha=1.7$]{
		\begin{minipage}[t]{0.22\linewidth}
			\centering
			\includegraphics[width=1.2in]{cf_2D_t-eps-converted-to}
	\end{minipage}}
	\subfigure[$T=0,\alpha=1.4$]{
		\begin{minipage}[t]{0.22\linewidth}
			\centering
			\includegraphics[width=1.2in]{cf_2D_t-eps-converted-to}
	\end{minipage}}
	\subfigure[$T=0,\alpha=1.1$]{
		\begin{minipage}[t]{0.22\linewidth}
			\centering
			\includegraphics[width=1.2in]{cf_2D_t-eps-converted-to}
	\end{minipage}}
	\subfigure[$T=2,\alpha=2$]{
		\begin{minipage}[t]{0.22\linewidth}
			\centering
			\includegraphics[width=1.2in]{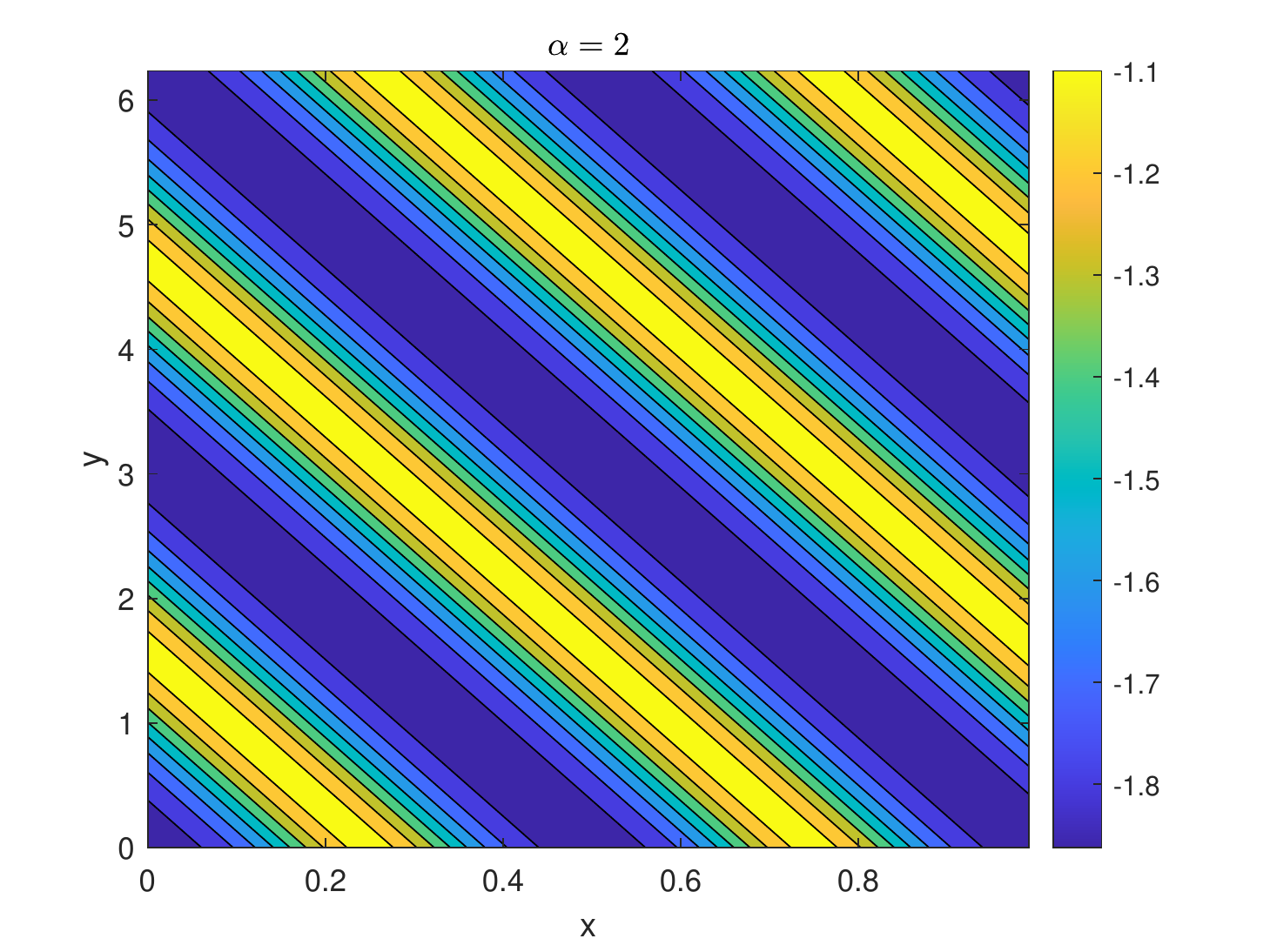}
	\end{minipage}}
	\subfigure[$T=2,\alpha=1.7$]{
		\begin{minipage}[t]{0.22\linewidth}
			\centering
			\includegraphics[width=1.2in]{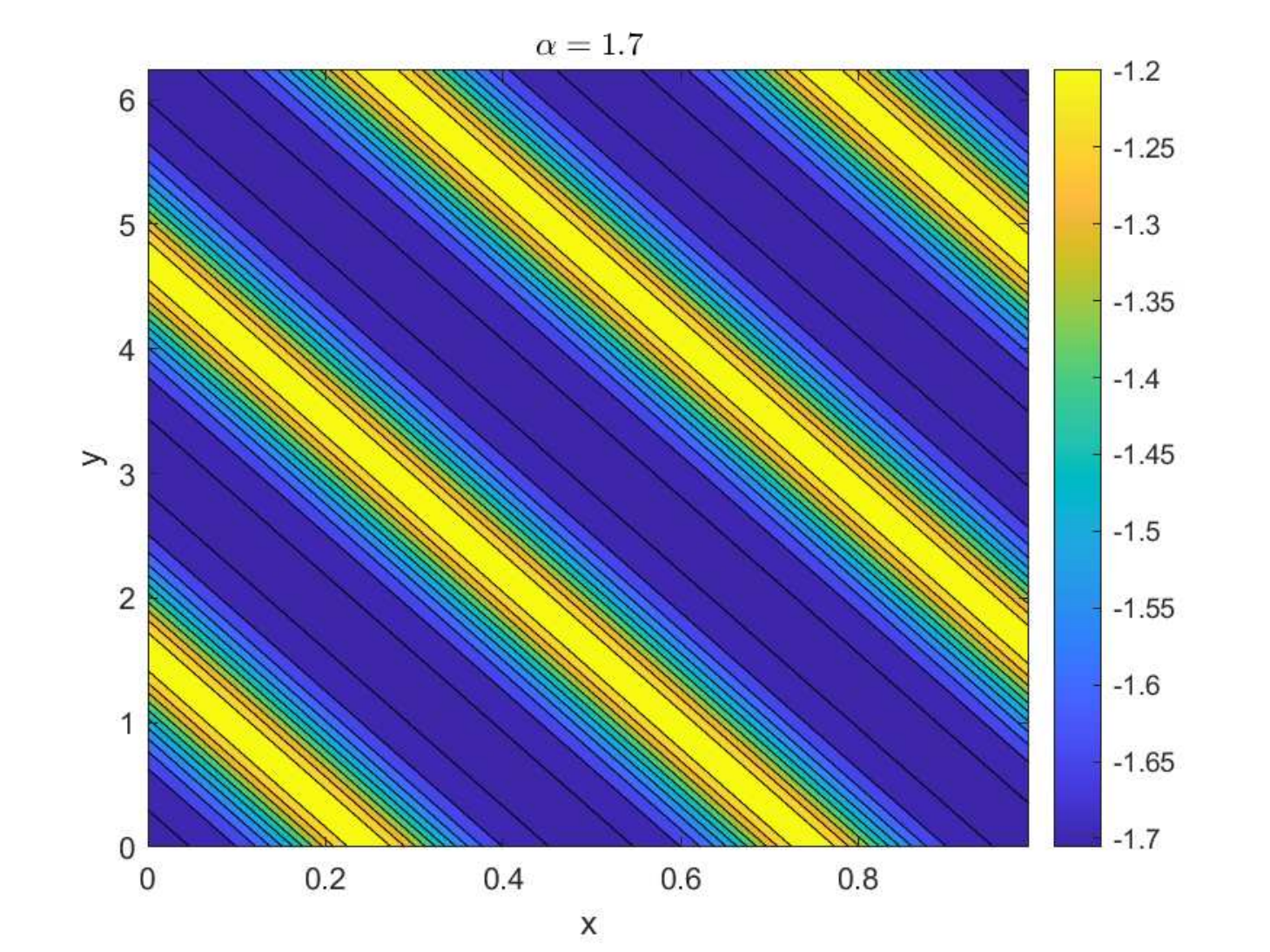}
	\end{minipage}}
	\subfigure[$T=2,\alpha=1.4$]{
		\begin{minipage}[t]{0.22\linewidth}
			\centering
			\includegraphics[width=1.2in]{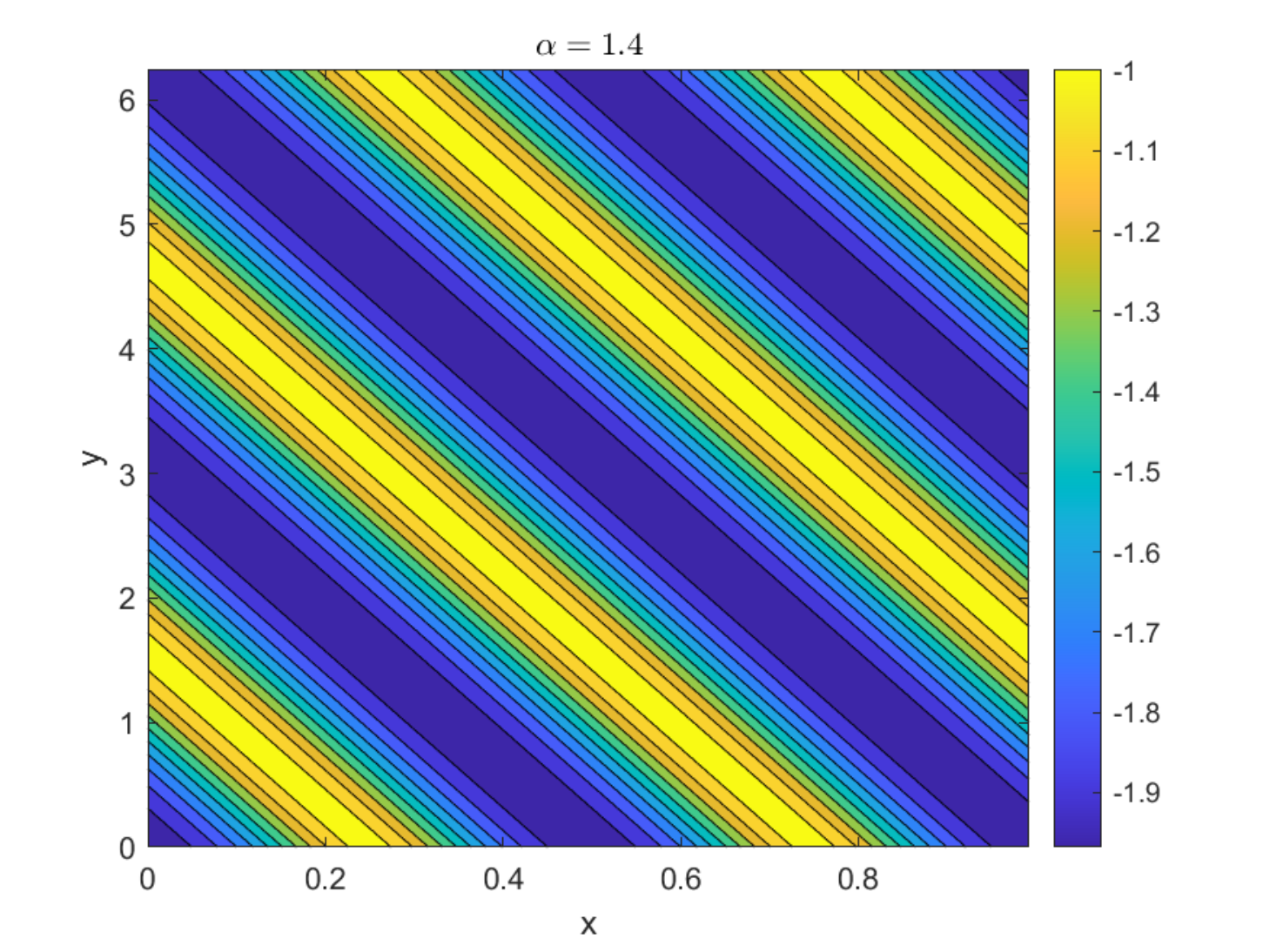}
	\end{minipage}}
	\subfigure[$T=2,\alpha=1.1$]{
		\begin{minipage}[t]{0.22\linewidth}
			\centering
			\includegraphics[width=1.2in]{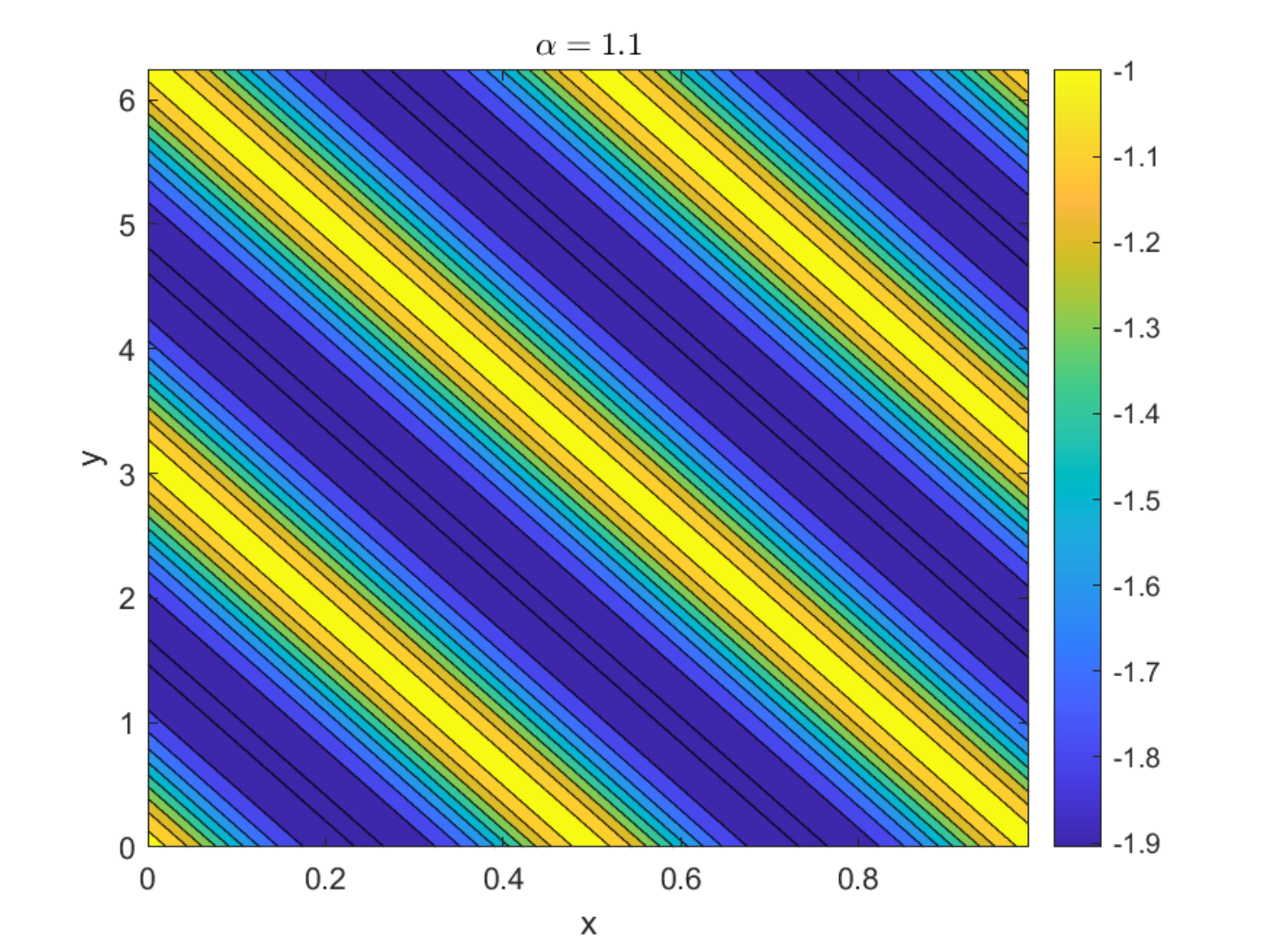}
	\end{minipage}}
	\subfigure [$T=8,\alpha=2$]{
		\begin{minipage}[t]{0.22\linewidth}
			\centering
			\includegraphics[width=1.2in]{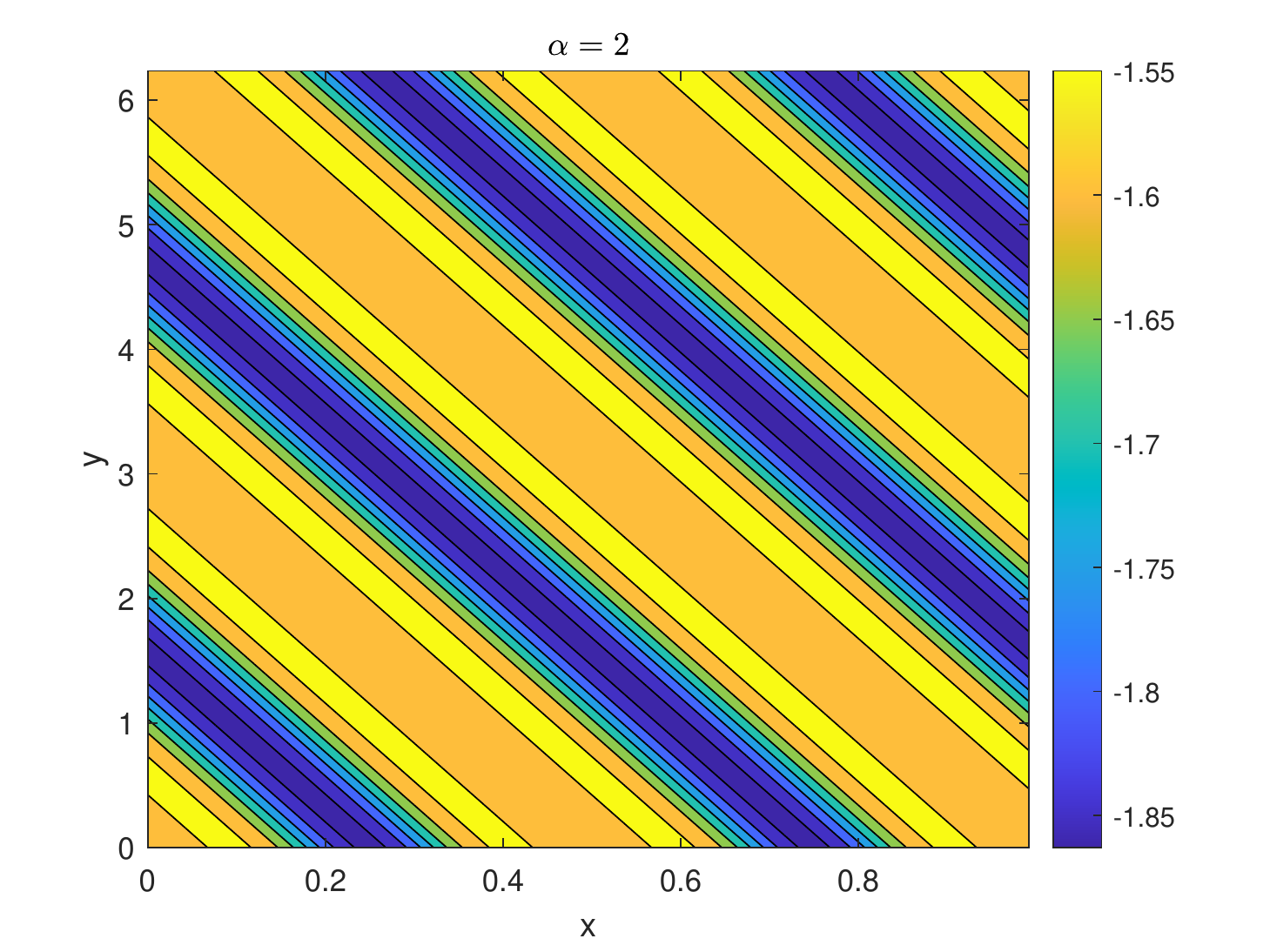}
	\end{minipage}}
	\subfigure[$T=8,\alpha=1.7$]{
		\begin{minipage}[t]{0.22\linewidth}
			\centering
			\includegraphics[width=1.2in]{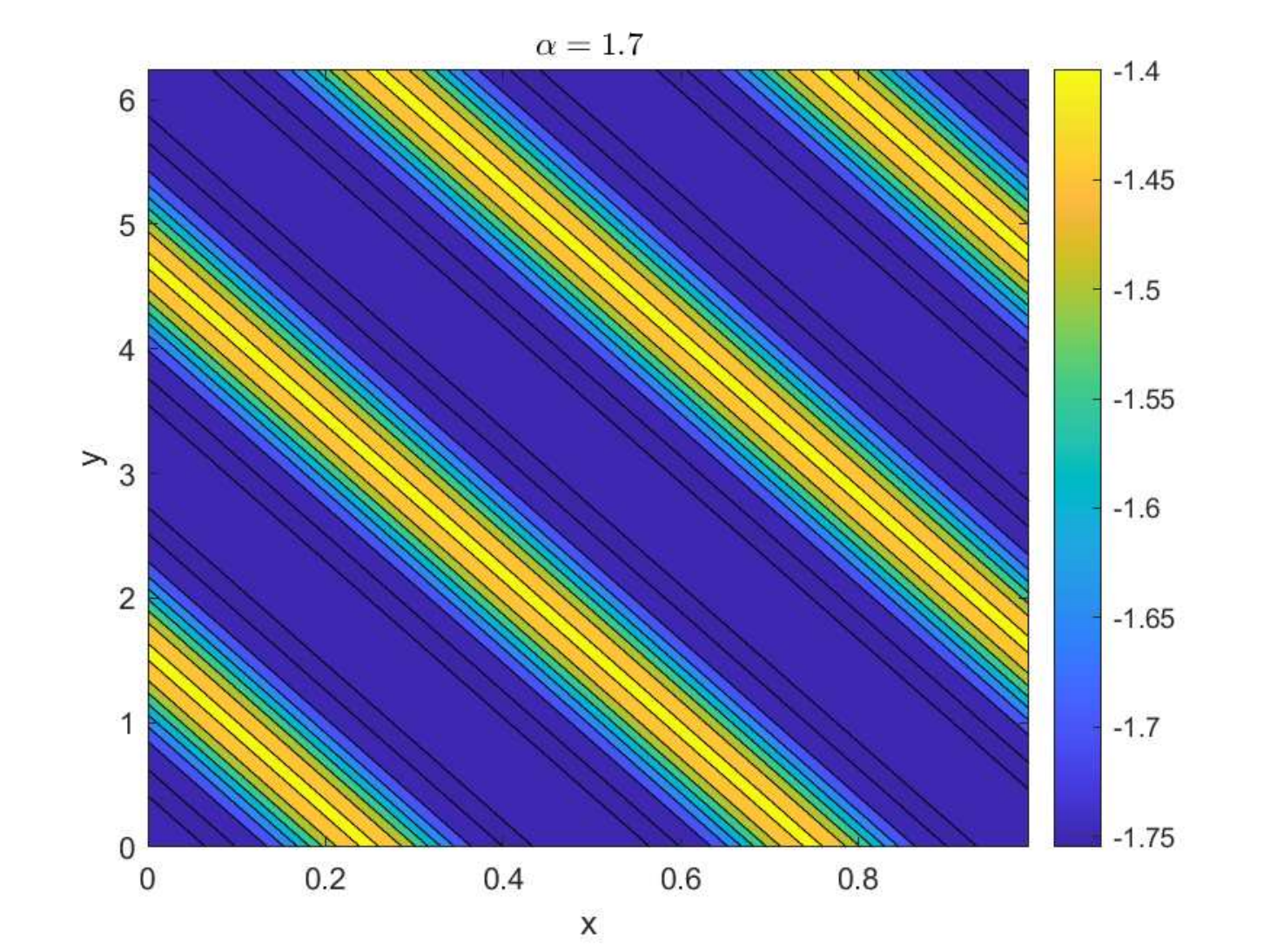}
	\end{minipage}}
	\subfigure[$T=8,\alpha=1.4$]{
		\begin{minipage}[t]{0.22\linewidth}
			\centering
			\includegraphics[width=1.2in]{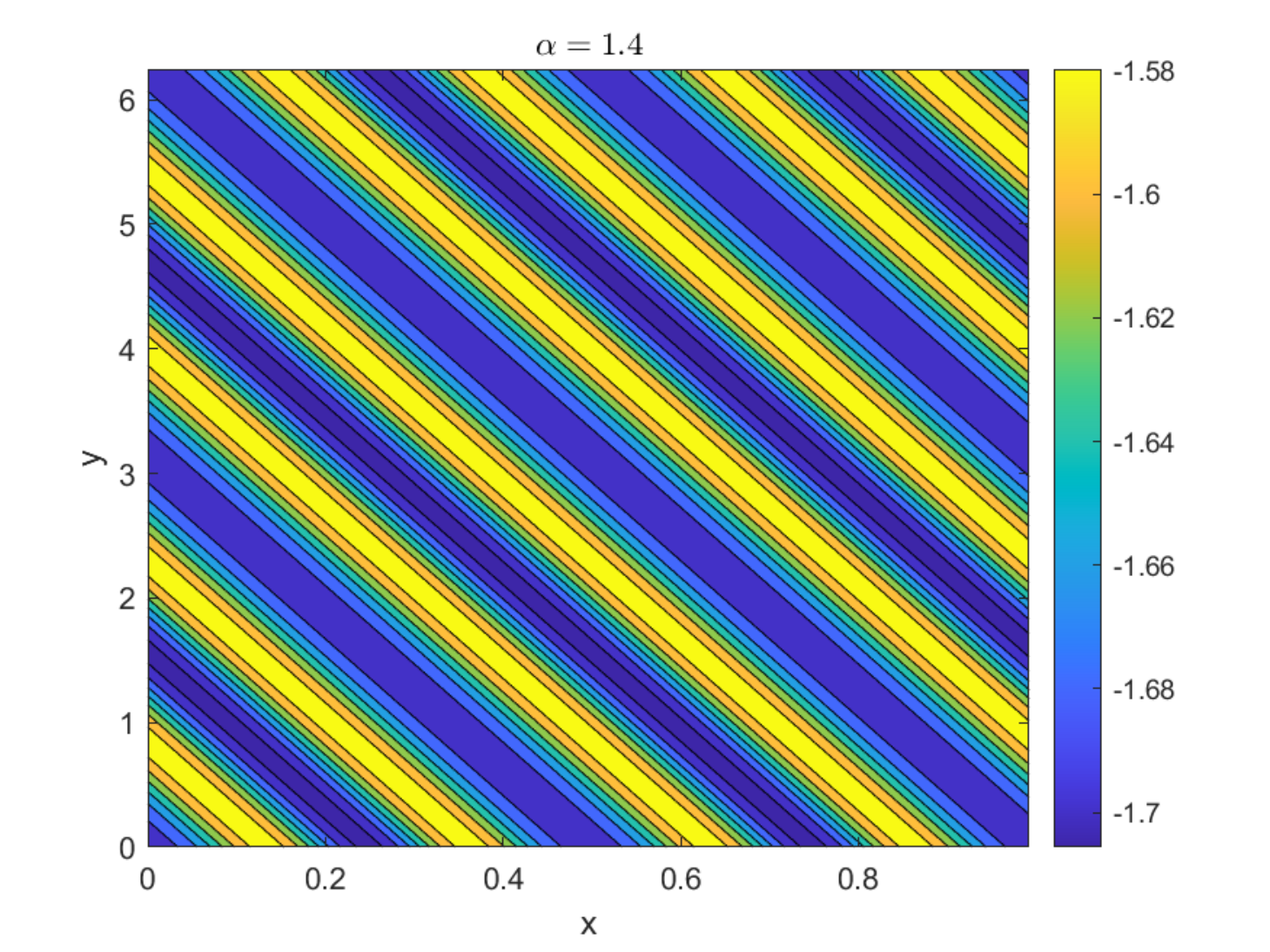}
	\end{minipage}}
	\subfigure[$T=8,\alpha=1.1$]{
		\begin{minipage}[t]{0.22\linewidth}
			\centering
			\includegraphics[width=1.2in]{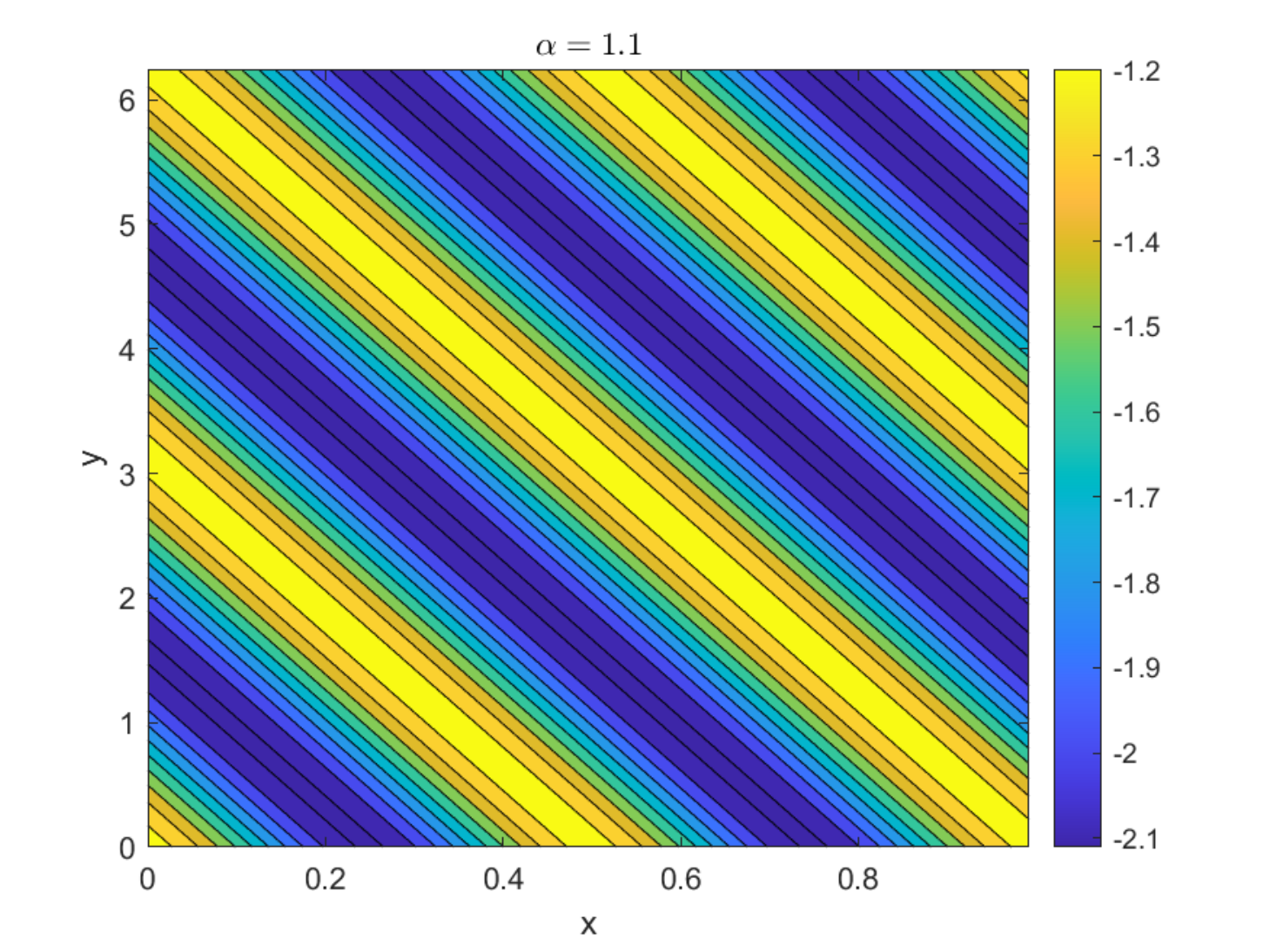}
	\end{minipage}}
	\subfigure[$T=32,\alpha=2$]{
		\begin{minipage}[t]{0.22\linewidth}
			\centering
			\includegraphics[width=1.2in]{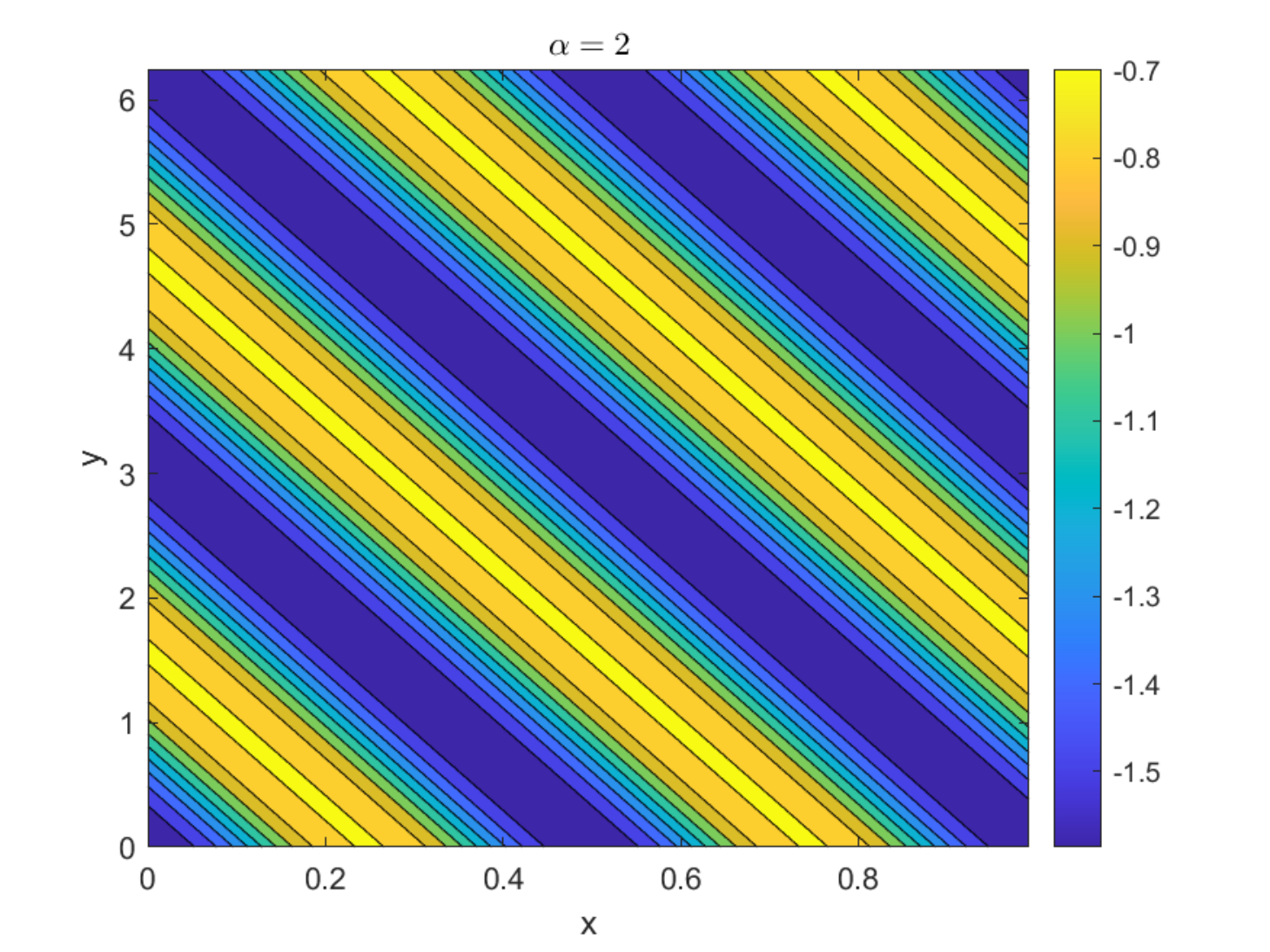}
	\end{minipage}}
	\subfigure[$T=32,\alpha=1.7$]{
		\begin{minipage}[t]{0.22\linewidth}
			\centering
			\includegraphics[width=1.2in]{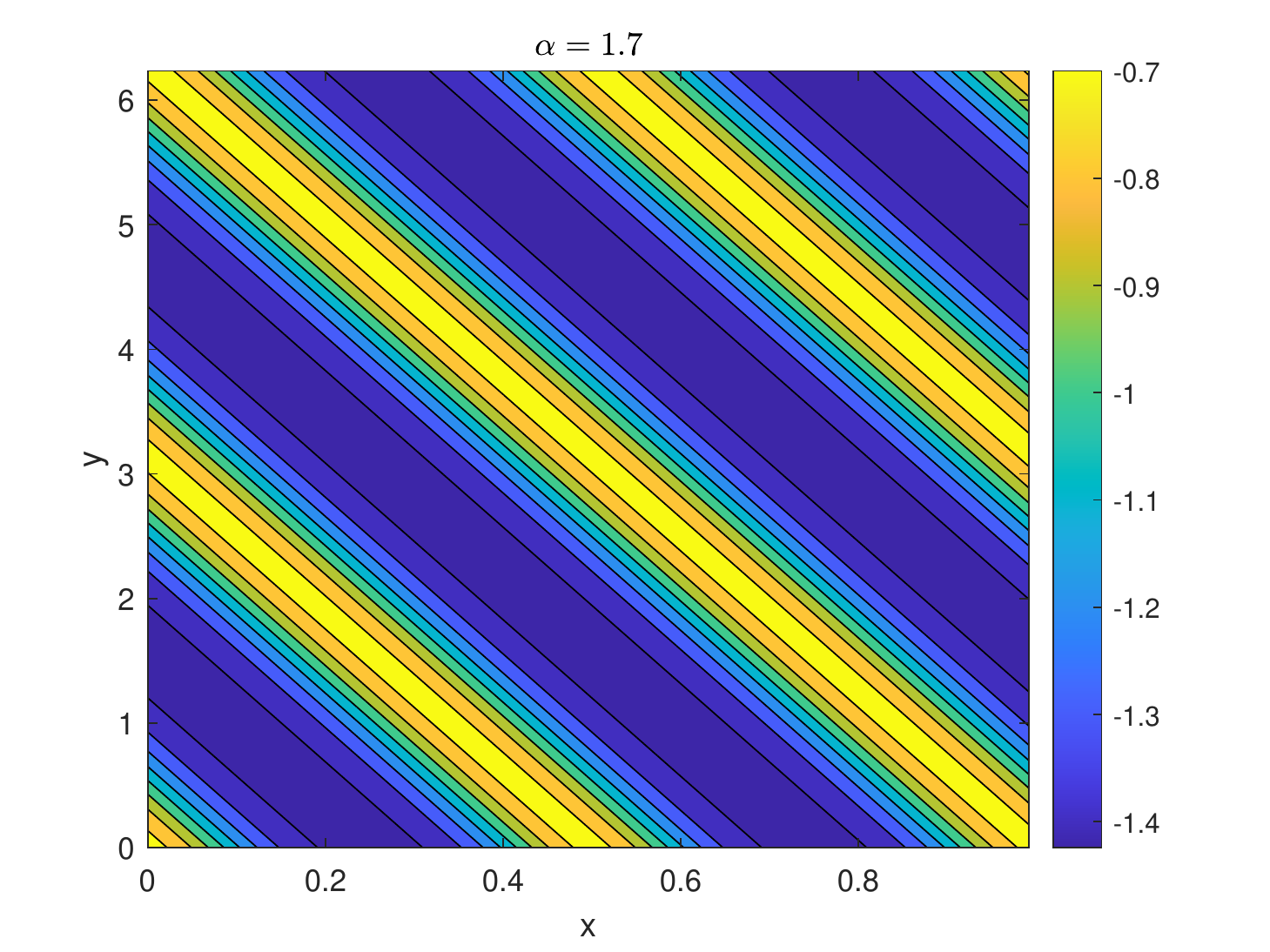}
	\end{minipage}}
	\subfigure[$T=32,\alpha=1.4$]{
		\begin{minipage}[t]{0.22\linewidth}
			\centering
			\includegraphics[width=1.2in]{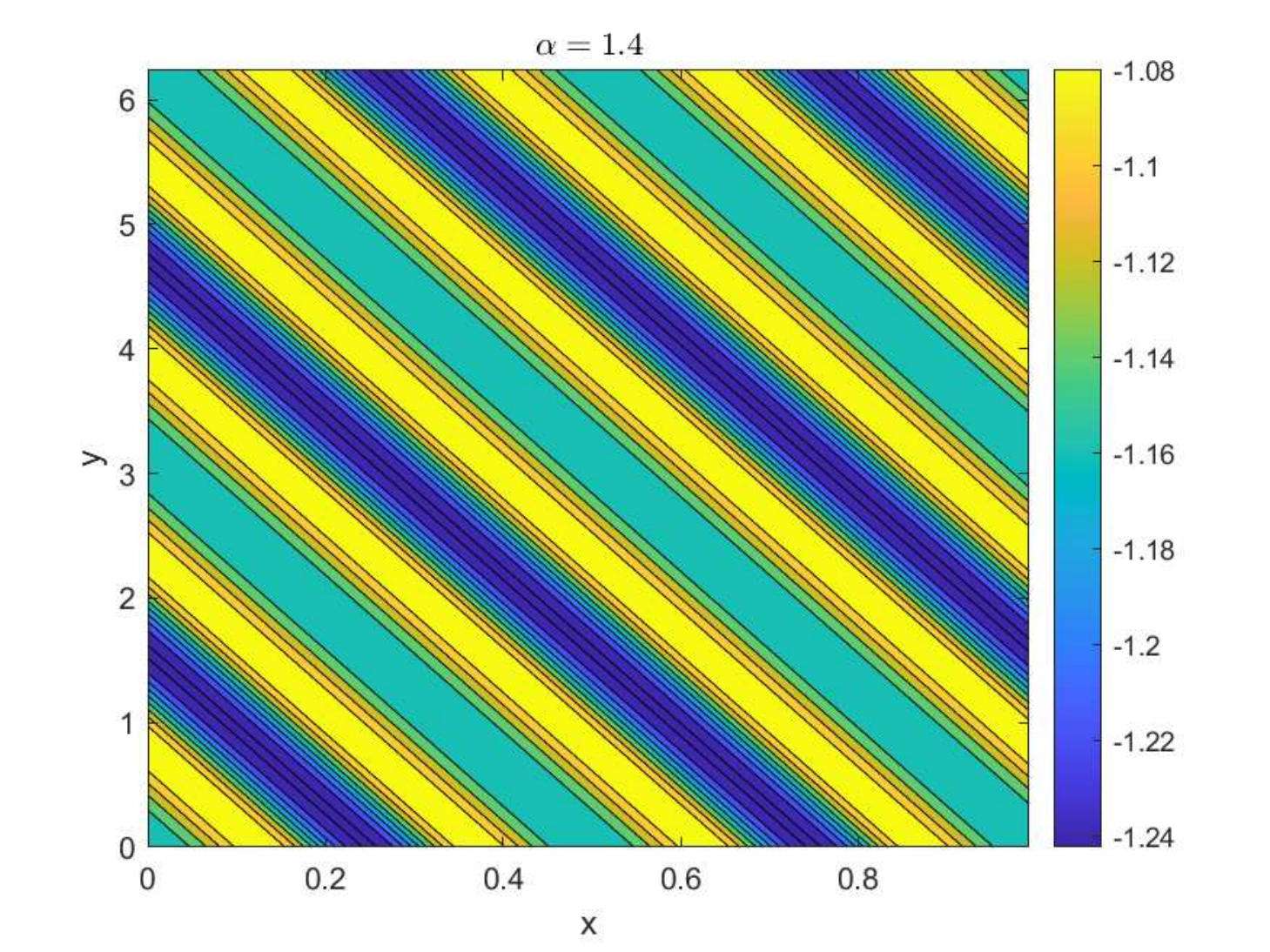}
	\end{minipage}}
	\subfigure[$T=32,\alpha=1.1$]{
		\begin{minipage}[t]{0.22\linewidth}
			\centering
			\includegraphics[width=1.2in]{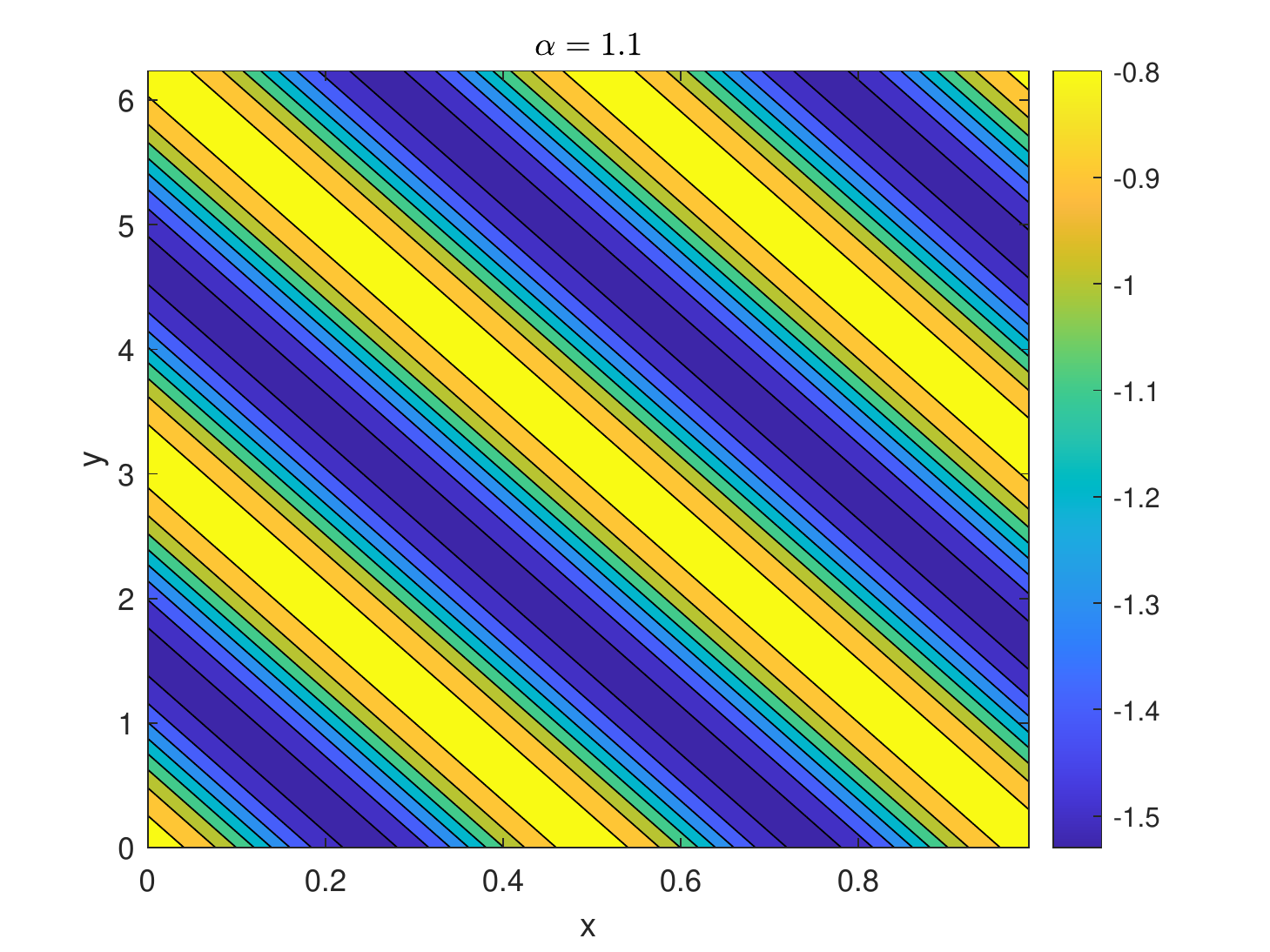}
	\end{minipage}}
\subfigure[$T=128,\alpha=2$]{
		\begin{minipage}[t]{0.22\linewidth}
			\centering
			\includegraphics[width=1.2in]{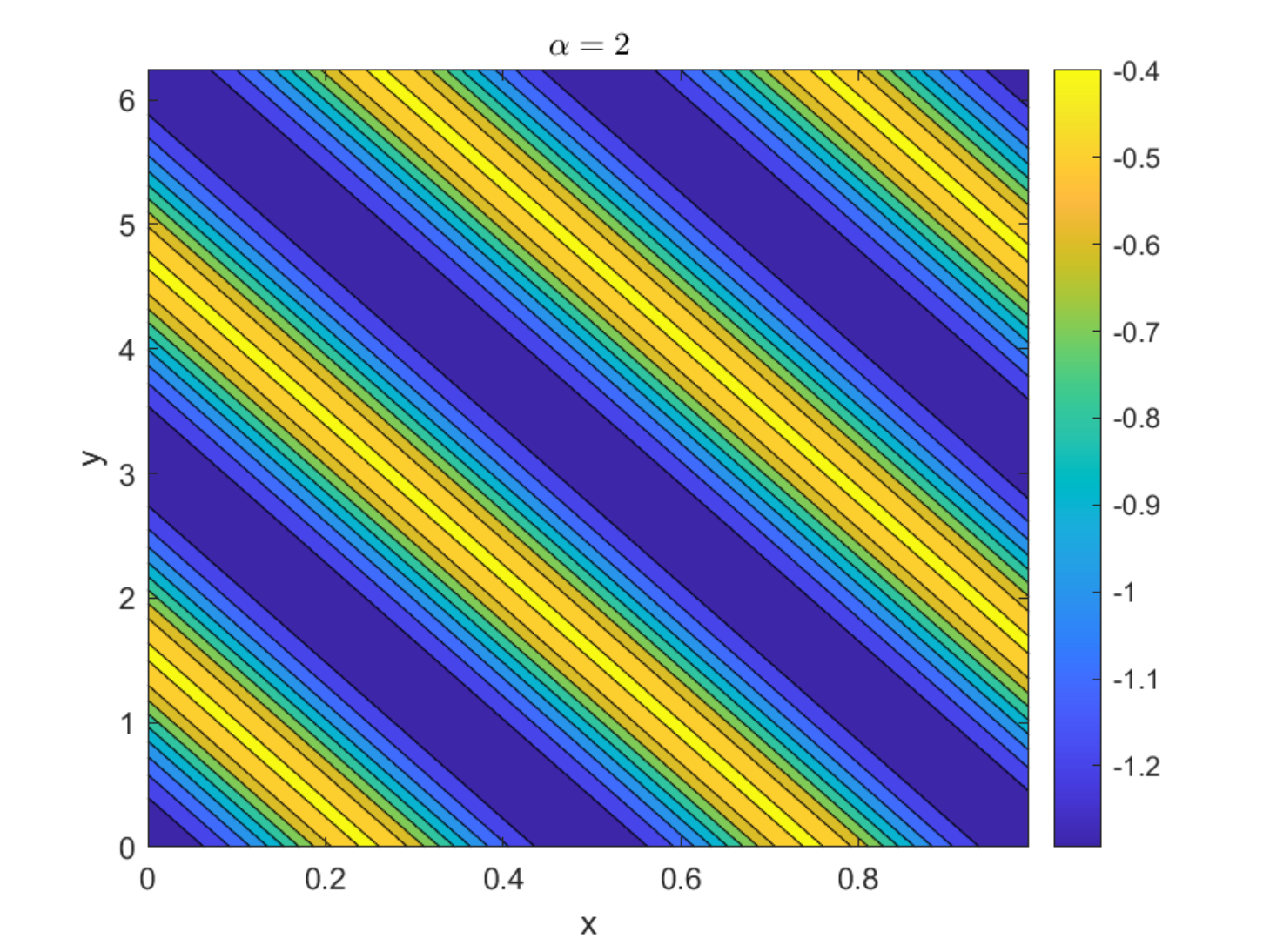}
	\end{minipage}}
	\subfigure[$T=128,\alpha=1.7$]{
		\begin{minipage}[t]{0.22\linewidth}
			\centering
			\includegraphics[width=1.2in]{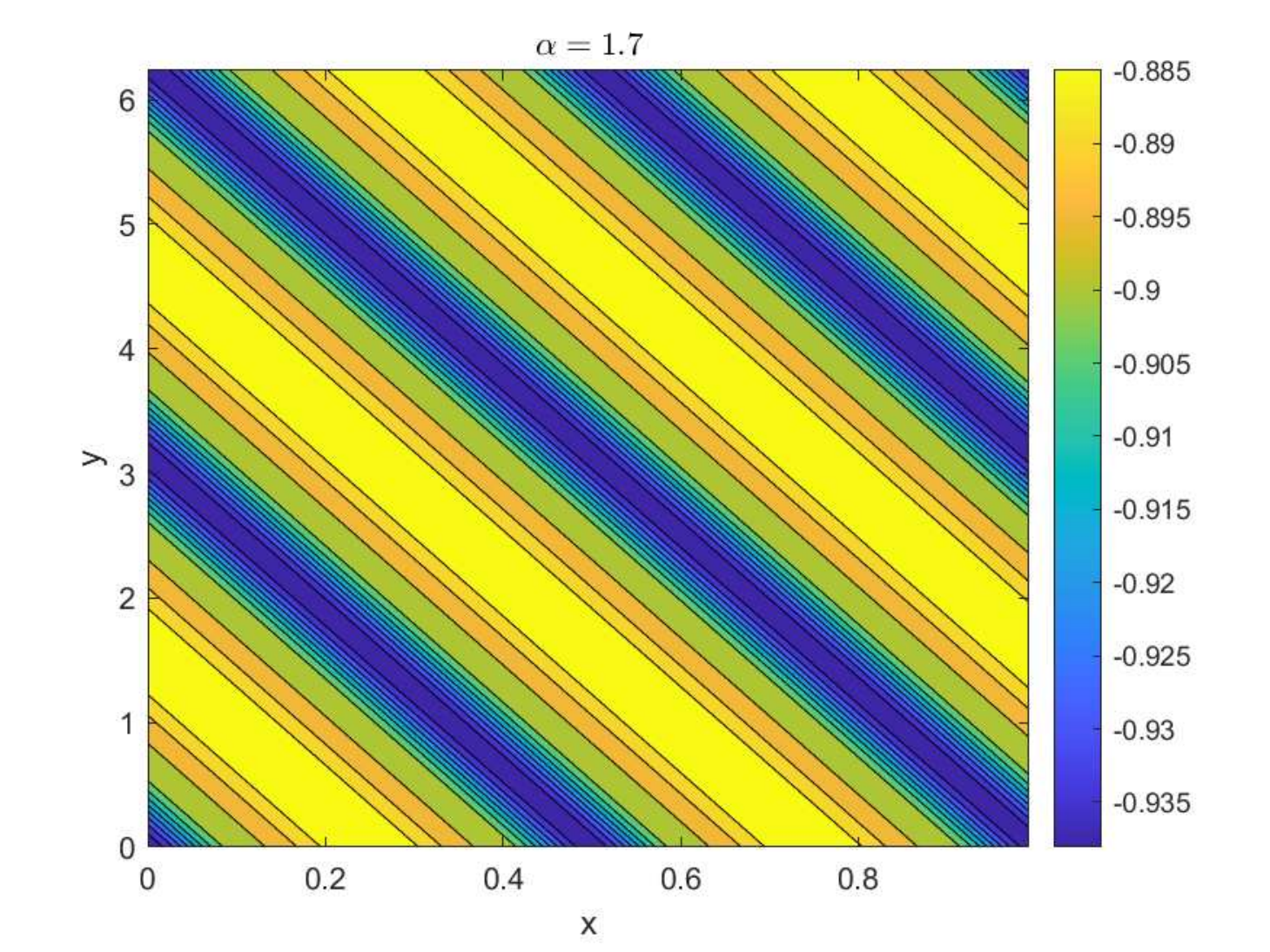}
	\end{minipage}}
	\subfigure[$T=128,\alpha=1.4$]{
		\begin{minipage}[t]{0.22\linewidth}
			\centering
			\includegraphics[width=1.2in]{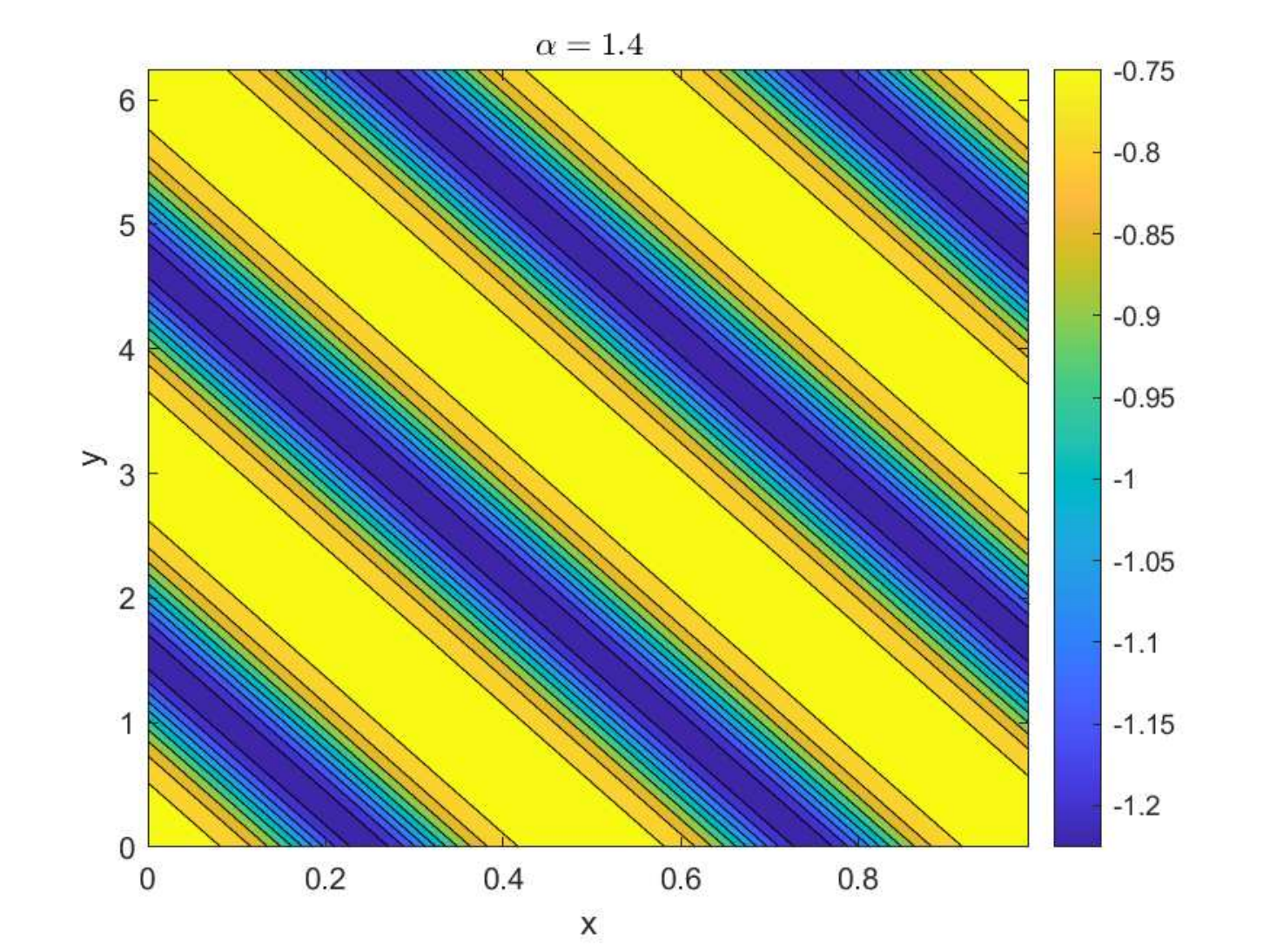}
	\end{minipage}}
	\subfigure[$T=128,\alpha=1.1$]{
		\begin{minipage}[t]{0.22\linewidth}
			\centering
			\includegraphics[width=1.2in]{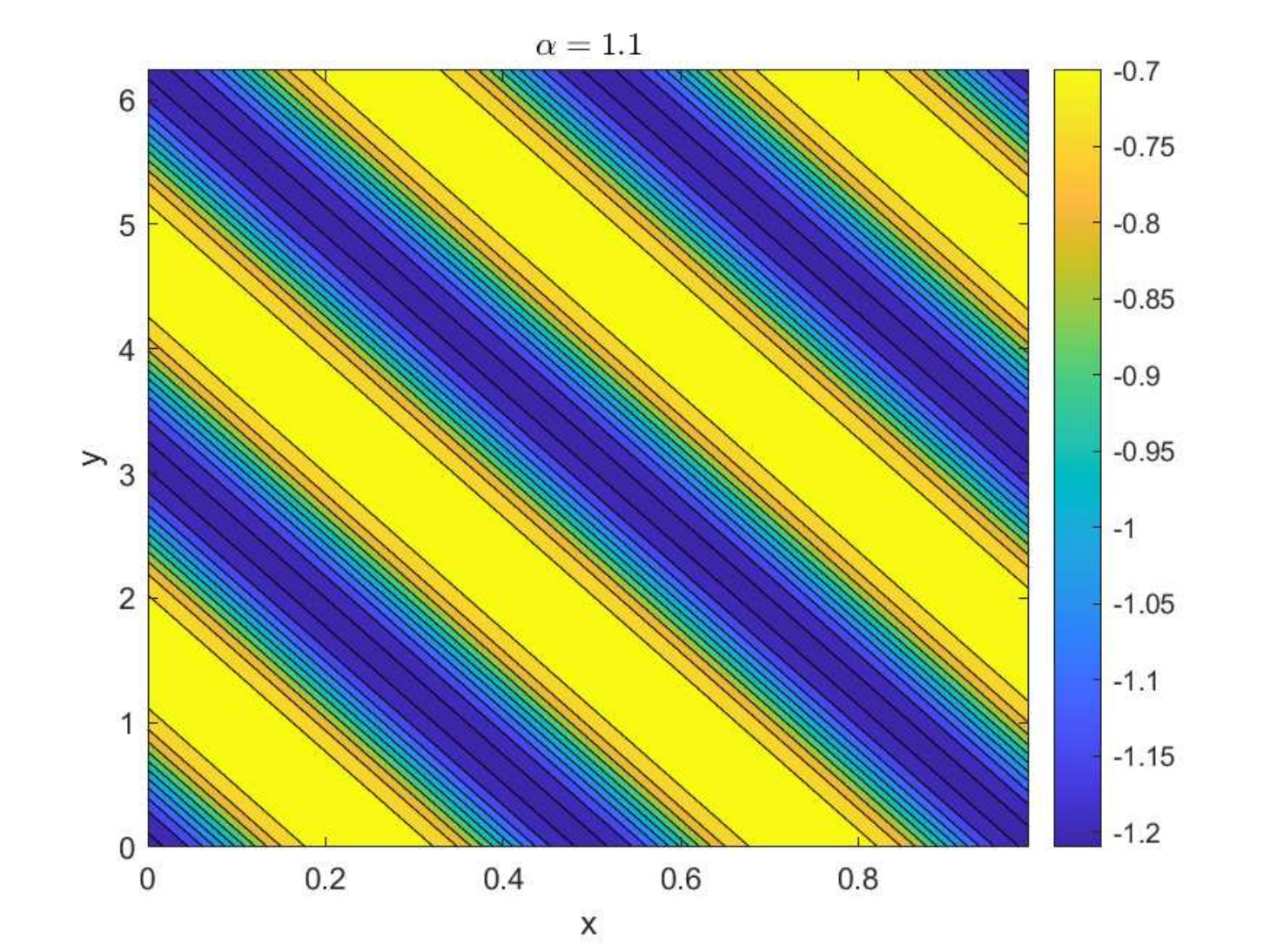}
	\end{minipage}}
	\caption{The contour figure for the NSFKGE \eqref{eq4.1.1} with $p=1$ in 2D}
	\label{fig5.42}
\end{figure}

Fig.\ref{fig5.41}-\ref{fig5.42} presents the figures and contour figures of the numerical solutions for the NSFKGE \eqref{eq4.1.1} for different $\alpha$ at $t=1/\varepsilon^2$. They indicate that $\alpha$ will affect the shape of the wave, change the periodicity of waves and generate new waves. It's obvious that as $\alpha$ decreases the relaxation time reaching the equilibrium increases. This is because the fractional diffusion process's long-range interactions and heavy-tailed influence \cite{Xu2005,Zhai2019}.

\begin{figure}[htb]
\centering
\includegraphics[scale=.34]{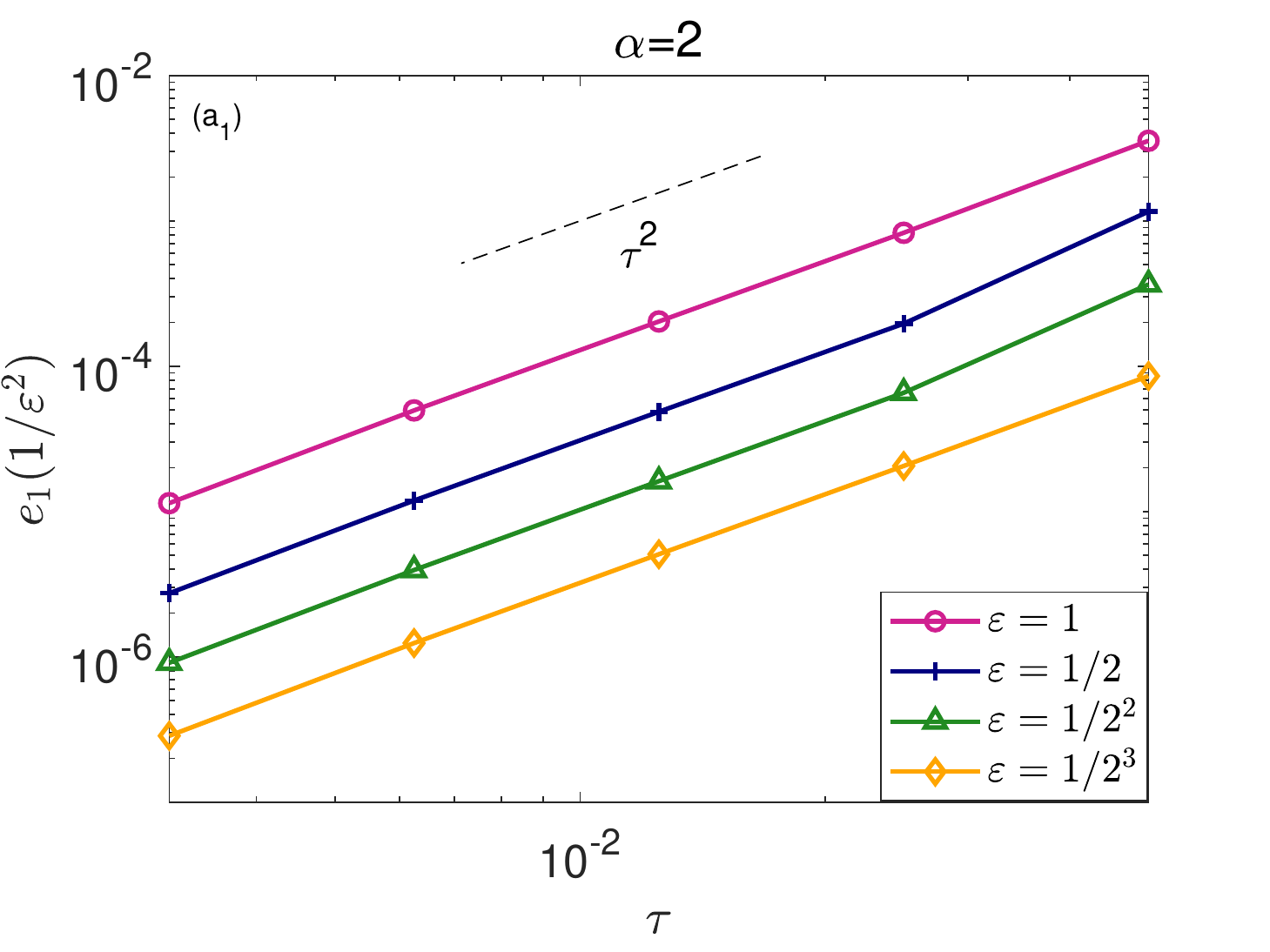}
\includegraphics[scale=.34]{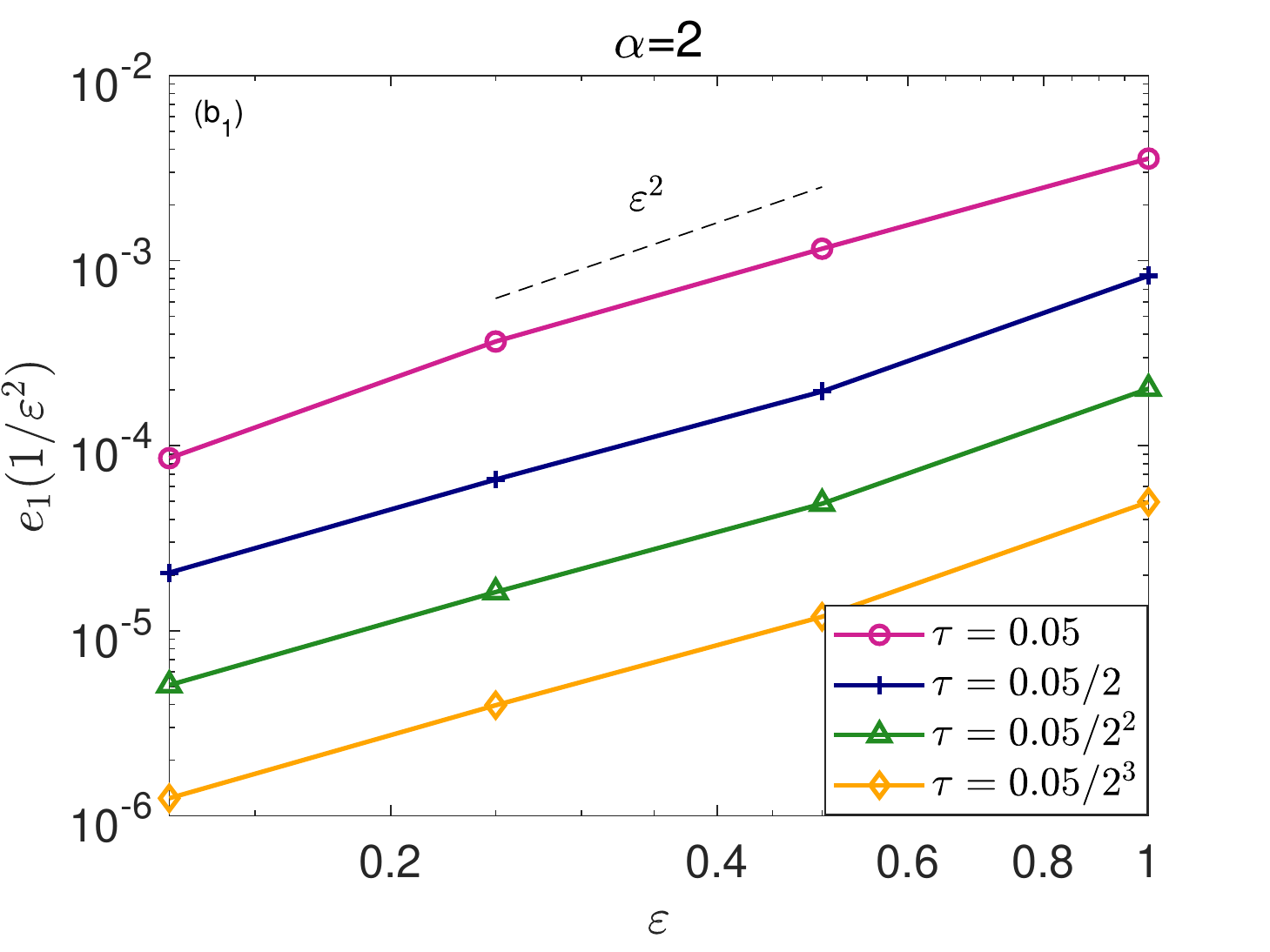}
\includegraphics[scale=.34]{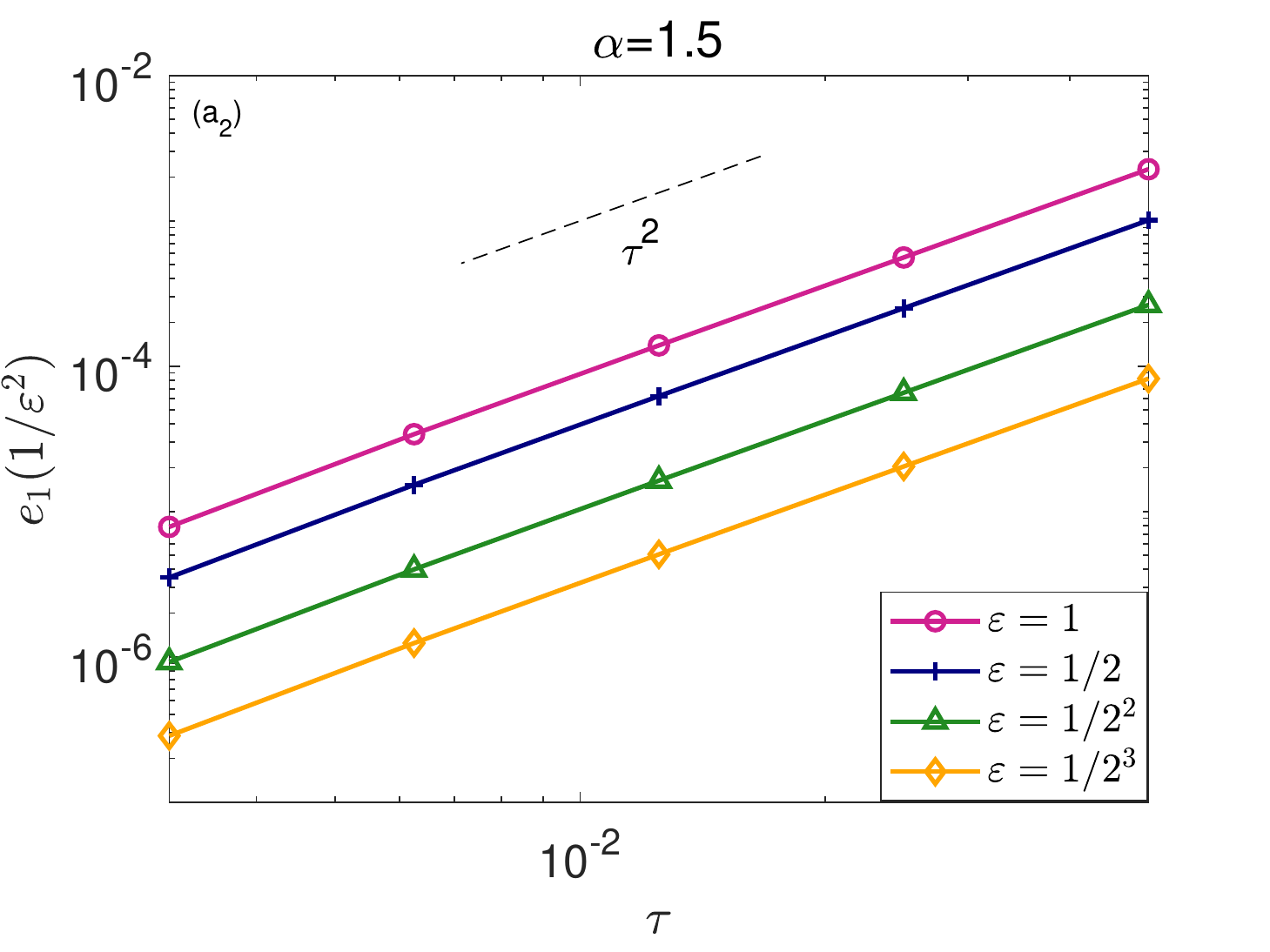}
\includegraphics[scale=.34]{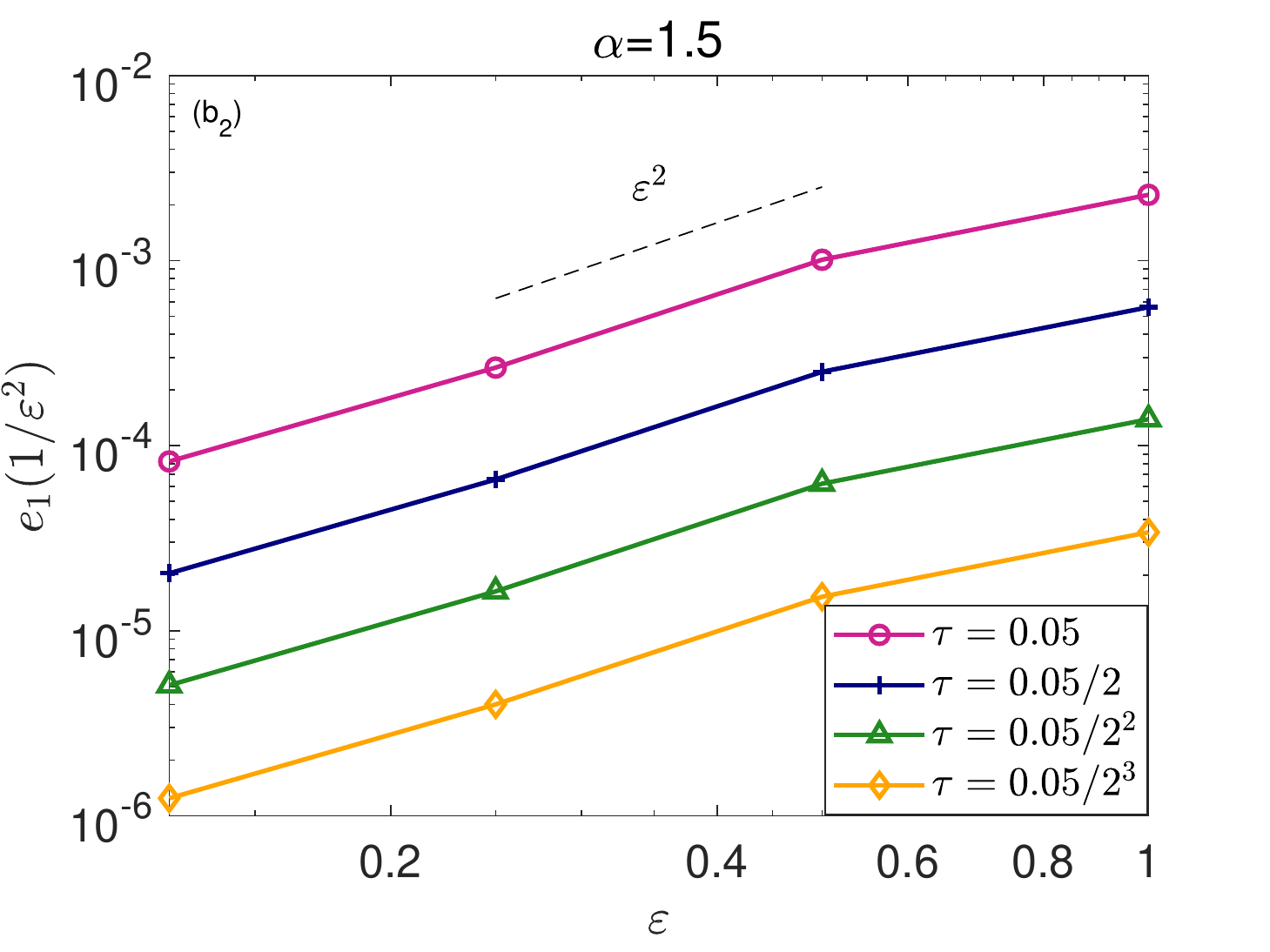}
\includegraphics[scale=.34]{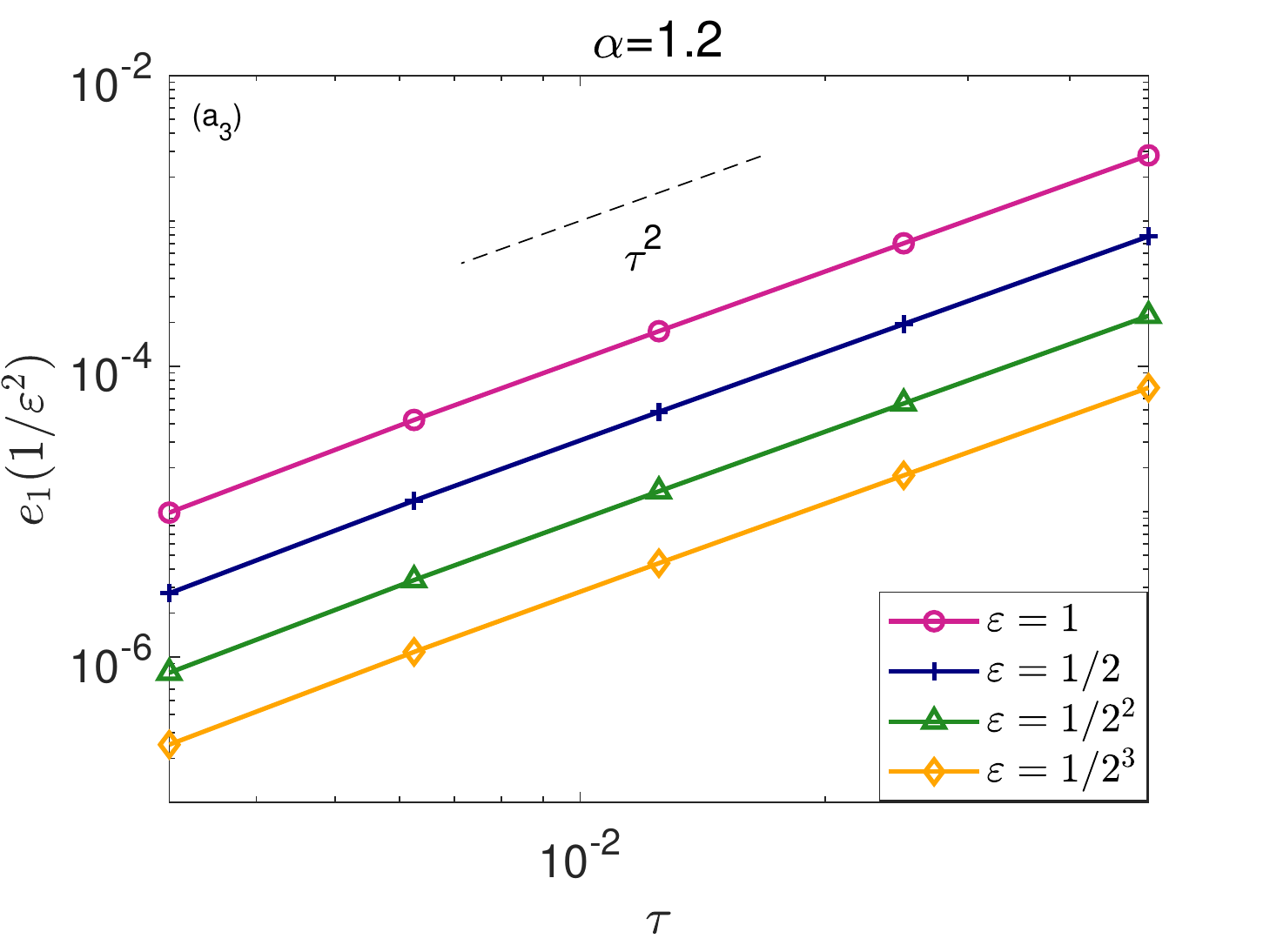}
\includegraphics[scale=.34]{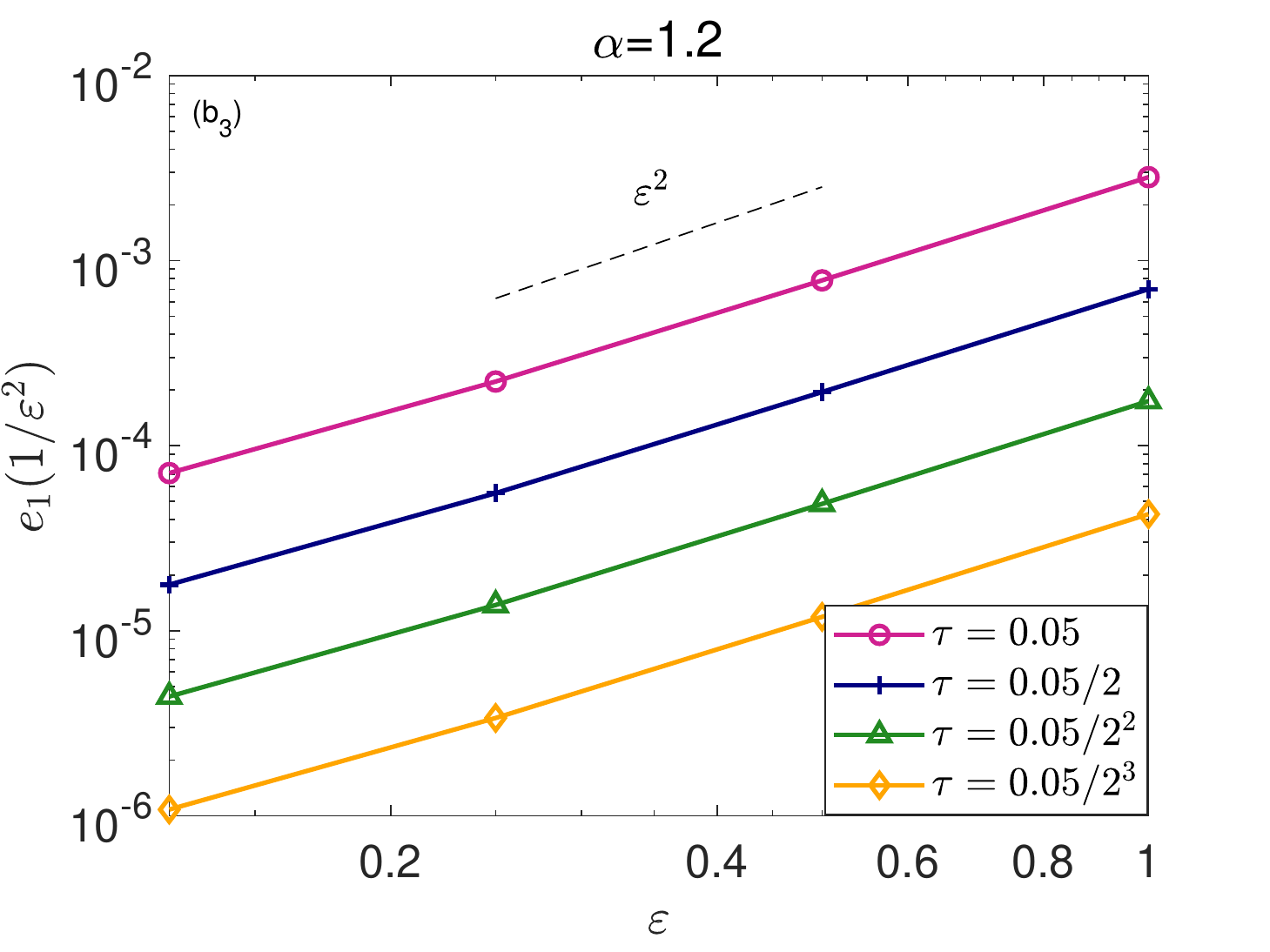}
\caption{Long-time temporal errors for the NSFKGE \eqref{eq4.1.1} with $p=1$ in 2D when $\alpha$ is taken as 2, 1.5 and 1.2 at $t=1/{\varepsilon^{2}}$ respectively.}
\label{fig5.5}
\end{figure}

Fig.\ref{fig5.5} describes the temporal errors of the NSFKGE at $t=1/\varepsilon^2$ when $\alpha$ is taken as 2, 1.5 and 1.2 respectively, which further verifies the improved uniform error bound is $O(\varepsilon^2\tau^2)$ up to the time $O\left(1 / \varepsilon^2\right)$.

\subsection{The oscillatory complex NSFKGE}
In this subsection, the numerical result of the 1D oscillatory complex NSFKGE \eqref{eq4.2.2} with $p=1$ is presented. And the complex-valued initial data is
$$
\psi_0(x)=x^2(x-1)^2+3, \quad \psi_1(x)=x(x-1)(2 x-1)+3 \textrm{i} \cos (2 \pi x), \quad x \in \Omega=(0,1).
$$

\begin{table}[htbp]
\caption{Temporal errors of the EWI method for the oscillatory complex NSFKGE \eqref{eq4.2.2} in 1D when $\alpha=2$.}
\centering
\label{tab1}
\begin{tabular}{cccccc}
\hline$e_1(r=1)$ & $\lambda_0=0.05$ & $\lambda_0 / 4$ & $\lambda_0 / 4^2$ & $\lambda_0 / 4^3$ & $\lambda_0 / 4^4$ \\
\hline$\varepsilon_0=1$ & $\mathbf{1.11 \mathrm{E}-2}$ & $6.90 \mathrm{E}-4$ & $4.31 \mathrm{E}-5$ & $2.65 \mathrm{E}-6$ & $1.24 \mathrm{E}-7$ \\
order & $-$ & $2.00$ & $2.00$ & $2.01$ & $2.21$ \\
\hline$\varepsilon_0 / 2$ & $5.90 \mathrm{E}-2$ & $\mathbf{3 . 22 \mathrm { E } - 3}$ & $2.00 \mathrm{E}-4$ & $1.25 \mathrm{E}-5$ & $7.67 \mathrm{E}-7$ \\
order & $-$ & $\mathbf{2.10}$ & $2.01$ & $2.00$ & $2.01$ \\
\hline$\varepsilon_0 / 2^2$ & $7.41 \mathrm{E}-1$ & $1.66 \mathrm{E}-2$ & $\mathbf{9.86 \mathrm { E } - 4}$ & $6.14 \mathrm{E}-5$ & $3.84 \mathrm{E}-6$ \\
order & $-$ & $2.74$ & $\mathbf{2.04}$ & $2.00$ & $2.00$ \\
\hline$\varepsilon_0 / 2^3$ & $1.43$ & $3.10 \mathrm{E}-1$ & $1.60 \mathrm{E}-2$ & $\mathbf{9.90 \mathrm { E } - 4}$ & $6.19 \mathrm{E}-5$ \\
order & $-$ & $1.11$ & $2.14$ & $\mathbf{2.01}$ & $2.00$ \\
\hline$\varepsilon_0 / 2^4$ & $2.03$ & $3.74$ & $7.76 \mathrm{E}-2$ & $4.25 \mathrm{E}-3$ & $\mathbf{2 . 6 4 \mathrm { E } - 4}$ \\
order & $-$ & $-0.44$ & $2.80$ & $2.10$ & $\mathbf{2.01}$ \\
\hline
\end{tabular}
\end{table}

\begin{table}[htbp]
\caption{Temporal errors of the EWI method for the oscillatory complex NSFKGE \eqref{eq4.2.2} in 1D when $\alpha=1.5$.}
\centering
\label{tab2}
\begin{tabular}{cccccc}
\hline$e_1(r=1)$ & $\lambda_0=0.05$ & $\lambda_0 / 4$ & $\lambda_0 / 4^2$ & $\lambda_0 / 4^3$ & $\lambda_0 / 4^4$ \\
\hline$\varepsilon_0=1$ & $\mathbf{1.22 \mathrm{E}-2}$ & $7.60 \mathrm{E}-4$ & $4.74 \mathrm{E}-5$ & $2.92 \mathrm{E}-6$ & $1.37 \mathrm{E}-7$ \\
order & $-$ & $2.00$ & $2.00$ & $2.01$ & $2.21$ \\
\hline$\varepsilon_0 / 2$ & $5.91 \mathrm{E}-2$ & $\mathbf{3 . 23 \mathrm { E } - 3}$ & $2.00 \mathrm{E}-4$ & $1.25 \mathrm{E}-5$ & $7.69 \mathrm{E}-7$ \\
order & $-$ & $2.10$ & $2.01$ & $2.00$ & $2.01$ \\
\hline$\varepsilon_0 / 2^2$ & $7.41 \mathrm{E}-1$ & $1.66 \mathrm{E}-2$ & $\mathbf{9.88 \mathrm { E } - 4}$ & $6.15 \mathrm{E}-5$ & $3.84 \mathrm{E}-6$ \\
order & $-$ & $2.74$ & $\mathbf{2.04}$ & $2.00$ & $2.00$ \\
\hline$\varepsilon_0 / 2^3$ & $1.43$ & $3.09 \mathrm{E}-1$ & $1.60 \mathrm{E}-2$ & $\mathbf{9.91 \mathrm { E } - 4}$ & $6.19 \mathrm{E}-5$ \\
order & $-$ & $1.11$ & $2.14$ & $\mathbf{2.01}$ & $2.00$ \\
\hline$\varepsilon_0 / 2^4$ & $2.03$ & $3.74$ & $7.75 \mathrm{E}-2$ & $4.25 \mathrm{E}-3$ & $\mathbf{2 . 6 4 \mathrm { E } - 4}$ \\
order & $-$ & $-0.44$ & $2.80$ & $2.10$ & $\mathbf{2.01}$ \\
\hline
\end{tabular}
\end{table}

\begin{table}[htbp]
\caption{Temporal errors of the EWI method for the oscillatory complex NSFKGE \eqref{eq4.2.2} in 1D when $\alpha=1.2$.}
\centering
\label{tab3}
\begin{tabular}{cccccc}
\hline$e_1(r=1)$ & $\lambda_0=0.05$ & $\lambda_0 / 4$ & $\lambda_0 / 4^2$ & $\lambda_0 / 4^3$ & $\lambda_0 / 4^4$ \\
\hline$\varepsilon_0=1$ & $\mathbf{1.03 \mathrm{E}-2}$ & $6.46 \mathrm{E}-4$ & $4.03 \mathrm{E}-5$ & $2.48 \mathrm{E}-6$ & $1.16 \mathrm{E}-7$ \\
order & $-$ & $2.00$ & $2.00$ & $2.01$ & $2.21$ \\
\hline$\varepsilon_0 / 2$ & $5.95 \mathrm{E}-2$ & $\mathbf{3 . 2 5 \mathrm { E } - 3}$ & $2.01 \mathrm{E}-4$ & $1.26 \mathrm{E}-5$ & $7.74 \mathrm{E}-7$ \\
order & $-$ & $\mathbf{2.10}$ & $2.01$ & $2.00$ & $2.01$ \\
\hline$\varepsilon_0 / 2^2$ & $7.42 \mathrm{E}-1$ & $1.66 \mathrm{E}-2$ & $\mathbf{9.86 \mathrm { E } - 4}$ & $6.14 \mathrm{E}-5$ & $3.83 \mathrm{E}-6$ \\
order & $-$ & $2.74$ & $\mathbf{2.04}$ & $2.00$ & $2.00$ \\
\hline$\varepsilon_0 / 2^3$ & $1.43$ & $3.10 \mathrm{E}-1$ & $1.60 \mathrm{E}-2$ & $\mathbf{9.91 \mathrm { E } - 4}$ & $6.19 \mathrm{E}-5$ \\
order & $-$ & $1.11$ & $2.14$ & $\mathbf{2.01}$ & $2.00$ \\
\hline$\varepsilon_0 / 2^4$ & $2.03$ & $3.74$ & $7.76\mathrm{E}-2$ & $4.25 \mathrm{E}-3$ & $\mathbf{2 . 6 4 \mathrm { E } - 4}$ \\
order & $-$ & $-0.44$ & $2.80$ & $2.10$ & $\mathbf{2.01}$ \\
\hline
\end{tabular}
\end{table}

From the upper triangle above the diagonal with bold letters in Table \ref{tab1}-\ref{tab3}, we find that the temporal errors are $O\left(\lambda^2\right)$ convergence accuracy when $\lambda \lesssim \varepsilon^2$ and in $H^{\alpha/2}$-norm behave like $O\left(\lambda^2 / \varepsilon^2\right)$, which confirm the improved error bound \eqref{eq4.2.3}. In addition, the fractional index $\alpha$ will not affect the result.

\section{Conclusions}
In this paper, the second order Deuflhard-type integrator was used to deal with the integral generated in the time semi-discretization, and Fourier pseudospectral method was applied to discrete the space. With the aid of the RCO technique, we strictly proved that the $H^{\alpha/2}$-norm error bound of NSFKGE was $O\left(h^m+\varepsilon^2 \tau^2\right)$ up to the time $O(1/\varepsilon^2)$. In RCO, the high frequency modes $(>1/\tau_{0})$, where $\tau_{0}$ is the frequency cut-off parameter, were controlled by the regularity of the exact solution, and the phase cancellation and energy method were used to evaluate the low frequency modes $(\leq 1/\tau_{0})$. Numerical results in 1D and 2D verified our error estimates and demonstrated they are sharp, the numerical example in 1D also showed the NSFKGE has the property of energy conservation, and the results in 2D indicated the the fractional diffusion process's long-range interactions and heavy-tailed influence.

\begin{acknowledgements}
The authors would like to specially thank Professor Weizhu Bao and Dr. Yue Feng for their valuable suggestions and comments. This work has been supported by the National Natural Science Foundation of China (Grants Nos. 12120101001, 12001326, 12171283), Natural Science Foundation of Shandong Province (Grant Nos. ZR2021ZD03, ZR2020QA032, ZR2019ZD42).
\end{acknowledgements}

\section*{Data Availability}
The datasets analysed during the current study are available from the corresponding author on reasonable request.

%
\section*{Declarations}
\textbf{Competing Interests}\
The authors declare that they have no known competing financial interests or personal relationships that could have appeared to influence the work reported in this paper.

\section*{Appendix}
\appendix
\section{Proof of Eq. \eqref{eq3.1.1}}
\begin{lemma}\label{lemA.1}
The exact solution of \eqref{eq2.1.4} with initial data $\varphi_0$ is denoted as $\varphi(t)=S_{e, t}\left(\varphi_0\right)$. Assume $\varphi(t) \in H^{m+\alpha/2}(m > 1)$, then for $0<\varepsilon \leq 1$ and $0<\tau \leq 1 / \varepsilon^2$, the local error of the scheme \eqref{eq2.2.4} is bounded by
$$
\left\| \mathcal{E}^{n} \right\|_{m+\alpha/2} = \left\|S_\tau\left(\varphi\left(t_n\right)\right)-S_{e, \tau}\left(\varphi\left(t_n\right)\right)\right\|_{m+\alpha/2} \leq K_0 \varepsilon^2 \tau^3,
$$
where

\begin{equation*}
\begin{aligned}
S_\tau\left(\varphi\left(t_n\right)\right)=e^{i \tau\langle\nabla\rangle_{\alpha}} \varphi\left(t_n\right)+\varepsilon^2\cdot \frac{\tau}{2}\left[G\left(\varphi\left(t_n+\tau\right)\right)+ e^{i\tau\langle\nabla\rangle_{\alpha}}G\left(\varphi\left(t_n\right)\right)\right],
\end{aligned}
\end{equation*}

\begin{equation*}
\begin{aligned}
S_{e, \tau}\left(\varphi\left(t_n\right)\right)=e^{i \tau\langle\nabla\rangle_{\alpha}} \varphi\left(t_n\right)+\varepsilon^2 \int_0^\tau e^{i(\tau-s)\langle\nabla\rangle_{\alpha}} G\left(\varphi\left(t_n+s\right)\right) \textrm{d} s,
\end{aligned}
\end{equation*}

and $K_0$ depends on $\|\varphi\|_{C^{2}\left(\left[0, T_{\varepsilon}\right] ; H^{m+\alpha/2}\right)}$.
\end{lemma}

\begin{proof}

\begin{equation*}
\begin{aligned}
&\left\| \mathcal{E}^{n} \right\|_{m+\alpha/2}\\
=& \left\|S_\tau\left(\varphi\left(t_n\right)\right)-S_{e, \tau}\left(\varphi\left(t_n\right)\right)\right\|_{m+\alpha/2}\\
=&\left\|\varepsilon^2\cdot \frac{\tau}{2}\left[G\left(\varphi\left(t_n+\tau\right)\right)+ \textrm{e}^{i\tau\langle\nabla\rangle_{\alpha}}G\left(\varphi\left(t_n\right)\right)\right]
-\varepsilon^2 \int_0^\tau \textrm{e}^{i(\tau-s)\langle\nabla\rangle_{\alpha}} G\left(\varphi\left(t_n+s\right)\right)\textrm{d}s\right\|_{m+\alpha/2}\\
=&\left\|\frac{\varepsilon^2\tau^{3}}{2}\int_{0}^{1}\sigma(1-\sigma)G''(\sigma\tau)\textrm{d}\sigma\right\|_{m+\alpha/2}
\leq K_0 \varepsilon^2 \tau^3.
\end{aligned}
\end{equation*}\qed
\end{proof}

Suppose
\begin{equation}\label{eqA.1}
\begin{aligned}
\|e^{[k]}\|_{\infty}\leq C_{0}, \quad 0\leq k\leq n\leq j, \quad n,j\in \mathbb{Z}^{+}.
\end{aligned}
\end{equation}
We also suppose that
\begin{equation}\label{eqA.2}
\begin{aligned}
\|e^{[k+1]}\|_{\infty}\leq C_{1}, \quad \text{if} \quad \|e^{[k]}\|_{\infty}\leq C_{0}.
\end{aligned}
\end{equation}
\setcounter{thm}{0}
\begin{thm}\label{thmA.1}[\textbf{Uniform error bound}]
Suppose that $\varphi^{[n+1]}\left(n=0,1,2, \ldots, T_{\varepsilon}/{\tau}-1\right)$ are solutions of the scheme \eqref{eq2.2.4} and $C, C^{*}$ are suitable positive constants independent of $n, \tau$,  $\varepsilon$. Assume $\|e^{[k+1]}\|_{\infty} \leq C_1$ under the condition $\|e^{[k]}\|_{\infty}\leq C_{0}(0 \leq k \leq n)$. If $\tau$ is sufficiently small and $0<\tau <\min\left\{\frac{2}{L\varepsilon^2}, \left(\frac{C_{0}}{C^{*}CK_{0}T\text{e}^{CLT}}\right)^{\frac{1}{2}}\right\}$, then
\begin{equation}\label{eqA.3}
\begin{aligned}
\left\|e^{[n+1]}\right\|_{m+\alpha/2} \leq {CK_{0}T\tau^2}\text{e}^{CL(n+1)\varepsilon^2\tau}.
\end{aligned}
\end{equation}
\end{thm}

\begin{proof}
We adopt the mathematical induction method to prove \eqref{eqA.3}.
Obviously, \eqref{eqA.3} holds for $n=0$. Then we prove \eqref{eqA.3} holds for any $0\leq n\leq T_{\varepsilon}/{\tau}-1$.

Suppose $\|e^{[k]}\|_{\infty}\leq C_{0}(0 \leq k \leq j)$. By Lemma \ref{lemA.1} and the Proposition 3.1 (iii) in \cite{Feng2021c}, we have
\begin{equation}\label{eqA.5}
\begin{aligned}
&\left\|e^{[k+1]}\right\|_{m+\alpha/2}\\
=&\left\|S_{\tau}(\varphi^{[k]})-S_{\tau}(\varphi(t_{k}))+S_{\tau}(\varphi(t_{k}))-S_{e,\tau}(\varphi(t_{k}))\right\|_{m+\alpha/2}\\
\leq& \left\|e^{[k]}\right\|_{m+\alpha/2}+\frac{\varepsilon^2\tau}{2}L_{1}\left\|e^{[k]}\right\|_{m+\alpha/2}
+\frac{\varepsilon^2\tau}{2}L_{2}\left\|e^{[k+1]}\right\|_{m+\alpha/2}+K_{0}\varepsilon^2\tau^3.
\end{aligned}
\end{equation}
Summing $k$ from 0 to $n$, for sufficiently small $\tau$, when $0<\tau<\frac{2}{L\varepsilon^2}$, one can get

\begin{equation}\label{eqA.8}
\begin{aligned}
\left\|e^{[n+1]}\right\|_{m+\alpha/2}\leq  {CL\varepsilon^2\tau}\sum_{k=0}^{n}\left\|e^{[k]}\right\|_{m+\alpha/2}+{CK_{0}T\tau^2},
\end{aligned}
\end{equation}
where $L=\max\{L_{1},L_{2}\}$.

Using Gronwall's inequality yields
\begin{equation}\label{eqA.9}
\begin{aligned}
\left\|e^{[n+1]}\right\|_{m+\alpha/2} \leq {CK_{0}T\tau^2}\text{e}^{CL(n+1)\varepsilon^2\tau},\quad 0\leq n\leq j.
\end{aligned}
\end{equation}
Next, we prove \eqref{eqA.9} holds for $n=j+1$.
\begin{equation}\label{eqA.11}
\begin{aligned}
\left\|e^{[j+1]}\right\|_{\infty}\leq C^{*}\left\|e^{[j+1]}\right\|_{m+\alpha/2}
\leq {C^{*}CK_{0}T\tau^2}\text{e}^{CL(n+1)\varepsilon^2\tau}\leq C_{0},
\end{aligned}
\end{equation}
where $0<\tau\leq \left(\frac{C_{0}}{C^{*}CK_{0}T\text{e}^{CLT}}\right)^{\frac{1}{2}}$.
Thus, $\|e^{[k]}\|_{\infty} \leq C_0$ for $0\leq k\leq j+1$.
We also derive that
\begin{equation}\label{eqA.12}
\begin{aligned}
\left\|e^{[j+2]}\right\|_{m+\alpha/2}&\leq {CK_{0}T\tau^2}+{CL\varepsilon^2\tau}\sum_{k=0}^{j+1}\left\|e^{[k]}\right\|_{m+\alpha/2}\\
&\leq {CK_{0}T\tau^2}\left(1+{CL\varepsilon^2\tau}\sum_{k=0}^{j+1}\text{e}^{CL(k+1)\varepsilon^2\tau}\right)\\
&\leq {CK_{0}T\tau^2}\text{e}^{CL(j+2)\varepsilon^2\tau}.
\end{aligned}
\end{equation}
Thus, \eqref{eqA.9} holds for $n=j+1$. The proof is completed.\qed
\end{proof}
For the fully discrete numerical solution, we have the similar conclusion.





\end{document}